\documentclass[a4paper,11pt,reqno]{amsart}
\usepackage{amsmath,amssymb,amsthm} 
\usepackage{graphics} 
\usepackage{enumitem}

\usepackage[colorlinks,citecolor=blue]{hyperref}

\usepackage[capitalize]{cleveref}

\usepackage{epsfig}
\usepackage{color}
\usepackage[english]{babel}

\numberwithin{equation}{section}

\newtheorem{theorem}{Theorem}[section]
\newtheorem{lemma}[theorem]{Lemma}
\newtheorem{proposition}[theorem]{Proposition}

\newtheorem{corollary}[theorem]{Corollary}

\theoremstyle{definition}
\newtheorem{definition}[theorem]{Definition}
\newtheorem{remark}[theorem]{Remark}

\newcommand{\restr}{\mathop{\raisebox{-.127ex}{\reflectbox{\rotatebox[origin=br]{-90}{$\lnot$}}}}}

\newcommand{\R}{\mathbb{R}}
\newcommand{\N}{\mathbb{N}}
\newcommand{\C}{\mathbf{C}}

\newcommand{\eps}{\varepsilon}

\newcommand{\be}{\begin{equation}}
\newcommand{\ee}{\end{equation}}

\newcommand\lt{\left}
\newcommand\rt{\right}

\def\les{\lesssim}
\def\ges{\gtrsim}

\newcommand{\cF}{\mathcal{F}}
\newcommand{\cB}{\mathcal{B}}
\newcommand{\cK}{\mathcal{K}}

\def \C{\mathbf{C}}

\def\EE{\mathbb{E}}
\def\PP{\mathbb{P}}
\def\diam{\operatorname{diam}}

\newcommand{\cN}{\mathcal{N}}
\newcommand{\cM}{\mathcal{M}}
\newcommand{\cS}{\mathcal{S}}
\renewcommand{\c}{\mathsf{c}}

\newcommand{\bra}[1]{\left( #1 \right)}
\newcommand{\sqa}[1]{\left[ #1 \right]}
\newcommand{\cur}[1]{\left\{ #1 \right\}}

\newcommand{\abs}[1]{\left| #1 \right|}
\newcommand{\nor}[1]{\left\| #1 \right\|}
\def\fref{f}

\newcommand{\TSP}{\operatorname{TSP}}
\renewcommand{\C}{\mathcal{C}}

\newcommand{\pP}{\mathsf{P}}

\newcommand{\CPp}{\C_{\pP}^p}

\renewcommand{\TSP}{\mathsf{TSP}^p}
\newcommand{\dist}{\mathsf{d}}

\newcommand{\Mp}{\mathsf{M}^p}
\newcommand{\W}{\mathsf{W}^p}
\newcommand{\cT}{\mathcal{T}}
\newcommand{\cG}{\mathcal{G}}
\newcommand{\Lip}{\operatorname{Lip}}
\newcommand{\x}{{\bf x}}
\newcommand{\y}{{\bf y}}
\newcommand{\z}{{\bf z}}

\newcommand{\cR}{\mathcal{R}}
\newcommand{\cU}{\mathcal{U}}
\newcommand{\cV}{\mathcal{V}}

\newcommand{\cQ}{\mathcal{Q}}

\newcommand{\cspan}{\mathsf{c}_{\operatorname{A2}}}
\newcommand{\cbdeg}{\mathsf{c}_{\operatorname{A3}}}
\newcommand{\cmerge}{\mathsf{c}_{\operatorname{A4}}}
\newcommand{\creg}{\mathsf{c}_{\operatorname{A5}}}

\usepackage[hyperref,doi,url=false,doi=false, isbn=false, giveninits=true, backref,style=numeric,maxbibnames=99, backend=biber]{biblatex}

\bibliography{OT.bib}

\title[OT methods for combinatorial optimization over two random point sets] {Optimal transport methods for combinatorial optimization over two random point sets
}
\author[M. Goldman]{Michael Goldman}
\address{M.G.:   CMAP, CNRS, \'Ecole polytechnique, Institut Polytechnique de Paris, 91120 Palaiseau,
France}
\email{michael.goldman@cnrs.fr}
\author[D. Trevisan]{Dario Trevisan}
\address{D.T.: Dipartimento di Matematica, Università degli Studi di Pisa, 56125 Pisa, Italy  }
\email{dario.trevisan@unipi.it}
\date{}
\subjclass[2010]{60D05, 90C05, 39B62, 60F25, 35J05}
\keywords{Travelling Salesperson Problem, matching problem, optimal transport, geometric probability}
\thanks{D.T. was
partially supported by the INdAM-GNAMPA project 2022 ``Temi di Analisi Armonica Subellittica''.}
\newcounter{proof-step}

\begin{document}

\begin{abstract}
We investigate the minimum cost of  a wide class of combinatorial optimization problems over random bipartite geometric graphs in $\R^d$  where the edge cost between two points is given by a $p$-th power of their Euclidean distance. This includes e.g.\ the travelling salesperson problem and the bounded degree minimum spanning tree.  We establish in particular almost sure convergence, as $n$ grows, of a suitable renormalization of the random minimum cost, if the points are uniformly distributed and $d \ge 3$, $1\le p<d$. Previous results were limited to the range $p<d/2$. 

Our proofs are based on subadditivity methods and build upon new bounds for random instances of the Euclidean bipartite matching problem, obtained through its optimal transport relaxation and functional analytic techniques.

\end{abstract}

\maketitle

\setcounter{tocdepth}{1}
\tableofcontents
\section{Introduction}

Combinatorial optimization problems on graphs are widespread in operation research, with applications in planning and logistics. Their study is strongly related to algorithm theory and computational complexity 
theory. The most representative example of such discrete variational problems is the travelling salesperson problem (TSP) \cite{shmoys1985traveling}: given a set of cities and distances between each
pair of them, one asks for the shortest route that visits each city exactly once and returns to the origin city (i.e.\ a tour). Like many related combinatorial problems and  despite its straightforward formulation,
the TSP belongs to the class of  NP-hard problems.  
In practical terms, computing an exact solution becomes computationally intractable as known algorithms perform exponentially many steps in the number of cities.

 In real-world situations, there is quite often the need to solve many similar instances of a given combinatorial optimization problem. In that case, additional structure, including geometry and randomness, can be exploited. The Euclidean formulation of the TSP, i.e., when cities are points in $\R^d$ and distances are given by the Euclidean distance, is still NP-hard \cite{papadimitriou1977euclidean}, but Karp \cite{karp1977probabilistic} observed that solutions to random instances, i.e., when cities are sampled independently and uniformly, can be efficiently approximated via a partitioning scheme.  His proof relies upon the seminal
 work by Beardwood, Halton and Hammersely \cite{beardwood1959shortest}, where precise asymptotics for optimal costs of a random instance of the problem were first established: 
 given i.i.d.\ points $(X_i)_{i=1}^n$ distributed according to a probability density $\rho$ on $\R^d$, denoting the length $
\mathcal{C}_{\mathsf{TSP}}((X_i)_{i=1}^n)$ of the (random) solution to the TSP cycling through such points satisfies the $\PP$-a.s.\ limit
\begin{equation}\label{eq:bhh} \lim_{n \to \infty} n^{\frac{1}{d}-1} \mathcal{C}_{ \mathsf{TSP}}((X_i)_{i=1}^n) = \beta_{\operatorname{BHH}} \int_{\R^d} \rho^{1-\frac{1}{d}},\end{equation}
where $\beta_{\operatorname{BHH}} = \beta_{\operatorname{BHH}}(d) \in (0, \infty)$ is a constant depending on the dimension $d$ only. The scaling $n^{1-1/d}$ is intuitively explained by the fact that the $n$ cities are connected through paths of typical length $n^{-1/d}$ (as if they were on a regular grid). 

Building upon these ideas, several authors \cite{papadimitriou1978probabilistic, steele1981subadditive, steele1997probability, yukich2006probability} contributed towards establishing a general theory to
obtain limit results of BHH-type, i.e., as in \eqref{eq:bhh}, for a wide class of random Euclidean combinatorial optimization problems. 
The theory allows also for more general weights than the Euclidean length, including $p$-th powers of the Euclidean distance, a variant often motivated by modelling needs. If $0<p<d$, with a
minimal modification of the techniques one obtains BHH-type results as in  \eqref{eq:bhh}, with the scaling replaced by $n^{1-p/d}$,  
the constant $\beta_{\operatorname{BHH}}$ now depending on $p$, $d$ and the specific combinatorial optimization problem, and the integrand $\rho^{1-1/d}$ replaced by $\rho^{1-p/d}$. 
For $p \ge d$, the situation becomes subtler and \eqref{eq:bhh} is known for the TSP only if $p=d$, see \cite{yukich1995asymptotics} and \cite[Section 4.3]{yukich2006probability}. 

Despite the wide applicability of this theory, several classical problems such as those formulated over two random sets of points, are not covered  and  require different mathematical tools. The Euclidean assignment problem, also called bipartite matching, is certainly the most representative among these: given two sets of $n$ points $(x_i)_{i=1}^n$, $(y_j)_{j=1}^n \subseteq \R^d$,  one defines the matching cost functional as
\[ \Mp\bra{ (x_i)_{i=1}^n, (y_j)_{j=1}^n } = \min_{\sigma} \sum_{i=1}^n |x_i - y_{\sigma(i)}|^p,\]
where the minimum is taken among all the permutations $\sigma$ over $n$ elements.
This is often interpreted in terms of optimal planning for the execution of a set of jobs at positions $y_j$'s to be assigned to a set of workers at the positions $x_i$'s. Although the assignment problem belongs to the P complexity class, i.e., an optimal $\sigma$ can be found in a polynomial number of steps (with respect to $n$) the analysis of random instances shows some interesting behavior in low dimensions. Indeed, if $(X_i)_{i=1}^n$, $(Y_j)_{j=1}^n$ are i.i.d.\ and  uniformly distributed on the cube $(0,1)^d$, it is known  \cite{dudley1969speed, AKT84, talagrand1992matching, dobric1995asymptotics} that \footnote{
The notation $A\les B$ means that there exists a
constant $C>0$,  such that $A\le C B$, where $C$ depends on the dimension $d$, $p$ and possibly other quantities tacitly considered as fixed, e.g.\ a domain $\Omega \subseteq \R^d$ or a probability density $\rho$.  We use the notation $\les_q$ to indicate the dependence on the parameter $q$. We write $A\sim B$ if both $A\les B$ and $B\les A$.}
\begin{equation*}\label{eq:boundE}
 \EE\lt[ \mathsf{M}^1( (X_i)_{i=1}^n, (Y_j)_{j=1}^n ) \rt]\sim\begin{cases}
                                                                      \sqrt{n} &\textrm{ for } d=1\\                                                                      \sqrt{ n \log n}  & \textrm{ for } d=2\\
                                                                      n^{1 -\frac{1}{d}} & \textrm{ for } d\ge 3.
                                                                     \end{cases}
\end{equation*}
In particular, for $d\in\cur{1,2}$ the cost is  asymptotically larger than the heuristically motivated $n^{1-1/d}$. This exceptional scaling is intuitively due to local fluctuations of the distributions of the two families of points. 


Inspired by the combinatorial approach in \cite{BoutMar} for the random Euclidean bipartite matching problem in dimension $d\ge 3$, Barthe and Bordenave  \cite{BaBo} first proposed a general theory to establish results of BHH-type \eqref{eq:bhh} 
for a wide class of random Euclidean combinatorial optimization problems over two sets of $n$ points. Let us point out that the equality in \eqref{eq:bhh} is actually only proven
for uniform measures while in general only upper and lower bounds (which are conjectured to coincide) are known. 
 In case of $p$-th power weighted distances, the theory developed in \cite{BaBo} applies in the range $0<p<d/2$, which appears quite naturally in their arguments. 
 The difficulty to go beyond the threshold $p=d/2$ is that  \eqref{eq:bhh}  cannot hold  without additional hypothesis on the density $\rho$. For example, because of fluctuations a necessary 
 condition is connectedness of the support of $\rho$.
 Nevertheless, in the case of the Euclidean bipartite matching problem, it was recently proved \cite{goldman2021convergence} that if $\rho$ is the uniform measure on the unit cube with $d \ge 3$ and $p \ge 1$, then
\begin{equation}\label{eq:matching-bhh-cube} \lim_{n \to \infty}n^{\frac{p}{d}-1}  \EE\lt[ \Mp( (X_i)_{i=1}^n, (Y_j)_{j=1}^n ) \rt] = \beta_{\mathsf{M}}.\end{equation}
Here $\beta_{\mathsf{M}} \in (0, \infty)$ depends on $d$ and $p$ only.
The proof is a combination of classical subadditivity arguments -- that originate from \cite{beardwood1959shortest} -- and tools from the theory of optimal transport.
In particular, the defect in subadditivity is  estimated  using the connection between Wasserstein distances and  negative Sobolev norms. In this context, 
the use of this type of  estimates  can be traced back to a recent PDE ansatz proposed in statistical physics \cite{CaLuPaSi14}. Since then, it has been   successfully used 
in the mathematical literature  \cite{AmStTr16,Le17,holden2018gravitational, goldman2021quantitative, BobLe19, ledoux2019optimal, goldman2022fluctuation, huesmann2021there, chen2022asymptotics}, even beyond the case of i.i.d.\ points \cite{wang2019limit, jalowy2021wasserstein, huesmann2022wasserstein, borda2021berry}.
We refer to  \cite{caracciolo2015scaling, BeCa, benedetto2021random} for further statistical physics literature. In fact, the technique in \cite{goldman2021convergence} is quite robust and coarser estimates 
can be used, avoiding the use of PDEs. Still, the results apply only for the Euclidean bipartite matching problem thanks to its connection with optimal transport. The main purpose of this paper is to show that for a quite 
general class of bipartite combinatorial problems it is actually possible to rely on the good bounds for the matching problem to obtain the analog of  \eqref{eq:bhh} provided $p<d$. This is inspired by \cite{capelli2018exact} 
where a similar idea is used for the TSP and the $2$-factor problem when $p=d=2$.  

As alluded to, an important open question left from the  theory developed in \cite{BaBo} (see also \cite{DeScSc13}) is the existence of a limit in \eqref{eq:bhh} for general densities. 
The only result in this direction is  \cite{ambrosio2022quadratic}, which established for $p=d=2$ that the limit of the expected cost (suitably renormalized) exists if $\Omega$ 
is a bounded connected open set, with Lipschitz boundary and $\rho$ is H\"older continuous and uniformly strictly positive and bounded from above on $\Omega$. 
This settled a conjecture from \cite{benedetto2021random} and, more importantly for our purposes, combined subadditivity and PDE arguments with a Whitney-type decomposition to take into account the structure of $\Omega$ and its boundary.
While we do not address this question here,  some of the  ideas from \cite{ambrosio2022quadratic} are further developed in this work.

\subsection{Main result}
Our aim is to establish limit results for the cost of a wide class of Euclidean combinatorial optimization problems of two random point sets, in the range $d/2 \le p <d$ for any dimension $d \ge 3$. This overcomes the limitations of \cite{BaBo}, showing that in higher dimensions bipartite problems behave much more similarly to non-bipartite ones. 
Our general theorem can be stated as follows (a precise description of all the assumptions and notation is given in \cref{sec:notation}).

\begin{theorem}\label{thm:main}
Let $d \ge 3$, $p\in [1,d)$ and let $\pP = (\cF_{n,n})_{n \in \mathbb{N}}$ be a combinatorial optimization problem over complete bipartite graphs such that assumptions \ref{as:isomorphism}, \ref{as:spanning}, \ref{as:bddegree}, \ref{ass:local-merging} and  \ref{as:growth} hold and write $\CPp( (x_i)_{i=1}^n, (y_j)_{j=1}^n)$ for the  optimal cost of the problem over the two sets of $n$ points $(x_i)_{i=1}^n$,  $(y_j)_{j=1}^n \subseteq \R^d$, with respect to the Euclidean distance raised to the power $p$. Then, there exists $\beta_{\pP}\in (0, \infty)$ depending on $p$, $d$ and $\pP$ only such that the following hold.

Let $\Omega \subseteq \R^d$ be a bounded open set and assume that it is  either convex  or has  $C^2$ boundary. Let $\rho$ be a H\"older continuous probability density on $\Omega$, uniformly strictly positive and bounded from above. Given i.i.d.\ random variables $(X_i)_{i=1}^\infty$, $(Y_j)_{j=1}^\infty$ with common law $\rho$ 
we have $\PP$-a.s.\ that
\begin{equation}\label{eq:limit-mean-main} 
\limsup_{n \to \infty} n^{\frac{p}{d}-1} \CPp\bra{ (X_i)_{i=1}^n, (Y_j)_{j=1}^n}  \le \beta_{\pP} \int_{\Omega} \rho^{1-\frac{p}{d}}.\end{equation}

Moreover, if $\rho$ is the uniform density and $\Omega$  is either  a cube or has $C^2$ boundary, then the above is a $\PP$-a.s.\ limit and equality holds. 

\end{theorem}

Our assumptions \ref{as:isomorphism}, \ref{as:spanning}, \ref{as:bddegree}, \ref{ass:local-merging} and in particular \ref{as:growth} are slightly stronger than  those introduced in \cite[Section 5.3]{BaBo},
but it is not difficult to show that all the specific examples discussed in \cite{BaBo}  satisfy them. 
In particular, our result apply to the TSP, the minimum weight connected $k$-factor problem and the $k$-bounded degree minimum spanning tree. 
It is thus fair to say that for compactly supported densities, \cref{thm:main}  extends the main results in \cite{BaBo}. 
\begin{remark}
 Let us point out that \eqref{eq:limit-mean-main} also holds in expectation (see \cref{prop:limit-mean-iid}).
\end{remark}
\begin{remark}
 Arguing as in \cite{BaBo} (see also \cite{ambrosio2022quadratic}) and considering a ``boundary'' variant of $\pP$ it should be  possible to adapt the proof of \cref{thm:main}
 to show that there exists $\beta_{\pP}^b>0$ such that 
 \[
  \beta_{\pP}^b \int_{\Omega} \rho^{1-\frac{p}{d}}\le \liminf_{n \to \infty} n^{\frac{p}{d}-1} \CPp\bra{ (X_i)_{i=1}^n, (Y_j)_{j=1}^n}. 
 \]
However since we are currently not able to prove that $\beta_{\pP}^b=\beta_{\pP}$ we decided to leave it aside. 
\end{remark}

\begin{remark}
In fact our result applies, at least in expectation, to any $p-$homogeneous bi-partite functional $\C$ satisfying the  subadditivity inequality \eqref{Sp} (which is similar to the condition $(\mathcal{S}_p)$ from \cite{BaBo}) and the growth condition \eqref{Rp} (somewhat reminiscent of condition $(\mathcal{R}_p)$ from \cite{BaBo}). See  \cref{rem:phom}.  
\end{remark}

Of course, our result applies in particular for the Euclidean assignment problem. 

\begin{corollary}\label{cor:matching}
For $d \ge 3$, $p\in [1,d)$, let $\Omega \subseteq \R^d$ be a cube or a bounded connected open set with $C^2$ boundary and let $\rho$ be a H\"older continuous probability density on $\Omega$, uniformly strictly positive and bounded from above.  Then, given i.i.d.\ $(X_i)_{i=1}^\infty$, $(Y_j)_{j=1}^\infty$ with common law $\rho$, we have $\PP$-a.s.\ that
\[ \limsup_{n \to \infty}n^{\frac{p}{d}-1}  \Mp\bra{ (X_i)_{i=1}^n, (Y_j)_{j=1}^n}  \le \beta_{\mathsf{M}} \int_{\Omega} \rho^{1-\frac{p}{d}},\]
with $\beta_{\mathsf{M}}$ as in \eqref{eq:matching-bhh-cube}. Moreover, if $\rho$ is the uniform density and $\Omega$  has $C^2$ boundary, then the above is a $\PP$-a.s.\ limit and equality holds.
\end{corollary}

\begin{remark}
 In the case of the matching problem, combining ideas from this paper and \cite{goldman2021convergence} the conclusion of \cref{cor:matching} could be extended to every $p\ge 1$ (at least in expectation).
\end{remark}

\subsection{Comments on the proof technique}
 
Our proof leverages on the techniques developed for the bipartite matching problem, in particular \cite{goldman2021convergence, ambrosio2022quadratic} to carefully estimate the defects in a geometric subadditivity argument. Comparing the approach in \cite{BaBo}, which works if $p<d/2$, with that in \cite{goldman2021convergence}, which holds instead for any $p$, a crucial difference is that the errors due to local oscillations in the two distributions of points are mitigated in the latter by spreading them evenly across all the points. This  is possible since the optimal transport relaxation allows for general couplings as well as continuous densities, rather than discrete matchings only.  

The overall strategy is thus to find a suitable replacement for such operation in the purely combinatorial setting. The starting point is \cref{prop:partition} where we prove a subadditivity inequality. The problem is then to estimate the defect in subadditivity. This is   achieved by combining the following three key observations. 



 The first one is to bound from above the cost of the problem over any two point sets $(x_i)_{i=1}^n$, $(y_j)_{j=1}^n$ by the sum of a term of order $n^{1-p/d}$ plus the bipartite matching cost between the two point sets. This is stated as an assumption (\ref{as:growth}), but can be easily checked on many specific problems (\cref{lem:capelli}): being an upper bound, it usually suffices to combine an optimal matching with the solution to an additional non-bipartite combinatorial optimization problem, such as the TSP, to build a feasible solution. This approach was first successfully used in \cite{capelli2018exact} (see also \cite{ambrosio2022quadratic}) for the random bipartite TSP in the case $p=d=2$, where one can simply argue that the main contribution comes from the logarithmic corrections in the matching cost. 

The second key observation  is that for point sets  mostly made of  i.i.d.\ points (while much less is assumed on the remaining ones), it is still  possible to obtain good bounds for the matching  cost.  We refer to \cref{sec:ot} for the precise statements, but the underlying idea is strongly related to bounds for the optimal transport cost in terms of the negative Sobolev norms -- thus relying again on the PDE ansatz originally introduced in the statistical physics literature.

 The third observation is that, in order to ensure that a small fraction of i.i.d.\ uniformly distributed points can indeed be found in the subadditivity defect terms, it is enough to keep them out of the optimization procedure on the smaller scales. As usual with those arguments, the proof of existence of the limit is performed first on the Poisson version of the random problem, so to retain a fraction of points we perform a thinning procedure.

Besides these main ideas, plenty of technical modifications with respect to the arguments in \cite{BaBo} and \cite{goldman2021convergence, ambrosio2022quadratic} are required, e.g.\ in order to establish improved subadditivity inequalities  (\cref{prop:partition}) and to extend the Whitney-type decomposition argument from \cite{ambrosio2022quadratic} to $p\neq 2$.

\subsection{Further questions and conjectures}
Our results raise several questions about costs and properties of solutions to Euclidean random combinatorial optimization problems over two point sets. We list here a few which we believe are worth exploring.
\begin{enumerate}[label=\emph{\arabic*}.]
\item Existence of a limit in \eqref{eq:limit-mean-main} for non-uniform densities is rather easy to conjecture, but so far our techniques do not improve upon \cite{BaBo}, hence the problem remains largely open.
\item Our techniques break down if $p \ge d$, but it is natural to conjecture that \cref{thm:main} should hold also in that range. In fact, the correct rate $n^{1-p/d}$ could follow directly from (\ref{as:growth}) combined with the corresponding result for the matching problem.
\item In this work we considered only the case of compactly supported densities $\rho$. It would be interesting to investigate the case where the support is $\R^d$. To the best of our knowledge, the only results available so far in this direction are \cite{Le17,ledoux2019optimal} where the correct rates are established for the Gaussian density in the case of the matching problem. 
\item The assumptions in \cite{BaBo} are slightly different than ours, although the specific problems considered therein satisfy both. It would be interesting to find examples which satisfy only one set of these, or possibly simplify even more our assumptions.
\item Many problems, such as the bounded degree minimum spanning tree, but also the bipartite matching problem itself, can be naturally formulated also for two families of points with different number of elements: it could be of interest to investigate limit results also in those cases.
\item The cases $d \in \cur{1,2}$ are necessarily excluded by our analysis, since subadditivity arguments do not apply already for the random bipartite matching problem. It is however already an open question, whether the additional logarithmic correction indeed appears in the asymptotic rates for many other problems. As an example, we mention that for the Euclidean minimum spanning tree over two random point sets (without any uniform bound on the degree) no logarithmic corrections appear \cite{correddu2021minimum}, but the maximum degree is unbounded, hence it is not covered by our results. 
\item In the deterministic literature, for the TSP and other NP-hard Euclidean combinatorial optimization problems, polynomial time approximation schemes are known \cite{arora2003approximation} for any (fixed) dimension $d$, as the number of points grows. Can our approach lead to similar schemes for problems on two families of points, possibly under some mild regularity assumption on their spatial distributions?
\end{enumerate}

\subsection{Structure of the paper}  In \cref{sec:notation} we first introduce some general notation.  We then discuss Whitney-type decompositions, Sobolev spaces as well as recall useful known facts on the Optimal Transport problem, and possibly some novel ones (\cref{prop:density-helps}). We close the section with a variant of the standard subadditivity (Fekete-type) arguments, suited for our purposes together with some simple concentration inequalities. \cref{sec:cop} is devoted to the combinatorial optimization problems we consider, discussing in particular the main assumptions that we require and some useful consequences. In \cref{sec:poisson} we establish a variant our main result in the case of Poisson point processes and in \cref{sec:main} we use it to deduce \cref{thm:main}. These two sections in fact rely upon the novel bounds for the Euclidean assignment problem that we finally establish in \cref{sec:ot}.

\section{Notation and preliminary results}\label{sec:notation}

\subsection{General notation}
Given $n \in \mathbb{N}$, we write $[n] = \cur{1, \ldots, n}$ and $[n]_1 = \cur{ (1,i)}_{i=1}^n$, $[n]_2 = \cur{(2,i)}_{i=1}^n$, which easily allows to define two disjoint copies of $[n]$. Given a finite set $A$, we write $|A|$ for the number of its elements, while, if $A \subseteq \R^d$ is infinite, $|A|$ denotes its Lebesgue measure.

Given a metric space $(\Omega, \dist)$,  $x \in \Omega$, $A\subseteq \Omega$, we write $\dist(x, A) = \min_{y \in A}\cur{ \dist(x,y)}$ and $\diam(A) = \sup_{x,y \in A} \dist(x,y)$. 
We endow every set  $\Omega \subseteq \R^d$  with the Euclidean distance. 
A partition  $\cur{\Omega_k}_{k=1}^K$ of a set $\Omega$ is  always intended up to a set of Lebesgue measure zero.
A rectangle $R \subseteq \R^d$ is a subset of the form $R = \prod_{i=1}^d (x_i,x_i+L_i)$, and is said to be of moderate aspect ratio if for every $i,j$, $L_i/L_j\le 2$. If $L_i= L$ for every $i$, then $R = Q$ is a cube of side length $L$. We write $Q_L = (0,L)^d$. We write $I_\Omega$ for the indicator function of a set $\Omega$.

\subsection{Families of points}
Given a set $\Omega$, we consider finite ordered families of points $\x = \bra{x_i}_{i=1}^n \subseteq \Omega$, with $n \in \N$, letting $\x = \emptyset$ if $n=0$. For many purposes the order will not be relevant, but we thus may allow e.g.\ for repetitions (which will be probabilistically negligible anyway). Given a family $\x \subseteq \R^d$, we write $\mu^{\x} = \sum_{i=1}^n \delta_{x_i}$ for the associated empirical measure and, for every (Borel) $\Omega \subseteq \R^d$, we let $\x(\Omega) = \mu^{\x}(\Omega)$. In the special case $\Omega = \R^d$, we simply write $|\x| = \x(\R^d) = \mu^{\x}(\R^d)$ for the total number of points (counted with multiplicity). We also write $\x_ \Omega$ for its restriction to $\Omega$, i.e., the family of all points $x_i \in \Omega$, so that  $\x = \x_{\Omega}$ if $\x \subseteq \Omega$ (conventionally, we naturally re-index it over $i=1, \ldots, \x(\Omega)$ with the order inherited from that in $\x$). Given $\x = \bra{x_i}_{i=1}^n$, $\y= \bra{y_j}_{j=1}^m\subseteq \R^d$, their union is $\x \cup \y = (x_1, \ldots, x_n, y_1, \ldots, y_m)$. 
Strictly speaking, the union should be called concatenation, since the operation is not commutative, in general. 

\subsection{Whitney partitions}
We recall following partitioning result \cite[Lemma 5.1]{ambrosio2022quadratic}. 


\begin{lemma}\label{lem:decomp}
Let $\Omega\subset \R^d$ be a bounded domain with Lipschitz boundary and let $\cQ = \{Q_i\}_i$ be a Whitney partition of $\Omega$. Then, for every $\delta>0$ sufficiently small, letting $\cQ_\delta=\{Q_i \ : \ \diam(Q_i) \ge \delta\}$, there exists a finite family $\cR_\delta=\{\Omega_j\}_j$ of disjoint open sets such that:
\begin{enumerate}[label=(\roman*),  series=partition]
\item \label{partition-1} $(\Omega_k)_{k=1}^K =  \cQ_\delta \cup \cR_\delta$ is a partition of $\Omega$,
\item \label{partition-2}$ |\Omega_k| \sim \diam(\Omega_k)^d$ for every $k=1, \ldots, K$,
\item \label{partition-3}if $\Omega_k \in \cQ_\delta$, then $\diam(\Omega_k) \sim \dist(x, \Omega^c)$ for every $x \in \Omega_k$,
\item \label{partition-4}if $\Omega_k \in \cR_\delta$, then $\diam(\Omega_k)\sim \delta$ and $\dist(x, \Omega^c) \les \delta$, for every $x \in \Omega_k$.
\end{enumerate} 
Here all the implicit constants depend only on the initial partition $\cQ$ (and not on $\delta$).
\end{lemma}

%
For later use, we collect  some useful bounds related to these partitions.

\begin{lemma}\label{lem:bound-partition}
Let $\Omega\subset \R^d$ be a bounded domain with Lipschitz boundary and let $\cQ = \{Q_i\}_i$ be a Whitney partition of $\Omega$. Then, for every $\delta>0$ sufficiently small, letting $(\Omega_k)_{k=1}^K = \cQ_\delta \cup \cR_\delta$ as in \cref{lem:decomp}, one has that $|\cR_\delta| \les \delta^{1-d}$ and the following holds:
\begin{enumerate}
\item For every $\alpha \in \R$,
 \begin{equation}\label{eq:whitney-general-q}
 \sum_{k=1}^K \diam(\Omega_k)^{\alpha} \les_\alpha \begin{cases} 1 & \text{if $\alpha>d-1$,}\\
 |\log \delta|& \text{if $\alpha=d-1$,}\\
 \delta^{1-d+\alpha} & \text{if $\alpha<d-1$.} 
 \end{cases}
\end{equation}
\item If $\alpha<0$, then for every $k =1,\ldots, K$, and $x \in \Omega_k$,
\begin{equation}\label{eq:final-bound-min-sum-partition}
\sum_{j=1}^K \diam(\Omega_j)^{\alpha} \min\cur{1, \bra{ \frac{\diam(\Omega_j)}{\dist(x, \Omega_j)}}^{d-1}} \les  \delta^{\alpha} |\log (\delta)|.
\end{equation}
\end{enumerate}
In all the inequalities the implicit constants depend upon $\alpha$ and $\cQ$  in \eqref{eq:whitney-general-q} only.
\end{lemma}

By property \emph{\ref{partition-2}}, inequality \eqref{eq:whitney-general-q} also holds for the sum $\sum_{k=1}^K |\Omega_k|^{\alpha}$, with $\alpha d$ instead of $\alpha$.

\begin{proof}

Since $\partial \Omega$ is Lipschitz, it follows from properties \ref{partition-1}, \ref{partition-2} and \ref{partition-4} that, for every $x \in \Omega$ and $r \ge s \ge \delta$, 
\begin{equation}\label{eq:uniform-bound-omega-k} \abs{ \cur{k\, : \, \Omega_k\subseteq B(x, r), \diam(\Omega_k) \in [s, 2s)}} \les (r/s)^{d-1},\end{equation}
 with the implicit constant depending on $\cQ$ only.
 It follows that
 $|\cR_\delta| \les \delta^{1-d}$ and, for every $\ell \le |\log_2 \delta|$, the number of cubes $\Omega_k \in \cQ_\delta$ with $\diam(\Omega_k) \in [2^{-\ell}, 2^{-\ell+1})$ is estimated by $ 2^{\ell (d-1)}$. Therefore, for $\alpha \in \R$,
\[
\begin{split}  \sum_{k=1}^K \diam(\Omega_k)^{\alpha} & \les \sum_{\Omega_k \in \cQ_{\delta}} \diam(\Omega_k)^{\alpha}  + \sum_{\Omega_k \in \cR_\delta} \diam(\Omega_k)^{\alpha} \\
&  \les  \sum_{\ell \le |\log_2 \delta|} \abs{\cur{ \Omega_k \in \cQ_{\delta} : \diam(Q_k) \in [2^{-\ell}, 2^{-\ell+1})}}  2^{-\ell \alpha} + |\cR_\delta| \cdot \delta^{\alpha}\\
 & \les \sum_{\ell \le |\log_2 \delta|} 2^{\ell(d-1)}\cdot  2^{-\ell \alpha} + \delta^{1-d} \cdot \delta^{\alpha}.\end{split} 
\]
Since  $\ell$ is also bounded from below in the summation (e.g.\ by $-|\log_2\diam(\Omega)|$), we obtain \eqref{eq:whitney-general-q}.



We next prove \eqref{eq:final-bound-min-sum-partition}. We claim that it follows from the following inequalities, valid for any $\gamma \in \mathbb{N}$:
\begin{equation}\label{eq:before-last-claim}
\sum_{j\, : \, \dist(x, \Omega_j) \le 2^{-\gamma} \diam(\Omega_k)} \diam(\Omega_j)^\alpha \les 2^{-\gamma(d-1)} \diam(\Omega_k)^{d-1} \delta^{\alpha+1-d},
\end{equation}
and, for $\beta<d-1$,
\begin{equation}\label{lastclaim}
   \sum_{j \, : \, \dist(x,\Omega_j)>2^{-\gamma} \diam(\Omega_k)} \frac{\diam(\Omega_j)^\beta }{\dist(x, \Omega_j)^{d-1}} \les |\gamma+ \log \bra{\diam(\Omega_k)}| \delta^{\beta+1-d}.
\end{equation}
Indeed, we can split the summation and use \eqref{eq:before-last-claim} and \eqref{lastclaim} to get
\begin{equation}\begin{split}
\sum_{j}  &  \diam(\Omega_j)^{\alpha} \min\cur{1, \bra{ \frac{\diam(\Omega_j)}{\dist(x, \Omega_j)}}^{d-1}} \\
&  \les  \sum_{j\, : \, \dist(x, \Omega_j) \le 2^{-\gamma} \diam(\Omega_k)} \diam(\Omega_j)^\alpha  +  \sum_{j\, : \, \dist(x, \Omega_j) > 2^{-\gamma} \diam(\Omega_k)} \frac{\diam(\Omega_j)^{d-1+\alpha}}{\dist(x, \Omega_j)^{d-1}}\\
& \les 2^{-\gamma(d-1)} \diam(\Omega_k)^{d-1} \delta^{\alpha+1-d} + |\gamma+ \log \bra{\diam(\Omega_k)}| \delta^{\alpha}.
\end{split}
\end{equation}
Recalling that $\diam(\Omega_k) \gtrsim \delta$ and choosing $\gamma$ so that $2^{-\gamma} \le \delta \le 2^{-\gamma+1}$ yields \eqref{eq:final-bound-min-sum-partition}.

In order to prove \eqref{eq:before-last-claim} and \eqref{lastclaim} we first notice that, given $\Omega_k$, $\Omega_j$ and $x \in \Omega_k$, we have that, for some constant $C = C(\cQ)$,
\begin{equation}\label{eq:omega-j-contained-ball} \Omega_j \subseteq B(x, C \max\cur{\dist(x, \Omega_j), \diam(\Omega_k)}).\end{equation}
Indeed, if $\Omega_j \in \cR_\delta$, then $\diam(\Omega_j)\les \delta\les \diam(\Omega_k)$, hence \eqref{eq:omega-j-contained-ball} holds.  
If instead $\Omega_j \in \cQ_\delta$, then we can find $y \in \Omega_j$ with $|x-y|\le 2 \dist(x,\Omega_j)$, so that, by the triangle inequality,
\begin{equation}
\dist(y, \Omega^c) \le |x-y| + \dist(x, \Omega^c) \les \max\cur{\dist(x, \Omega_j), \diam(\Omega_k)}
\end{equation}
and by property (\ref{partition-3}) in \Cref{lem:decomp} we obtain that $\diam(\Omega_j) \les \max\cur{\dist(x, \Omega_j), \diam(\Omega_k)}$, yielding again the desired inclusion.

Hence, we prove \eqref{eq:before-last-claim} and \eqref{lastclaim}. Let $\ell_k \le |\log_2 \delta|$ be such that $\diam(\Omega_{k}) \in [2^{-\ell_k}, 2^{-\ell_k+1})$.  Combining \eqref{eq:omega-j-contained-ball} and \eqref{eq:uniform-bound-omega-k}, we see that, for every $\ell \le |\log_2\delta|$, there are  at most $2^{(\ell-\ell_k-\gamma)(d-1)}$ sets $\Omega_j$ such that $\dist(x, \Omega_j) \le 2^{-\gamma}\diam(\Omega_k)$ and $\diam(\Omega_j) \in [2^{-\ell}, 2^{-\ell+1})$. Therefore,
\begin{equation}
 \begin{split} \sum_{j\, : \, \dist(x, \Omega_j) \le 2^{-\gamma} \diam(\Omega_k)} \diam(\Omega_j)^{\alpha} & \les \sum_{\ell \le |\log_2\delta|} 2^{-\ell \alpha} 2^{(\ell-\ell_k)(d-1)} \\
& \les 2^{-(\gamma+\ell_k)(d-1)} \sum_{\ell \le |\log_2\delta|}  2^{-\ell(\alpha+1-d)} \\ &   \les  2^{-\gamma(d-1)}\diam(\Omega_k)^{d-1} \delta^{\alpha+1-d}.
\end{split}
\end{equation}
This proves \eqref{eq:before-last-claim}. To prove \eqref{lastclaim}, we split dyadically,
\begin{equation}\label{eq:final-claim-step}\begin{split}
 \sum_{j \, : \, \dist(x,\Omega_j)>2^{-\gamma}\diam(\Omega_k)}  \frac{\diam(\Omega_j)^\beta }{d(x,\Omega_j)^{d-1}} & \les \sum_{\ell \le \ell_k+\gamma} \frac{1}{(2^{-\ell})^{d-1}} \sum_{ j\, : \, d(x,\Omega_j) \in [2^{-\ell}, 2^{-\ell+1})} \diam(\Omega_j)^\beta\\ 
 & \stackrel{\eqref{eq:omega-j-contained-ball}}{\les} \sum_{\ell \le \ell_k+\gamma} 2^{\ell(d-1)} \sum_{ \Omega_j \subset B(x,C 2^{-\ell})} \diam(\Omega_j)^\beta.
\end{split}\end{equation}
Let us also notice that, if $\Omega_j \subseteq B(x, C2^{-\ell})$, then necessarily $\delta \le \diam(\Omega_j) \les 2^{-\ell}$ (since $\diam(\Omega_j)^d \sim |\Omega_j|$).  Thus for $\ell'$ with $2^{-\ell'} \sim 2^{-\ell}$,
\begin{equation}\begin{split}
 \sum_{ \Omega_j \subset B(x, C 2^{-\ell})}\diam(\Omega_j)^\beta& \les \sum_{\ell' \le u \le |\log_2 \delta|} 2^{-u \beta} \sharp {\cur{ \Omega_j \subseteq B(x, C 2^{-\ell}) \, : \diam(\Omega_j) \in [2^{-u}, 2^{-u+1}) }}\\
 & \stackrel{\eqref{eq:uniform-bound-omega-k}}{\les} \sum_{\ell' \le u \le |\log_2 \delta|} 2^{-u\beta} \cdot 2^{(u-\ell) (d-1) }= 2^{-\ell(d-1)} \sum_{\ell' \le u \le |\log_2 \delta|} 2^{-u(\beta+1-d)}\\
 & \les 2^{-\ell(d-1)} \delta^{\beta+1-d},
\end{split}
\end{equation}
using again that $\ell'$ is bounded from below by a constant depending on $\mathcal{Q}$ only.
Plugging this bound in \eqref{eq:final-claim-step}, we conclude that
\begin{equation}\begin{split}
 \sum_{j \, : \, \dist(x,\Omega_j)>2^{-\gamma} \diam(\Omega_k)}  \frac{\diam(\Omega_j)^\beta}{d(x,\Omega_j)^{d-1}}  & \le \sum_{\ell \le \ell_k+\gamma} 2^{\ell(d-1)} \cdot  2^{-\ell(d-1)} \delta^{\beta+1-d}\\
 & \les  \bra{ \gamma + |\log \bra{ \diam(\Omega_k)}|} \delta^{\beta+1-d}.
\end{split}
\end{equation}
This concludes the proof of \eqref{lastclaim}.
\end{proof}
\subsection{Sobolev norms}


Given a bounded domain $\Omega \subseteq \R^d$ with Lipschitz boundary and $p \in (1,\infty)$, with H\"older conjugate $q = p/(p-1)$, we write $\| f \|_{L^p(\Omega)}$ for the Lebesgue norm of $f$, and 
\[
 \|f\|_{W^{-1,p}(\Omega)}=\sup_{|\nabla \phi|_{L^{q}(\Omega)}\le 1} \int_{\Omega} f \phi=\inf_{\textrm{div} \xi=f} \|\xi\|_{L^p(\Omega)}
\]
for the negative Sobolev norm. We notice in particular that if $\|f\|_{W^{-1,p}(\Omega)}<\infty$ then $\int_\Omega f=0$. In this case we may also restrict the supremum to functions $\phi$ having also average zero. When it is clear from the context, we will drop the explicit dependence on  $\Omega$ in the norms. 

Let us recall that we can bound the $W^{-1,p}$ norm by the $L^p$ norm.
We give  here a proof    based on the  embedding $L^{pd/(p+d)}\subset W^{-1,p}$ (for $p>d/(d-1)$)  which is  an elementary alternative to the PDE arguments used in  \cite[Lemma 3.4]{goldman2021convergence}.
\begin{lemma}
 Let $\Omega$ be a bounded domain with Lipschitz boundary and let $f:\Omega\to \R$ such that $\int_\Omega f=0$. Then, for every $p>d/(d-1)$,
 \begin{equation}\label{eq:Lp}
  \|f\|_{W^{-1,p}(\Omega)}\les |\Omega|^{\frac{1}{d}} \|f\|_{L^p(\Omega)}.
 \end{equation}
Moreover, the implicit constant depends on $\Omega$ only through the corresponding constant for the Sobolev embedding.
\end{lemma}
\begin{proof}
 Let $q$ be the H\"older conjugate of $p$, $q^*$ the Sobolev conjugate of $q$ and $p^*=pd/(p+d)$ the H\"older conjugate of $q^*$. We then have for every $\phi$ with $\|\nabla \phi\|_{L^q(\Omega)}\le 1$,
 \[
  \int_{\Omega} f \phi\le \lt(\int_{\Omega} |f|^{p^*}\rt)^\frac{1}{p^*} \lt(\int_{\Omega} |\phi|^{q^*}\rt)^{\frac{1}{q^*}}\les  \lt(\int_{\Omega} |f|^{p^*}\rt)^\frac{1}{p^*} \lt(\int_{\Omega} |\nabla \phi|^{q}\rt)^{\frac{1}{q}}\le \lt(\int_{\Omega} |f|^{p^*}\rt)^\frac{1}{p^*}.
 \]
Using that $p^*<p$ and H\"older inequality concludes the proof of \eqref{eq:Lp}.
\end{proof}
As in \cite{ambrosio2022quadratic}, \eqref{eq:Lp} will however not be precise enough when estimating the error in subadditivity in the case of general densities and domains. We will instead rely on gradient bounds for 
the Green kernel $(G(x,y))_{x,y\in \Omega}$ of the Laplacian with Neumann boundary conditions to obtain sharper estimates. See \cite{AmStTr16,AmGlau,Le17,goldman2022fluctuation} for related results. Let us however point out that in our case we will not rely on any stochastic cancellation in the form of Rosenthal inequality \cite{rosenthal1970subspaces} but will instead use a purely deterministic estimate. We will assume that 

\begin{equation}\label{eq:green-kernel-bound} \abs{ \nabla _x  G(x,y) } \les |x-y|^{1-d}, \quad \text{for every $x$, $y \in \Omega$,}\end{equation}
where the implicit constant depends uniquely on $\Omega$. 
\begin{remark}\label{rem:green}This condition is satisfied for instance  if  $\Omega$ is $C^2$ or convex, see e.g.\ \cite{wang2013gradient}. Notice that since it is a local condition it also holds for $Q\backslash \Omega$ with $\Omega$ a $C^2$ open set with $d(\partial Q,\partial \Omega)>0$.
\end{remark}
\begin{remark}\label{eq:Riesz}
 Let us point out that as in \cite{Le17}, instead of \eqref{eq:green-kernel-bound} it would have been enough to have $L^p$ bounds (for  the same $p$ as for the cost $\CPp$)   on the Riesz transform for the Neumann Laplacian. From the available results for the Dirichlet Laplacian \cite{JerKenInhom,Shen}, we expect that for every Lipschitz domain there is $p>3$ (depending on the domain) for which these bounds hold. In particular, this would allow to extend the validity of Theorem \ref{thm:main} to every Lipschitz domain when $d=3$. However, since we were not able to find in the literature the corresponding results for the case of Neumann boundary conditions we kept the stronger hypothesis \eqref{eq:green-kernel-bound}.
\end{remark}

We then have 
\begin{lemma}\label{lem:W1pWhitney}
 Let $\Omega\subset \R^d$ be a bounded domain with Lipschitz boundary, such that \eqref{eq:green-kernel-bound} holds and let $\rho$ be a density bounded above and below on $\Omega$. For $\delta>0$ sufficiently small, let $(\Omega_k)_{k=1}^K = \cQ_\delta \cup \cR_\delta$ as in \cref{lem:decomp}. 
 If there exists $h>0$ such that   $|b_k|\le h^{\frac{1}{2}} |\Omega_k|^{\frac{1}{2}}$ for $k=1, \ldots, K$,  then for every $p\ge 1$,
\begin{equation}\label{estimW1pWhitney}
 \lt\|\sum_{k=1}^K \frac{b_k}{\rho(\Omega_k)}( I_{\Omega_k} -\rho(\Omega_k))\rho\rt\|_{W^{-1,p}(\Omega)}\les \delta^{1-\frac{d}{2}} |\log (\delta) |h^{\frac{1}{2}}.
\end{equation}

\end{lemma}
\begin{proof}
Set
$$ B_k = \frac{ b_k}{\rho(\Omega_k)}, \qquad f_k= \bra{I_{\Omega_k}  - \rho(\Omega_k)} \rho.$$
Let then $\phi_k$ denotes the solution to the equation $\Delta \phi_k = f_k$, with null Neumann boundary conditions on $\Omega$ and use as competitor $\xi=\sum_{k=1}^K  B_k \nabla \phi_k$ in the definition of the $W^{-1,p}$ norm.
We get,
\begin{equation}\label{normW1pA}
 \nor{ \sum_{k=1}^K  B_k f_k}_{W^{-1,p}(\Omega)}^p  \le \int_{\Omega} \abs{ \sum_{k=1}^K B_k {\nabla \phi_k}}^p \les h^{\frac{p}{2}}  \int_{\Omega} \bra{\sum_{k=1}^K |\Omega_k|^{-\frac{1}{2}} \abs{\nabla \phi_k}}^p.
\end{equation}
To bound the last term, we use the integral representation  in terms of the Green's function,
\[ \phi_k = \int_{\Omega}  G(x,y) f_k(y) d y, \]
to obtain that, for every $x\in \Omega$,
\begin{equation}\label{boundgradphii}
|\nabla \phi_k(x)|\les \min\cur{\diam(\Omega_k), \frac{|\Omega_k|}{\dist(x,\Omega_k)^{d-1}}}.
\end{equation}
Indeed, by \eqref{eq:green-kernel-bound},
\[\begin{split}
  |\nabla \phi_k(x)|& \les \int_{\Omega_k} \frac{dy}{|x-y|^{d-1}}+ |\Omega_k| \int_{\Omega} \frac{dy}{|x-y|^{d-1}} 
   \le \int_{\cur{|y|\le \diam(\Omega_k)}} \frac{dy}{|y|^{d-1}} +|\Omega_k|\\
   & \les \diam(\Omega_k).
\end{split}\]
Moreover, for $x\notin \Omega_k$, we get directly from \eqref{eq:green-kernel-bound},
\[
 |\nabla \phi_k(x)|\les \frac{|\Omega_k|}{\dist(x,\Omega_k)^{d-1}}.
\]
For any $k=1, \ldots, K$ and $x \in \Omega_k$, we then estimate
\[ \begin{split} \sum_{j=1}^K |\Omega_j|^{-\frac{1}{2}} \abs{\nabla \phi_j(x)} &  \stackrel{\eqref{boundgradphii}}{\les}  \sum_{j=1}^K \diam(\Omega_j)^{1-d/2} \min\cur{1, \bra{ \frac{\diam(\Omega_j)}{\dist(x, \Omega_j)}}^{d-1}} \\
& \les \delta^{1-d/2}|\log(\delta)|
\end{split}\]
 having used inequality \eqref{eq:final-bound-min-sum-partition} from \cref{lem:bound-partition} with $\alpha = 1-d/2$.

Therefore, we can split the integration
\[\begin{split} \int_{\Omega} \bra{\sum_{j=1}^K |\Omega_j|^{-\frac{1}{2}} \abs{\nabla \phi_j}}^p & = \sum_{k=1}^K \int_{\Omega_k} \bra{\sum_{j=1}^K |\Omega_j|^{-\frac{1}{2}} \abs{\nabla \phi_j}}^p 
 \les \delta^{(1-d/2)p} | \log (\delta ) |^p
\end{split} \]
In combination  with \eqref{normW1pA} this concludes the proof of \eqref{estimW1pWhitney}.
\end{proof}


\subsection{Optimal Transport}

Given two positive Borel measures $\mu$, $\lambda$ on $\R^d$ with $\mu(\R^d) = \lambda(\R^d) \in (0, \infty)$ and finite $p$-th moments, the optimal transport cost of order $p\ge 1$ between $\mu$ and $\lambda$ is defined as the quantity
\[ \W(\mu, \lambda) = \min_{\pi\in\Gamma(\mu,\lambda)} \int_{\R^d\times\R^d} {|x-y|}^p d \pi(x,y),\]
where $\Gamma(\mu, \lambda)$ is the set of couplings between $\mu$ and $\lambda$, i.e., finite Borel measures $\pi$ on the product $\R^d\times \R^d$ such that their marginals are respectively $\mu$ and $\lambda$.  
%
Notice that if $\mu(\R^d) = \lambda(\R^d) = 0$ then $\W(\mu, \lambda) = 0$, while if $\mu(\Omega) \neq \lambda(\Omega)$, we conveniently extend the definition setting $\W(\mu, \lambda) = \infty$.
Let us recall that the triangle inequality for the Wasserstein distance of order $p$ (which is defined as the the $p$-th root of $\W(\mu, \lambda)$) yields 
\begin{equation}\label{eq:triangle}
\W( \mu, \nu) \lesssim \W(\mu, \lambda) + \W(\nu, \lambda),
\end{equation}
A straightforward, but useful subadditivity inequality is
\begin{equation}\label{eq:convexity}
\W\bra{ \sum_{k} \mu_k, \sum_{k} \nu_k} \le \sum_{k} \W(\mu_k, \nu_k).
\end{equation}
valid for any (countable) family of measures $(\mu_k, \nu_k)_{k}$.


To keep notation simple, we write 
\[ \W_{\Omega} (\mu, \lambda)  = \W (\mu\restr\Omega, \lambda\restr \Omega).\]
and, if a measure is absolutely continuous with respect to Lebesgue measure, we only write its density. For example, $\W_{\Omega} \bra{ \mu,  \mu(\Omega)/|\Omega| }$
denotes the transportation cost between $\mu \restr \Omega$ to the uniform measure on $\Omega$ with total mass $\mu(\Omega)$.  

For $q \ge p$, Jensen inequality gives
\begin{equation}\label{eq:jensen}  \W_{\Omega} (\mu, \nu) \le \mu(\Omega)^{1-\frac{p}{q}} \bra{ \mathsf{W}_{\Omega}^q (\mu, \nu)}^{\frac{p}{q}}.\end{equation}

Our arguments make substantial use of two crucial properties of the optimal transport cost. The first one \cite[Lemma 3.1]{goldman2021convergence} is a simple consequence of \eqref{eq:triangle} and \eqref{eq:convexity}. 

\begin{lemma}
 For every $p\ge 1$, there exists a constant $C>0$ depending only on $p$ such that the following holds. 
Let $\Omega \subseteq \R^d$ be Borel and $(\Omega_k)_{k\in \N}$ be a countable Borel partition of $\Omega$. Then, for finite measures $\mu$, $\lambda$, and $\eps\in(0,1)$, we have the inequality
 \begin{equation}\label{eq:mainsub}
 \W_{\Omega}\lt(\mu, \alpha \lambda \rt)\le (1+\eps)\sum_i \W_{\Omega_k}\lt(\mu, \alpha_k  \lambda  \rt)+ \frac{C}{\eps^{p-1}}\W_\Omega\lt(\sum_k \alpha_k \chi_{\Omega_k}  \lambda , \alpha  \lambda  \rt).
 \end{equation}
 where $\alpha = \mu(\Omega)/\lambda(\Omega)$ and  $\alpha_k = \mu(\Omega_k)/\lambda(\Omega_k)$. 
\end{lemma}

The second one is \cite[Lemma 2.2]{ambrosio2022quadratic} which gives an upper bound for the Wasserstein distance in terms of a negative Sobolev norm. It follows from the Benamou-Brenier formulation of the optimal transport problem  (see also \cite[Corollary 3]{peyre2018comparison}).

 \begin{lemma}\label{lem:peyre} Assume that $\Omega \subseteq \R^d$ is a bounded connected open set with  Lipschitz boundary. If $\mu$ and $\lambda$ are measures on $\Omega$ with $\mu(\Omega) = \lambda(\Omega)$,  absolutely continuous with respect to the  Lebesgue measure and $\inf_\Omega \lambda>0$, then, for every $p\ge 1$,
 \begin{equation}\label{eq:estimCZ}
  W_{\Omega}^p(\mu,\lambda)\les \frac{1}{\inf_{\Omega} \lambda^{p-1}}\nor{ \mu - \lambda}_{W^{-1,p}(\Omega)}^p.
 \end{equation}
 \end{lemma}
 
As in many recent works on the matching problem, we will use this inequality to improve on the trivial bound
\begin{equation}\label{eq:w-trivial} \W_{\Omega} (\mu, \lambda) \le \diam(\Omega)^p \mu(\Omega),\end{equation}
which holds as soon as $\mu(\Omega) = \nu(\Omega)$. Much of our effort in the proofs will be ultimately to deal with an intermediate situation, where the measures can be decomposed as the sum of a ``good'' part, i.e., absolutely continuous with smooth density and a ``bad'' remainder about which not much can be assumed. We prove here a general inequality which  could also be of independent interest.

\begin{proposition}\label{prop:density-helps}
Let $\Omega \subseteq \R^d$ be a bounded Lipschitz domain, $\rho$ be a density bounded above and below on $\Omega$, $\mu$ be any finite measure on $\Omega$ and $h >0$. Then, for every $p>d/(d-1)$,
\begin{equation}\label{claim} \W_\Omega\bra{ \mu +h\rho, \alpha \rho} \les  \frac{1}{h^{\frac{p}{d}}} \mu(\Omega)^{1+\frac{p}{d}},
\end{equation}

where $\alpha= \frac{\mu(\Omega)}{\rho(\Omega)} +h$. Moreover, this inequality is invariant by rescaling of $\Omega$.
\end{proposition}

\begin{proof}
By scaling we may assume that $|\Omega|=1$. Notice that, 
by the trivial bound 
\[
 \W(\mu +h\rho, \alpha \rho)\les \mu(\Omega),
\]
we can assume that $\mu(\Omega) \ll h$. Let $P_t$ be the heat semi-group with null Neumann boundary conditions on $\Omega$ and set $\mu_t=P_t\mu$. 
 By triangle inequality \eqref{eq:triangle} and \eqref{eq:estimCZ}, we have 
\[\begin{split} \W(\mu +h\rho, \alpha \rho)&\les \W(\mu +h\rho, \mu_t+h\rho)+\W(\mu_t+h\rho, \alpha \rho)\\
 &\les \W\bra{ \mu , \mu_t}+\frac{1}{h^{p-1}}\nor{\mu_t- \frac{\mu(\Omega)}{\rho(\Omega)}\rho}^p_{W^{-1,p}}\\
 &\les t^\frac{p}{2} \mu(\Omega) +\frac{1}{h^{p-1}}\nor{ \mu_t- \frac{\mu(\Omega)}{\rho(\Omega)}\rho}^p_{W^{-1,p}}.
\end{split}\]
We now estimate the last term. For this let $q$ be the H\"older conjugate exponent of $p$, i.e., $q=p/(p-1) \in (1, d)$ and $q^* = qd/(d-q)$ be the Sobolev conjugate of $q$. We first use the triangle inequality and the fact that $\rho$ is bounded from above and below to estimate
\[
\begin{split}
 \nor{ \mu_t- \frac{\mu(\Omega)}{\rho(\Omega)}\rho}_{W^{-1,p}} & \le \nor{ \mu_t- \mu(\Omega)}_{W^{-1,p}} + \mu(\Omega)\nor{ 1- \frac{1}{\rho(\Omega)}\rho}_{W^{-1,p}}\\
 & \les \nor{ \mu_t- \mu(\Omega)}_{W^{-1,p}} + \mu(\Omega).
\end{split}
\]
Using that the Sobolev embedding is equivalent to ultra-contractivity i.e. if $\int_\Omega \phi=0$,
\[
 \|\phi_t\|_{L^\infty(\Omega)}\les  t^{-\frac{d}{2q_*}} \|\phi\|_{L^{q^*}(\Omega)}\les t^{-\frac{d}{2q_*}} \|\nabla \phi\|_{L^{q}(\Omega)} ,
\]
we finally estimate for every $\phi$ with $\|\nabla \phi\|_{L^{q}(\Omega)}\le 1$ and $ \int_{\Omega} \phi=0$,
\[
 \int_{\Omega}\phi(\mu_t-\mu(\Omega))= \int_\Omega \phi_t d\mu\le \mu(\Omega) \|\phi_t\|_{L^\infty(\Omega)}\le \mu(\Omega) t^{-\frac{d}{2q_*}}.\]
Therefore, by taking the supremum over $\phi$ we find 
\[
 \nor{ \mu_t- \mu(\Omega)}_{W^{-1,p}}\les \mu(\Omega )t^{-\frac{d}{2q_*}}.
\]
Taking the $p$-th power we find for $t\le 1$,
\[\begin{split}
 \W(\mu +h\rho, \alpha \rho)& \les \mu(\Omega)\lt[t^{\frac{p}{2}}  +t^{-\frac{pd}{2q^*}}\lt(\frac{\mu(\Omega)}{h}\rt)^{p-1}\rt]\\
 & =\mu(\Omega)\lt[t^{\frac{p}{2}}  +t^{-\frac{p}{2}\lt(\frac{d}{q}-1\rt)}\lt(\frac{\mu(\Omega)}{h}\rt)^{p-1}\rt]
\end{split}.\]
Optimizing in $t$ we find  $t^{\frac{p}{2}}=\lt(\frac{\mu(\Omega)}{h}\rt)^{\frac{(p-1)q}{d}}$ which satisfies $t\ll1$ if $\mu(\Omega)\ll h$. Since $(p-1)q=p$, this concludes the proof of \eqref{claim}.
\end{proof}

\begin{remark}
Since by H\"older inequality it will be enough for us to apply \cref{prop:density-helps} for $p$ arbitrarily close to $d$,  the condition  $p>d/(d-1)$ will not be a limitation for us.   Let us  however mention that, in the critical case $p=d/(d-1)$ one can argue similarly, relying instead on the Moser-Trudinger inequality \cite[Remark 1.4]{Cianchi}, to obtain (in the case $\rho=1$ and $\Omega=Q$ a cube for simplicity)
$$ \W_Q\bra{ \mu +\frac{h}{|Q|}, \frac { \mu(Q)}{|Q|} + \frac{h}{|Q|} } \les  |Q|^{1/(d-1)} \mu(Q) \lt|\log \lt(\frac{\mu(Q)}{h}\rt)\rt| \lt(\frac{\mu(Q)}{h}\rt)^{\frac{1}{d-1}}.$$
If instead $1\le p<d/(d-1)$, using the same proof as above but with the inclusion $W^{1,q}(Q) \subseteq L^\infty(Q)$ and letting $t \to 0$ gives the estimate 
$$ \W_Q\bra{ \mu + \frac{h}{|Q|}, \frac { \mu(Q)}{|Q|} + \frac{h}{|Q|} } \les |Q|^{p/d}  \mu(Q)\lt(\frac{    \mu(Q)}{h}\rt)^{p-1}.$$ 
\end{remark}

We close this section  with the following result easily adapted from \cite[Proposition 2.4]{ambrosio2022quadratic} which helps in particular to reduce the transport problem from H\"older to constant densities.

\begin{proposition}\label{prop:map-heat-semigroup}
For $d\ge 1$, $\alpha \in (0,1)$ and $\rho_0>0$, there exists $C = C(\rho_0, d, \alpha)>0$ such that the following holds: for any $\rho\in C^\alpha( (0,1)^d)$ with
\[ \int_{(0,1)^d} \rho = 1 \quad \text{and} \quad  \rho_0 \le \rho\le \rho_0^{-1},\]
there exists $T: (0,1)^d \to (0,1)^d$ such that $T_{\sharp} \rho = 1$, with
\[ \Lip T, \Lip T^{-1} \le  1 + C \nor{ \rho -1}_{C^\alpha}. \]
\end{proposition}

\subsection{A subadditivity lemma}

We will need a slight variant of the usual convergence results for subadditive functions, see e.g.\ \cite{steele1997probability, BoutMar}.

\begin{lemma}\label{lem:sub}
Let $\alpha, \beta, c>0$, $f: [1, \infty) \to [0, \infty)$ be continuous and such that the following holds: for every $\eta\in (0,1/2]$, there exists $C(\eta)>0$ such that, for every $m \in \N\setminus\cur{0}$ and $L \ge C(\eta)$,
\begin{equation}\label{eq:monotonicity} f(m L) \le f( L (1-\eta)) + c \eta^{\alpha} + C(\eta)L^{-\beta}.\end{equation}
Then $\lim_{L\to \infty} f(L) \in [0, \infty)$ exists.
\end{lemma}

\begin{proof}
We use the following fact: for any open interval $(a,b) \subseteq [0, \infty)$, the union $$\bigcup_{m=1} ^\infty (ma, mb) \supseteq (A, +\infty)$$
contains a half-line, for some $A>0$. Indeed, one has $(ma, mb) \cap ((m+1)a, (m+1)b)\neq \emptyset$ if $mb > (m+1)a$, which holds for every $m> a/(b-a)$. 

First, we show that $f$ is uniformly bounded. Let $\eta = 1/2$ and use the fact that both $f(L)$ and $L^{-\beta}$ are continuous for $L \in [1/2,2]$, hence bounded, so that by \eqref{eq:monotonicity}, for every  $m \ge 1$, $L \in [1,2]$,
\[ f(m L) \le \sup_{\ell \in [1,2]} \bra{f(\ell/2) + c2^{-\alpha} + C(1/2) \ell^{-\beta}} < \infty.\]
since $\bigcup_{m= 1}^\infty [m, 2m] = [1, \infty)$, it follows that $f$ is uniformly bounded on $[1, \infty)$. To show that the limit exists (and is finite) we argue that
\[ \limsup_{L \to \infty} f(L)  \le \liminf_{L \to \infty} f(L).\]
Given $\eps\ll 1$, let $\eta  = \eta(\eps) \in (0,1/2]$ such that $c\eta^{\alpha}=  \eps$ and $L_{\eps}>0$ such that $C(\eta) L_{\eps}^{-\beta}=\eps$, so that, for every $L \ge L_\eps$,
$$ C(\eta) L^{-\beta} \le  \eps.$$
Let then $L^* > \max \cur{ L_\eps, C(\eta)}$ be such that
\[ f(L^*) < \liminf_{L \to \infty} f(L) + \eps.\]
By continuity of $f$, there exists $a < L^* <b$ with  $a>\max \cur{ L_\eps, C(\eta)}$ such that the same inequality holds for $L \in (a,b)$. For every $m \ge 1$, and $L \in (a/(1-\eta), b/(1-\eta))$, we have $L \ge \max\cur{L_\eps, C(\eta)}$ and $L(1-\eta) \in (a,b)$, hence using  \eqref{eq:monotonicity} we obtain
\[ f(m L) \le f(L (1-\eta)) + c \eta^{\alpha} + C(\eta) L^{-\beta}  \le  \liminf_{L \to \infty} f(L) + 3 \eps .\]
Using that $\cup_{m =1}^\infty (ma/(1-\eta), mb/(1-\eta))$ contains a half-line $(A,+\infty)$, it follows that
$$ \limsup_{L\to \infty}   f(L) \le \liminf_{L \to \infty} f(L) + 3 \eps,$$
and the thesis follows letting $\eps \to 0$.
\end{proof}
\subsection{Concentration inequalities}

We close this section by recalling some  standard concentration inequalities. Let us start with a general definition.
\begin{definition}\label{def:concen}
 We say that a random variable $X$ with $\EE\sqa{X}=h$ satisfies (algebraic) concentration if for every $q\ge 1$ there exists $C(q)\in (0, \infty)$ such that 
 \[
  \EE\sqa{|X-h|^q}\le C(q) |h|^{\frac{q}{2}}.
 \]

\end{definition}
We then have 

\begin{lemma} Poisson, binomial and hypergeometric random variables satisfy concentration. More precisely, if :
\begin{itemize}
\item[i)] 
 $N$ is a Poisson random variable with parameter $n \ge 1$ then, 
 for every $q\ge 1$,
\begin{equation}\label{eq:momentPoi}
 \EE\lt[|N-n|^q\rt]\les_q n^{\frac{q}{2}}.
\end{equation}
Hence, for every $\gamma \in (0,1)$, 
\begin{equation}\label{eq:density-bound-below-Poi}
\PP \bra{  N < \gamma n \quad \text{or} \quad N > (1+\gamma) n} 
\les_{q,\gamma} (1-\gamma)^{-2q} n^{-q}.
\end{equation}
\item[ii)]  $B$ is a binomial random variable  with parameters $n$ and $p\in (0,1)$ (so that $\EE\sqa{B} = np$) then, for every $q \ge 1$,
\begin{equation}\label{eq:binomial-concentration} 
\EE\lt[|B-np|^q\rt]\les_q n^{\frac{q}{2}}.
\end{equation}
\item[iii)] $H$  is  a hypergeometric random variables counting the number of red marbles extracted in $z$ draws without replacement from an urn containing  $u$ marbles, $r$ of which are red (so that $\EE\sqa{H} = z r/u$) then, for every $q \ge 1$,
\begin{equation}\label{eq:concentration-hyper} \EE\sqa{  \abs{ H - zr/u}^q } \les_q r^{\frac{q}{2}}.\end{equation}
\end{itemize}

\end{lemma}

\begin{proof}
We only prove concentration in the  hypergeometric case, since it is classical for both Poisson and binomial random variables. We may assume that $r \ge 1$, otherwise there is nothing to prove since $H = \EE\sqa{H} = 0$. From \cite[Theorem 1]{hush2005concentration}, we have, for $\lambda \ge 2$,
$$ \PP\bra{  \abs{ H - \EE\sqa{ H}} \ge \lambda } \le 2 \exp\bra{ -\alpha \lambda^2 },$$
where
$$ \alpha = \min\cur{ \frac{1}{z+1} + \frac{1}{u-z+1},  \frac{1}{r+1} + \frac{1}{u-r+1} } \ge \frac{u+2 }{(r+1)(u-r+1)} \ge \frac{1}{r+1}.$$
As usual, writing 
$$ \EE\sqa{ \abs{ H - \EE\sqa{ H}}^q } = \int_0^\infty \PP\bra{  \abs{ H - \EE\sqa{ H}} \ge \lambda } p\lambda^{p-1} d \lambda, $$
yields the  bound
$$ \EE\sqa{  \abs{ H - \EE\sqa{ H}}^q  } \les_q 1+ \alpha^{-\frac{q}{2}} \les_q 1 + (r+1)^{\frac{q}{2}},$$
which  is bounded from above by $r^{q/2}$, since $r \ge 1$.
\end{proof}

\section{Combinatorial optimization problems over bipartite graphs}\label{sec:cop}
\subsection{Graphs}
Although we are interested in random combinatorial optimization over Euclidean bipartite graphs, it is useful to recall some general terminology.
A (finite, undirected) graph $G = (V, E)$ is defined by a finite set $V = V_G$ of vertices (or nodes) and a set of edges $E = E_G$, which is a collection of unordered pairs $e = \cur{x,y} \subseteq V$ with $x \neq y$.
A graph $G'$ is a subgraph of $G$ and we write $G' \subseteq G$, if $V_{G'} \subseteq V_G$ and $E_{G'}\subseteq E_G$. The induced subgraph over a subset of vertices $V' \subseteq V_G$ is defined as the subgraph $G'$ with $V_{G'} = V'$ and all the edges from $E_G$ connecting vertices in $V'$. It will be useful to denote by $\emptyset$ the empty graph, i.e., $V = \emptyset$, $E = \emptyset$, which is a subgraph of any graph $G$.

Given a vertex $x \in V$, its neighborhood in $G$ is the set 
\[ \mathcal{N}_G(x) = \cur{ y \in V\, : \cur{x,y} \in E}.\]
The degree of $x$ in $G$, $\deg_G(x)$, is the number of elements in $\mathcal{N}_G(x)$. Given $\kappa \in \mathbb{N}$, a graph $G$ is $\kappa$-regular if $\deg_G(x) = \kappa$ for every $x \in V_G$. We say that a subgraph $G' \subseteq G$ spans $V_G$ if $V_{G'} = V_G$ and $\cN_{G'}(x) \neq \emptyset$ for ever $x \in V_{G'}$. We say that two subgraphs $G_1$, $G_2$ of $G$ are disjoint if $V_{G_1} \cap V_{G_2} =\emptyset$. A graph $G$ is connected if it cannot be decomposed as the union of two disjoint subgraphs $G = G_1 \cup G_2$, i.e., $V_G = V_{G_1} \cup V_{G_2}$ with both $V_{G_1}$, $V_{G_2} \neq \emptyset$,  $V_{G_1} \cap V_{G_2} = \emptyset$ and $E_{G}  =E_{G_1} \cup E_{G_2}$. Given $\kappa \in \mathbb{N}$, $\kappa \ge 1$, we say that a graph $G$ is $\kappa$-connected if any subgraph $G' \subseteq G$ obtained by removing from $G$  $(\kappa-1)$-edges is still connected. A cycle is a connected $2$-regular graph, a tree is a connected graph which contains no cycles as subgraphs.

Given two graphs $G_1$, $G_2$ and an injective function  $\sigma: V_{G_1} \to V_{G_2}$, we let $\sigma(E_1) = \cur{ \cur{\sigma(x), \sigma(y)} \, : \,  \cur{x,y} \in E_{G_1}}$. If $\sigma(E_1) \subseteq E_2$, then we say that $G_1$ embeds into $G_2$ via $\sigma$. If $\sigma$ is bijective and $\sigma(E_1) = E_2$, then we say that $G_1$ is isomorphic to $G_2$ via $\sigma$.

A graph $G$ is complete if $E_G$ consists of all the pairs $\cur{x,y} \subseteq V$ with $x \neq y$. The complete graph over $V = [n]$ is commonly denoted by $\cK_n$. Any complete graph $G$ with $n$ vertices is isomorphic to $\cK_n$.  We say that the  graph $G$ is bipartite over a partition $V = X \cup Y$ (i.e., $X \cap Y = \emptyset$), if every $e\in E$ can be written as $e= \cur{x,y}$ with $x\in X$, $y \in Y$. A graph is complete bipartite if it is bipartite over a partition $V = X \cup Y$ and every pair $\cur{x,y}$ with $x\in X$, $y \in Y$ is an edge. For any $n, m \in \mathbb{N}$, any two complete bipartite graphs with $X$ having $n$ elements and $Y$ having $m$ elements are isomorphic. To fix a representative, we define $\cK_{n,m}$ as the complete bipartite graph over the vertex set $V = [n]_1 \cup [m]_2$.
%
%

%
%
 

We introduce a weight function on edges  $w: E \to [0, \infty)$, $w(e) = w(x,y)$. The total weight of $G$ is then
\[  w(G) = \sum_{e \in E} w(e).\]
A subgraph $G' \subseteq G$ of a weighted graph is always intended with the restriction of $w$ on $E'$. Notice that for the empty graph $\emptyset \subseteq G$ we have $w(\emptyset) = 0$.

%
We are interested in geometric realizations of graphs, where vertices are in correspondence with points in a metric space $(\Omega, \dist)$,  and the weight function is a power of the distance between the corresponding points, with a fixed exponent $p>0$. Since we consider only complete and complete bipartite graphs, we introduce the following notation. Given $\x = (x_i)_{i=1}^n \subseteq \Omega$, we let $\cK(\x)$ be the complete graph $\cK_n$ endowed with the weight function $w(i,j) =  \dist(x_i, x_j)^p$. Similarly, given $\x = (x_i)_{i=1}^n$, $\y = (y_j)_{j=1}^m \subseteq \Omega$, we let $\cK(\x, \y)$ denote the complete bipartite graph $\cK_{n,m}$ endowed with the weight function $w((1,i), (2,j)) = \dist(x_i, y_j)^p$. Notice that the points in $\x$ and $\y$ may not be all distinct, but this will in fact occur with probability zero. If all the points are distinct, then we can and will identify the vertex set directly with the set of points $\x$ for $\cK(\x)$, and with the set of points in $\x \cup \y$ for $\cK(\x, \y)$. With this convention, if $\x = \x^0 \cup \x^1$, $\y = \y ^0 \cup \y^1$, then both $\cK(\x^0, \y^0)$ and $\cK(\x^1, \y^1)$ are naturally seen as subgraphs of $\cK(\x, \y)$. 

%
\subsection{Combinatorial problems}
A combinatorial optimization problem $\mathsf{P}$ on weighted graphs  is informally defined by prescribing, for every graph $G$, a set of subgraphs $G' \subseteq G$, also called feasible solutions $\mathcal{F}_{G}$, and, after introducing a weight $w$, by minimizing $w(G')$ over all $G'\in \mathcal{F}_G$. 

Our aim is to study problems on random geometric realizations of complete bipartite graphs $\cK_{n,n}$, thus it is sufficient to define a combinatorial optimization problem over complete bipartite graphs as a collection of feasible solutions $\mathsf{P} = ( \mathcal{F}_{n,n})_{ n \in \mathbb{N}}$, with $\cF_{n,n}$ being the feasible solutions on $\cK_{n,n}$. We will mostly consider problems $\pP$ that satisfy the following assumptions:
\begin{enumerate}[label=A\arabic*,  series=A]
\item \label{as:isomorphism} \emph{(isomorphism)} if $\sigma$ is any isomorphism of $\cK_{n,n}$ into itself and $G \in \cF_{n,n}$, then $\sigma(G) = (\sigma(V_G), \sigma(E_G)) \in \cF_{n,n}$;
\item \label{as:spanning} \emph{(spanning)} for every $n \in \mathbb{N}$, $\cF_{n,n}$ is not empty and there exists $\cspan >0$ such that, for $n < \cspan$, $\cF_{n,n} = \cur{ \emptyset}$ while for $n \ge \cspan$, every $G \in \cF_{n,n}$ spans $\cK_{n,n}$;\
\item \label{as:bddegree} \emph{(bounded degree)} there exists $\cbdeg >0$ such that, for every  $n \in \N$ and every feasible solution $G \in \cF_{n,n}$, one has $\deg_G(x) \le \cbdeg$ for every $x \in G$.
\end{enumerate}

Given $\mathsf{P} = ( \mathcal{F}_{n,n})_{n \in \mathbb{N}}$, we canonically extend it to graphs $\cK_{n,m}$, with $n \neq m$, defining $\cF_{n,m}$ as the collection of all graphs $\sigma(G)$ where $G\in \cF_{z,z}$, $z = \min\cur{n,m}$ and $\sigma$ is an isomorphism of $\cK_{n,m}$ into itself.

In the geometric setting, i.e., when $\cK_{n,m}$ is mapped into $\cK(\x,\y)$ with $\x = (x_i)_{i=1}^n$, $\y = (y_j)_{j=1}^m\subseteq \Omega$, with $(\Omega, \dist)$ metric space, we introduce the following notation for the cost of a problem $\pP$:
$$ \CPp (\x, \y) = \min_{G \in \cF_{n,m}} \sum_{\cur{(1,i),(2,j)}\in E_G} \dist(x_i,y_j)^p .$$
Recalling the definition of $\cF_{n,m}$ if $n \neq m$, we also have the identity
\begin{equation}\label{eq:min-subgraphs}  \CPp (\x, \y) =  \min_{\substack{\x' \subseteq \x,  \y' \subseteq \y \\ |\x'| = |\y'| = \min\cur{|\x|, |\y|}}} \CPp (\x', \y').\end{equation}

\begin{remark}
Assumption \ref{as:spanning} ensures that, if $\min\cur{|\x|, |\y|} < \cspan$,  then $\CPp(\x, \y) = 0$.
\end{remark}

\begin{remark}
If $(\Omega', \dist')$ is a metric space and $S: \Omega \to \Omega'$ is Lipschitz, i.e., for some constant $\Lip S$ one has $\dist'(S(x), S(y)) \le (\Lip S) \dist(x, y)$ for every $x$, $y\in \Omega$, then writing $S(\x) = (S(x_i))_{i=1}^n$, $S(\y) = (S(y_j))_{j=1}^m$, we clearly have the inequality
\begin{equation}\label{eq:lipschitz-bound-deterministic} \CPp(  S(\x), S(\y) ) \le (\Lip S)^p \CPp(  \x, \y ).
\end{equation}
\end{remark}

\begin{remark}
Similar definitions and assumptions may be given in the non-bipartite case, thus defining combinatorial optimization problems $\mathsf{P} = ( \mathcal{F}_{n})_{n \in \mathbb{N}}$ over complete graphs, as a collection of feasible solutions $\mathcal{F}_{n}$ over the complete graph $\cK_{n}$. 
\end{remark}

\subsection{Examples}

Let us introduce some fundamental examples of these problems. 

\subsubsection*{Assignment problem} 
The minimum weight bipartite matching problem, also called assignment problem, is defined letting $\cF_{n,n}$ be the set of perfect matchings in $\cK_{n,n}$, i.e., spanning subgraphs induced by a collection of edges which have no vertex in common (if $n=0$ we simply let $\cF_{n,n} = \cur{\emptyset}$). Feasible solutions are in correspondence with permutations $\sigma$ over $[n]$, letting 
$$ E_{\sigma} = \cur{ \cur{(1,i), (2,\sigma(i))} \, : \, i \in [n] }.$$
When $n \neq m$, e.g.\ $n \le m$, the same correspondence holds with the set of injective maps $\sigma:[n] \to [m]$. Therefore, given a weight $w$ on $\cK_{n,m}$, the cost of the assignment problem is
\[\min_{\sigma} \sum_{i=1}^n w\bra{ (1,i), (2,\sigma(i))}.\]
 In the geometric case, i.e., on the weighted graph $\cK(\x, \y)$ with $\x = (x_i)_{i=1}^n$, $\y=(y_j)_{j=1}^m \subseteq \Omega$ and $w\bra{ (1,i), (2,j)}=\dist(x_i, y_{j} )^p$, this expression becomes 
\[ \Mp(\x, \y) = \min_{\sigma} \sum_{i=1}^n \dist(x_i, y_{\sigma(i)} )^p.\]
If $n>m$, then one simply exchanges the roles of $n$ and $m$.

\begin{remark}
If $n=m$, Birkhoff's theorem ensures equivalence between the bipartite matching problem and the optimal transport between the associated empirical measures $\mu^\x = \sum_{i=1}^n \delta_{x_i}$, $\mu^\y = \sum_{j=1}^n \delta_{y_j}$, i.e.,
\begin{equation}\label{eq:wass-equals-matching} \Mp(\x, \y) = \W( \mu^\x, \nu^\y). \end{equation}
Therefore, using the triangle inequality \eqref{eq:triangle}, we can bound from above as follows:
\begin{equation}
\label{eq:matching-below-wasserstein}
\Mp(\x, \y) \les \W\bra{ \mu^\x,  n \lambda} + \W\bra{ \mu^\y,  n \lambda},
\end{equation}
for every probability measure $\lambda$ on $\R^d$.
\end{remark}

\subsubsection*{Travelling salesperson problem}
The travelling salesperson problem (TSP) is usually defined on a general graph by prescribing as feasible solutions the cycles visiting each vertex exactly once (also called Hamiltonian cycles). In the complete bipartite case $\cK_{n,n}$, such cycles exist for every $n \ge 2$, and assumptions \ref{as:isomorphism},  \ref{as:spanning} and \ref{as:bddegree} are also clearly satisfied (letting $\cF_{n,n} = \cur{\emptyset}$ if $n\in \cur{0,1}$). Similarly as in the case of the assignment problem, feasible solutions are in this case in correspondence with pairs of permutations $\sigma$, $\tau$ over $[n]$, letting 
\begin{equation}\label{eq:e-sigma-tau} E_{\sigma, \tau} = \cur{ \cur{(1,\sigma(i)), (2,\tau(i))}, \cur{(1,\sigma(i), (2,\tau(i+1)) )} \, : \, i \in \cur{1, \ldots, n}},\end{equation}
where we conventionally let $\tau(n+1) = \tau(1)$ (we will always use summation $\operatorname{mod} n$ in such cases). In words, $\sigma$ and $\tau$ prescribe the order at which the vertices are visited by the cycle. When $n \neq m$, e.g.\ $n \le m$, the same correspondence holds with injective maps $\sigma$, $\tau$ from $[n]$ into $[m]$.

%

Therefore, given a weight $w$ on $\cK_{n,m}$, the cost of the TSP reads
\[\min_{\sigma, \tau} \sum_{i=1}^n w\bra{(1,\sigma(i)), (2,\tau(i))} + w\bra{(1,\sigma(i)), (2,\tau(i+1)) }.\]
 In the geometric case, i.e., on the weighted graph $\cK(\x, \y)$ with $\x = (x_i)_{i=1}^n$, $\y=(y_j)_{j=1}^m \subseteq \Omega$, this becomes 
\[ \C_{\mathsf{TSP}}^p(\x, \y) = \min_{\sigma, \tau} \sum_{i=1}^n \dist(x_{\sigma(i)}, y_{\tau(i)} )^p + \dist(x_{\sigma(i)},  y_{\tau(i+1)} )^p.\]
 If $n>m$, then one simply exchanges the roles of $n$ and $m$. 

The non-bipartite version of the TSP , i.e., on $\cK_n$, feasible solutions to the TSP are in correspondence with permutations $\sigma$ over $[n]$, letting 
$$ E_{\sigma} =  \cur{ \cur{\sigma(i), \sigma(i+1)} \, : \, i \in [n]}.$$
In the geometric case $\x = (x_i)_{i=1}^n \subseteq \Omega$, it becomes
\[ \C_{\mathsf{TSP}}^p(\x) =  \min_{\sigma} \sum_{i=1}^n \dist(x_{\sigma(i)}, x_{\sigma(i+1)} )^p. \]
%



\subsubsection*{Connected $\kappa$-factor problem}
The TSP can be generalized in many directions. For example, since a cycle is a connected graph such that every vertex has degree $2$, i.e., it is $2$-regular, we may instead define as feasible solutions $\kappa$-regular spanning connected subgraphs, for a fixed $\kappa \in \mathbb{N}$, $\kappa \ge 2$. This defines a non-empty set of feasible solutions $\cF_{n,n}$ over $\cK_{n,n}$ if $n \ge \kappa$ (otherwise we let $\cF_{n,n} = \cur{\emptyset}$) and assumptions \ref{as:isomorphism},  \ref{as:spanning} and \ref{as:bddegree} are easily seen to be satisfied. 
%
We refer to such problem as the (minimum weight) connected $\kappa$-factor problem. A  simpler variant is to require that feasible solutions are  $\kappa$-regular but not necessarily connected: this is simply known as (minimum weight) $\kappa$-factor problem. Let us notice that, for $\kappa=1$, this reduces to the assignment problem.  

Back to the the connected $\kappa$-factor problem, a simple fact worth noticing, that we will use below, is that any connected $\kappa$-regular bipartite graph $G$ is $2$-connected, i.e., it remains connected even after removing a single edge. Assume that $V_G = X \cup Y$, with $X \cap Y = \emptyset$  and by contradiction let $x\in X$, $y \in Y$ be such that $\cur{x,y} \in E_G$ and the subgraph $G' \subseteq G$ with edge set $E_{G'} = E_G \setminus \cur{x,y}$ is not connected: there are two disjoint subgraphs $G_1'$, $G_2'$ with  $x \in V_{G_1'}$, $y \in V_{G_2'}$ with $G' = G_1' \cup G_2'$. All the vertices in $G_1'$ have degree $\kappa$, except for $x$, whose degree is $\kappa-1$. However, if we let $n_X = |V_{G_1'} \cap X|$, $n_Y = |V_{G_1'} \cap Y|$, then using the fact that the graph $G_1'$ is bipartite we can count the number of edges as the sum of the degrees of the vertices in  $V_{G_1'} \cap X$ or equivalently of those in $V_{G_1'} \cap Y$, which leads to the identity $\kappa n_X - 1 = \kappa n_Y$, from which $\kappa(n_X - n_Y ) = 1$, which gives a contradiction.
 

\subsubsection*{$\kappa$-bounded degree minimum spanning tree}

The minimum weight spanning trees (MST) problem is defined by letting feasible solutions be all spanning subgraphs that are trees, i.e.,  connect and acyclic, whose existence on any given connected graph is guaranteed by standard algorithms. This problem however may not have uniformly bounded degree, thus assumption \ref{as:bddegree} may not be satisfied. Therefore, we restrict the set of feasible solutions to spanning trees over $\cK_{n,n}$ such that that each vertex degree is less than or equal to some fixed $\kappa \ge 2$ (letting $\cF_{0,0}=\cur{\emptyset}$). This problem, known as the $\kappa$-bounded degree minimum spanning tree ($\kappa$-MST), satisfies assumptions \ref{as:isomorphism}, \ref{as:spanning} and \ref{as:bddegree}: notice in particular that removing any edge from a Hamiltonian cycle, i.e., a feasible solution for the TSP, gives a $2$-bounded degree minimum spanning tree. 

We remark here that the $\kappa$-MST problem may be also directly defined over graphs $\cK_{n,m}$, with $n \neq m$, with a non trivial set of feasible solutions (provided that $|n-m|$ is not too large). However, also in this case we follow our the general convention, so that if $n \neq m$, the set $\cF_{n,m}$ does not contain spanning trees of $\cK_{n,m}$ but only spanning trees over subgraphs isomorphic to $\cK_{z,z}$ with $z = \min\cur{n,m}$.

A simple fact that we will use below is that any $G \in \cF_{n,n}$ contains at least one leaf (i.e., a vertex with degree $1$) in $[n]_1$ and one in $[n]_2$. This is because more generally any spanning tree over $\cK_{n,n}$ contains at least one leaf in $[n]_1$ and one in $[n]_2$.  Indeed, assume by contradiction that there are no leaves in $[n]_1$. Then, since the tree spans, all the vertices in $[n]_1$ must have degree at least $2$ (the graph is connected, hence every vertex has at least degree $1$) and since no edges connect pairs of vertices in $[n]_1$, these are all distinct, hence the tree contains at least $2n$ edges, which contradicts the well-known fact that any tree (not necessarily bipartite) over $2n$ vertices must have $2n-1$ edges. 


In order to perform our analysis, we introduce two further assumptions that we discuss in the following subsections.



\subsection{Local merging}

Our analysis relies on a key subadditivity inequality, that ultimately follows by a stability assumption with respect to local merging operations, besides assumptions \ref{as:isomorphism} and \ref{as:bddegree}. Let us give the following general definition.

\begin{definition}[gluing] Given a graph $G$ and two disjoint subgraphs $G_1$, $G_2 \subseteq G$, we say that $G' \subseteq G$ is obtained by gluing at $x_1 \in V_{G_1}$, $x_2 \in V_{G_2}$ if $V_{G'} = V_{G_1}\cup V_{G_2}$, 
\[ (E_{G_1} \cup E_{G_2}) \setminus E_{G'} \subseteq \cN_{G_1}(x_1) \cup \cN_{G_2}(x_2)\]
and 
\[ E_{G'} \setminus (E_{G_1} \cup E_{G_2}) \subseteq \cur{ \cur{x_1, y} : y \in \cN_{G_2}(x_2)} \cup \cur{ \cur{x_2, y} : y \in \cN_{G_2}(x_1)}.\] 
\end{definition}

In words, gluing at $x_1$, $x_2$ means that the two subgraphs are joined by (possibly) removing and adding edges connecting $x_2$ to vertices from the neighborhood of $x_1$ in $G_1$, and similarly $x_1$ to vertices from the the neighborhood of $x_2$ in $G_2$. In particular, we have that $\mathcal{N}_{G'}(x) = \mathcal{N}_{G_1}(x)$ for every $x\in V_{G_1} \setminus \bra{ \cN_{G_1}(x_1)\cup \cur{x_1}}$, and similarly $\mathcal{N}_{G'}(x) = \mathcal{N}_{G_2}(x)$ for every $x\in V_{G_2} \setminus \bra{ \cN_{G_2}(x_2)\cup \cur{x_2}}$.

Back to combinatorial optimization problems over bipartite graphs, our assumption is, loosely speaking, that any two (non empty) feasible solutions $G \in \cF_{n,n}$,  $G' \in \cF_{n',n'}$, can be glued together yielding a feasible solution $G \in \cF_{n+n', n+n'}$. In fact,  we also allow adding up to $\mathsf{c}$ edges, but only connecting vertices of $G$, where $\mathsf{c} \in \mathbb{N}$ is a constant (depending only on the problem $\pP$).  Before  giving a precise formulation of the assumption, we notice that $G$ and $G'$ are in general not disjoint: what we mean is that $G'$ must be suitably ``translated''. Precisely, given $n \in \mathbb{N}$, we introduce the map
$$ \tau_n: V_{G'} \to [n+n']_1 \cup [n+n']_2$$
defined as
$$ \tau\bra{ (1,i)} = (1, n+i), \quad \tau\bra{ (2,j)} = (2,n+j),$$
so that $G$, $\tau( G') \subseteq \cK_{n+n', n+n'}$ are disjoint.


We consider therefore combinatorial optimization problems $\pP$ over bipartite graphs which satisfy the following assumption:
\begin{enumerate}[label=A\arabic*,  series=A, resume]
\item\label{ass:local-merging} \emph{(local merging)} there exists $\cmerge \ge 0$ such that, for every $n$, $n' \in \mathbb{N}$, and $G\in \mathcal{F}_{n,n}$, $G'\in \mathcal{F}_{n',n'}$ with both $G\neq \emptyset$ and $G' \neq \emptyset$, one can find $G'' \in \mathcal{F}_{n+n', n+n'}$ obtained by gluing $G$ and $\tau (G')$ at the vertices $(1,1)$, $(1,n+1)$
and possibly adding up to $\cmerge$ edges from those of $\cK_{n,n}$.
%
%
\end{enumerate}

%

The reason why we also allow  up to $\cmerge$ additional edges is to include some problems where connectedness may be destroyed by gluing, such as the $\kappa$-MST. This should be compared  with the \emph{merging} assumption \cite[(A4)]{BaBo}, where a bounded number of edges from  the whole $\cK_{n+n', n+n'}$ instead is allowed to be added to the union $G \cup \tau(G')$ (with our notation). Notice however that, in our case, since the extra edges are from $\cK_{n,n}$ it remains true that
\begin{equation}\label{eq:neighbour-same} \text{ $\mathcal{N}_{G''}(x) = \mathcal{N}_{\tau (G')}(x)$, for every $x\in V_{\tau (G')} \setminus \bra{ \cN_{\tau (G')}((1,n+1))\cup \cur{(1, n+1)}}$,}\end{equation}
which is a key condition that we use below.

All the problems described in the previous section satisfy \ref{ass:local-merging}.

\begin{lemma}
The TSP, the connected $\kappa$-factor problem (as well as the non connected one) and the $\kappa$-MST over complete bipartite graphs satisfy assumption \ref{ass:local-merging}.
\end{lemma}
\begin{proof}
Let $G \in \cF_{n,n}$, $G'\in\cF_{n', n'}$ be both non empty. Then (e.g.\ by assumption \ref{as:spanning}) $\deg_{G}(1,1) \ge 1$ but also $\deg_{\tau (G')}(1, n+1) \ge 1$. The basic idea is to pick $y \in  \cN_{G}(1,1)$, $y' \in \cN_{G'}(1,n+1)$, remove the edges $\cur{(1,1), y}$, $\cur{(1,n+1), y'}$ and add instead $\cur{ (1,1), y'}$, $\cur{(1,n+1), y'}$. This operation does not change the vertex degrees, in particular at $(1,1)$ and $(1,n+1)$.

For the TSP and more generally the connected $\kappa$-factor problem, the resulting graph $G''$ is connected, because after removing a single edge, both graphs $G$ and $\tau (G')$ are still connected, and adding the new edges has the effect of connecting the two graphs (hence in this case $\cmerge = 0$).

For the $\kappa$-bounded degree MST, we use the fact that the tree $G \in \cF_{n,n}$ must have at least one leaf in the set of $[n]_1$ and one in the set $[n]_2$. Therefore, we obtain a connected tree (with degree bounded by $\kappa$) if we add also one edge connecting two such leaves (hence is this case $\cmerge = 1$).
\end{proof}
\subsection{Subadditivity inequality}
Using all the assumptions introduced so far, in particular \ref{ass:local-merging}, we establish a fundamental subadditivity inequality.
%
%

 \begin{proposition}[Approximate subadditivity]\label{prop:partition}
 Let $\pP$ be a combinatorial optimization problem over bipartite graphs satisfying assumptions \ref{as:isomorphism}, \ref{as:spanning}, \ref{as:bddegree}  and \ref{ass:local-merging}. 
 
 For  a metric space $(\Omega,\dist)$ and a finite partition $\Omega = \cup_{k=1}^K \Omega_k$,  $K \in \mathbb{N}$, 
\begin{enumerate}[label=\emph{\roman*})]
 \item \label{prop-partition-x0-assumption} let $\x^0$, $\y^0 \subseteq \Omega$ be such that $\min\cur{|\x^0|, |\y^0|} \ge \max\cur{\cspan, K}$,
 \item  for every $k =1, \ldots, K$, let $\x^k$, $\y^k \subseteq \Omega_k$ with $|\x^k| = |\y^k| =n_k$, with either $n_k \ge \cspan$ or $n_k = 0$ (i.e., both families are empty) ,
 \item let $\z=(z_k)_{k=1}^K$ with $z_k \in \Omega_k$, for every $k=1,\ldots, K$.
\end{enumerate}
Then, the following inequality holds:

 \begin{equation}\label{eq:sub}  \CPp\bra{ \x^0 \cup \bigcup_{k=1}^K \x^k, \y^0 \cup  \bigcup_{k=1}^K \y^k }  - \sum_{k=1}^K \CPp (\x^k, \y^k) 
  \lesssim  \CPp(\x^0, \y^0)+  \Mp(\z, \x^0)+  \sum_{k=1}^K   \diam(\Omega_k)^p.
 \end{equation}
 The implicit constant depends only upon $p$, $\cspan$, $\cbdeg$ and $\cmerge$ (in particular not on $K$).
 \end{proposition}
 
\begin{remark}\label{rem:remove-z-points}
The role played by the points $\z$ is quite marginal, and indeed if  $\x^0(\Omega_k) >0$ for every $k$, then by choosing $z_k \in \x^0_{\Omega_k}$, the  term  $\Mp(\z, \x^0)$ vanishes. 
\end{remark}

 \begin{proof}
 Recalling  \eqref{eq:min-subgraphs}, up to replacing $\x^0$, $\y^0$ with subsets $\x'$, $\y'$ with $|\x'| = |\y'| = \min\cur{|\x^0|, |\y^0|}$, we may also assume that $|\x^0| = |\y^0|$. For every $k=1, \ldots, K$ let $G_k \subseteq \cK(\x^k, \y^k)$ be a minimizer for $\pP$. If $n_k=0$, then $G_k =\emptyset$. Otherwise, $n_k \ge \cspan$, and by assumption \ref{as:spanning} it is in particular non-empty and using Markov inequality, we can choose $x^k \in \x^k$ such that 
 $$  \sum_{y \in \cN_{G_k} (x^k)} \dist(x^k, y)^p \le \frac{ 4\CPp(\x^k, \y^k)}{|\x^k|}\les \diam(\Omega_k)^p.$$
For the last estimate we used that $\deg_{x^k}(G_k) \le \cbdeg$.
 Similarly, let $G_0\subseteq \cK(\x^0, \y^0)$ be a (also non-empty) minimizer for $\CPp(\x^0, \y^0)$ and let $\sigma :\{1, \ldots, K\}\to \{ 1, \ldots, |\x^0|\}$ be an optimal matching between $\z$ and $\x^0$. 
 
 We iteratively use  assumptions \ref{as:isomorphism} and \ref{ass:local-merging} to define feasible solutions
 $$ \tilde{G}_k \subseteq \cK\bra{ \x^0 \cup \bigcup_{i=1}^{k} \x^i, \y^0 \cup \bigcup_{i=1}^{k} \y^i}.$$
 We begin by letting $\tilde{G}_0 = G_0$. For $k =1, \ldots, K$, having already defined $\tilde{G}_{k-1}$, if $n_k = 0$, then we simply let $\tilde{G}_k = \tilde{G}_{k-1}$. Otherwise, we obtain a feasible solution $\tilde{G}_k$ by gluing $G_{k}$ with $\tilde{G}_{k-1}$ at the vertices $x^k$, $x_{k}^0$ and adding up to $\cmerge$ edges from $\cK(\x^k, \y^k)$. The fact that we can glue at any such pair of vertices is due to assumption \ref{as:isomorphism}: up to isomorphisms we can assume that $x^k$ corresponds to the abstract graph vertex $(1,1)$ and that $x_{k}^0$ to $(1, n_k+1)$.
 %
 
 This construction gives the following inequality between the graph weights, if $n_k \neq 0$:
 \begin{equation}\label{eq:gluing-error} \begin{split} w( \tilde{G}_k ) - w( \tilde{G}_{k-1}) - w(G_k) \le  & \cmerge \diam(\Omega_k)^p \\
 &  + \sum_{y \in \cN_{G_k}(x^k) } \dist(x_{\sigma(k)}^0, y)^p +  \sum_{y \in \cN_{\tilde{G}_{k-1}}(x_{\sigma(k)}^0) } \dist(x^k, y)^p,
 \end{split} \end{equation}
 while if $n_k = 0$, we simply have $w(\tilde{G}_k) = w(\tilde{G}_{k-1})$. We bound from above the last two terms in \eqref{eq:gluing-error} as follows: first,
 \[\begin{split} \sum_{y \in \cN_{G_k}(x^k) } \dist(x_{\sigma(k)}^0, y)^p & \les \sum_{y \in \cN_{G_k}(x^k) } \dist( x_{\sigma(k)}^0, z_k)^p + \dist( z_k, x^k)^p + \dist(x^k, y)^p\\
 & \les \dist( x_{\sigma(k)}^0, z_k)^p  +  \dist( z_k, x^k)^p +\sum_{y \in \cN_{G_k}(x^k) } \dist(x^k, y)^p \\
 & \les   \dist( x_{\sigma(k)}^0, z_k)^p +  \diam(\Omega_k)^p,
 \end{split}\]
 where we used that $\deg_{x^k}(G_k) \le \cbdeg$. To bound the last term, we notice that each step in the construction we are locally merging at different points in $\x^0$: since no such points are adjacent because the graph is bipartite, using \eqref{eq:neighbour-same}  by induction yields
 \[ \cN_{\tilde{G}_{k-1}}(x_{\sigma(k)}^0)=\cN_{G_0}(x_{\sigma(k)}^0),\]
 which in particular contains at most $\cbdeg$ elements, since $G_0$ is feasible. Therefore,
 \[ \begin{split} \sum_{y \in  \cN_{\tilde{G}_{k-1}}(x_{\sigma(k)}^0)} \dist(x^k, y)^p & = \sum_{y \in  \cN_{G_0}(x_{\sigma(k)}^0)} \dist(x^k, y)^p \\
 & \les  \sum_{y \in  \cN_{G_0}(x_{\sigma(k)}^0)} \dist(x^k, z_k)^p +\dist(z_k, x_{\sigma(k)}^0)^p +  \dist(x_{\sigma(k)}^0, y)^p\\
 & \les  \diam(\Omega_k)^p +  \dist(z_k, x_{\sigma(k)}^0)^p + \sum_{y \in  \cN_{G_0}(x^0_{\sigma(k)})}\dist(x_{\sigma(k)}^0, y)^p.\end{split}\]
 Summing \eqref{eq:gluing-error} upon $k=1,\ldots, K$, we obtain \eqref{eq:sub} because
 $$ \sum_{k=1}^K \dist(z_k, x_{\sigma(k)}^0)^p =  \Mp(\z, \x^0)$$
 and, being all the points $x^0_{\sigma(k)}$ different,
 \[ \sum_{k=1}^K \sum_{y \in  \cN_{G_0}(x^0_{\sigma(k)})}\dist(x_{\sigma(k)}^0, y)^p \le \sum_{x\in \x^0} \sum_{y \in  \cN_{G_0}(x)}\dist(x, y)^p   =  \CPp(\x^0, \y^0). \qedhere \] 
 \end{proof}

\subsection{Growth/regularity}

The last assumption that we introduce for a combinatorial optimization problem $\pP$ over bipartite graphs is a general upper bound for the cost when specialized to a geometric graph in the Euclidean cube $(0,1)^d$:
\begin{enumerate}[label=A\arabic*,  series=A, resume]
\item\label{as:growth} (growth/regularity) There exists $\creg \ge 0$ such that, for every $\x, \y \subseteq (0,1)^d$, we have
\begin{equation} \label{eq:upper-bound-deterministic} \CPp(\x, \y) \le \creg \bra{ \min\cur{|\x|^{1-\frac{p}{d}}, |\y|^{1-\frac{p}{d}}}+ \Mp(\x,\y)}.\end{equation}
\end{enumerate}
\begin{remark}\label{remgrowth}
 Notice that if $\Omega\subset (0,1)^d$ then \eqref{eq:upper-bound-deterministic} applies in particular for $\x,\y \subseteq \Omega$. By scaling we obtain that for every bounded set $\Omega$ and every  $\x,\y \subseteq \Omega$,
  \[\CPp(\x, \y) \le \creg \bra{ \diam(\Omega)^p\min\cur{|\x|^{1-\frac{p}{d}}, |\y|^{1-\frac{p}{d}}}+ \Mp(\x,\y)}.\]
\end{remark}

Using \eqref{eq:min-subgraphs}, we obtain at once that in order to establish that a given problem $\pP$ satisfies \eqref{eq:upper-bound-deterministic} it is enough to consider the case where $\x$, $\y \subseteq (0,1)^d$ have the same number of elements.

Notice that this assumption seems slightly different with respect to the previous ones, as it explicitly refers to the cost for Euclidean realizations of the graph, instead of feasible solutions, and relies as well on the assignment problem. In fact, the constant $\creg$ depends upon the problem $\pP$ but also on the dimension $d$ and the exponent $p$, which however will be fixed in our derivations so we avoid to explicitly state it.

It is well known that quite general arguments, such as the space-filling curve heuristics \cite[Chapter 2]{steele1997probability}, lead to an upper bound in terms of $n^{1-p/d}$ for non-bipartite combinatorial optimization problems over $n$ points in a cube, under very mild assumptions, including those introduced above. Simple examples show that similar bounds cannot hold for their bipartite counterparts, which explains the second term in the right-hand side of \eqref{eq:upper-bound-deterministic}. 

To establish it in our examples we follow the strategy from \cite{capelli2018exact}, where  limit results for the random Euclidean bipartite TSP for $p=d=2$ were first obtained. 


%

\begin{lemma}
\label{lem:capelli}
The TSP, the connected $\kappa$-factor problem (as well as the non-connected one) and the $\kappa$-MST problems over complete bipartite graphs satisfy assumption \ref{as:growth} (with a constant $\creg$ depending on $\kappa$, $p$, $d$ only).
\end{lemma}

\begin{proof}
Let us first observe that the cost of the $\kappa$-MST problem is always bounded from above by the cost of the minimum weight connected $\kappa$-factor problem, since given any connected $\kappa$-factor, one can extract from it a MST whose degree at every vertex is then bounded by $\kappa$. Therefore it is sufficient to check that assumption \ref{as:growth} holds with $\pP$ being the connected $\kappa$-factor problem, for any $\kappa \ge 2$ (the case $\kappa=2$ being the TSP).

For $(\Omega, \dist)$ a general metric space and $\x, \y \subseteq \Omega$ we establish first the bound
\begin{equation}\label{eq:c-bip-c-mono} \CPp(\x, \y) \les \C_{\mathsf{TSP}}(\x) + \Mp(\x, \y).\end{equation}
Combining this with the fact that when $(\Omega,\dist)$ is the unit cube $(0,1)^d$ with the Euclidean distance,  $\C_{\mathsf{TSP}}(\x) \les |\x|^ {1-p/d}$ (a well-known fact, proved e.g.\ via space-filling curves) this would conclude the proof of \eqref{eq:upper-bound-deterministic}. \\
Assume without loss of generality that  $|\x| = |\y|  = n \ge \kappa$ and  let $\rho$  be a permutation over $[n]$ that induces an optimal assignment between $\x$ and $\y$.  
Consider then an optimizer for the TSP over $\cK(\x)$, which we also identify with a permutation $\sigma$ over $[n]$. We then define the feasible solution $G \in \cF_{n,n}$ for the connected $\kappa$-factor problem whose edge set is
$$ E_{G} = \cur{  \cur{ (1, \sigma(i)), (2, \rho(\sigma(i+\ell))} \, : \, i \in [n], \ell \in \cur{0,1,\ldots, \kappa-1}},$$
which generalizes $E_{\sigma, \tau}$ from  \eqref{eq:e-sigma-tau} with $\tau= \rho \circ \sigma$ in the $\kappa=2$ case, and as in \eqref{eq:e-sigma-tau} we use the summation $\operatorname{mod} n$, i.e., $i+ \ell  = i + \ell - n$ if $i+\ell>n$. Clearly, any vertex has degree $\kappa$ and the graph is connected, since $E_{G} \supset E_{\sigma, \tau}$.

%
%
In follows that
\[
 \CPp(\x, \y)  \le \sum_{i=1}^n \sum_{\ell = 0}^{\alpha-1}  \dist(x_{\sigma(i)}, y_{\rho( \sigma(i+\ell))})^p.\]
Using  the triangle inequality for every $i$ and $\ell$, we bound from above
\[\dist(x_{\sigma(i)}, y_{\rho( \sigma(i+\ell))})^p \les \sum_{j = 0}^{\ell-1} \dist(x_{\sigma(i+j)}, x_{\sigma(i+j-1)})^p +\dist ( x_{\sigma(i+\ell)},  y_{\rho( \sigma(i+\ell))})^p.  \]
 Summation upon $i$ (keeping $\ell$ fixed) gives
 \[  \sum_{j = 0}^{\ell-1} \sum_{i=1}^n \dist(x_{\sigma(i+j)}, x_{\sigma(i+j-1)})^p +  \sum_{i=1}^n \dist ( x_{\sigma(i+\ell)},  y_{\rho( \sigma(i+\ell))})^p 
\les \ell \TSP(\x) + \Mp(\x, \y),\]
hence, after summing upon $\ell = 0, \ldots, \kappa-1$, we obtain \eqref{eq:c-bip-c-mono}.
%
\end{proof}

\section{Convergence results for Poisson point processes}\label{sec:poisson}
\subsection{Point processes}

We define a point process on $\R^d$ as a random finite family of points $\cN = (X_i)_{i=1}^N \subseteq \R^d$, i.e. a  $N$-uple of random variables with values in $\R^d$, where the total number of points $N$ is also random and a.s.\ finite (if $N=0$, then $\cN = \emptyset$). 
We extend the notation for families of points to point processes (naturally defined for each realization of the random variables): for a process $\cN = \bra{X_i}_{i=1}^N$, write $\mu^\cN := \sum_{i=1}^N \delta_{X_i}$ and, given a Borel $\Omega \subseteq \R^d$, let $\cN(\Omega) = \mu^{\cN}(\Omega)$ be the (random) number of variables belonging to $\Omega$, while $\cN_{\Omega}$ denotes its restriction to $\Omega$, i.e., the collection of the variables such that $X_i \in \Omega$ (naturally re-indexed over $i=1, \ldots, \cN(\Omega)$, with the order inherited from the original process). Given two point processes $\cN = (X_i)_{i=1}^N$, $\cM = (Y_j)_{j=1}^M$, their  union is $\cN \cup \cM = (X_1, \ldots, X_N, Y_1, \ldots, Y_M)$.



Given a finite Borel measure $\lambda$ on $\R^d$, a Poisson point process $\cN^\lambda$ with intensity $\lambda$ can be constructed from a random collection of i.i.d.\ variables $(X_i)_{i=1}^\infty$ with common law $\lambda/\lambda(\R^d)$ and, after introducing a further independent Poisson variable $N^{\lambda}$ with mean $\lambda(\R^d)$, by considering only the first $N^{\lambda}$ variables, i.e.,
$$ \cN^\lambda := (X_i)_{i=1}^{N^{\lambda}}.$$ 
A key property of a Poisson point process (with intensity $\lambda$) is that, given any countable Borel partition $\R^d = \cup_{k} \Omega_k$, the variables $(\cN^\lambda(\Omega_k))_k$ are independent Poisson variables, each with mean $\lambda(\Omega_k)$ and, conditionally upon their value, the points in each $\Omega_k$ are i.i.d.\ variables with common probability law $\lambda\restr \Omega_k / \lambda(\Omega_k)$. This property can be summarized by stating that the restrictions $(\cN^\lambda_{\Omega_k})_{k}$ are independent Poisson point processes, with each $\cN^\lambda_{\Omega_k}$ having intensity given by the restriction $\lambda\restr \Omega_k$. 

We will use the well-known thinning operation, which apparently dates back to R\'enyi \cite{renyi1956characterization}, to split a Poisson point process $\cN^\lambda$ with intensity $\lambda$ into two independent Poisson point processes, each containing approximatively a given fraction of points: for $\eta \in [0,1]$, the $\eta$-thinning of a Poisson point process $\cN^{\lambda} = (X_i)_{i=1}^{N^\lambda}$ defines the two processes
$$\cN^{(1-\eta)\lambda} = (X_i)_{i=1}^{N^{(1-\eta)\lambda}} \quad \text{and} \quad \cN^{\eta \rho} = (X_{N^{(1-\eta)\lambda}+i})_{i=1}^{N^\lambda-N^{(1-\eta)\lambda}},$$
where $N^{(1-\eta)\lambda} = \sum_{i=1}^{N} Z_i$ is defined using a further sequence of i.i.d.\ Bernoulli random variables $(Z_i)_{i=1}^\infty$ with $\PP(Z_i=1)=1-\eta$ (independent from the variables $(X_i)_i$ and $N^{\lambda}$). Clearly, $\cN^{\lambda} = \cN^{(1-\eta)\lambda}\cup \cN^{\eta \lambda}$, and it is straightforward to prove that both are independent and Poisson point processes with intensities respectively $(1-\eta)\lambda$ and $\eta \lambda$.

\subsection{Statement}

The aim of this section is to prove the  analogue of \cref{thm:main} for Poisson point processes (instead of i.i.d.\ points).

\begin{theorem}\label{thm:limit-poisson}
Let $d \ge 3$, $p\in [1,d)$ and let $\pP = (\cF_{n,n})_{n \in \mathbb{N}}$ be a combinatorial optimization problem over complete bipartite graphs such that assumptions \ref{as:isomorphism}, \ref{as:spanning}, \ref{as:bddegree}, \ref{ass:local-merging} and  \ref{as:growth} hold. Then, there exists $\beta_{\pP} \in (0, \infty)$ (depending on $p$ and $d$) such that the following holds.  

Let $\Omega \subseteq \R^d$ be a bounded domain with Lipschitz boundary and such that \eqref{eq:green-kernel-bound} holds. Let $\rho$ be a H\"older continuous probability density on $\Omega$, uniformly strictly positive and bounded from above. For every $n\in (0, \infty)$, let $\cN^{n\rho}$, $\cM^{n\rho}$ be independent Poisson point processes with intensity $n \rho$ on $\Omega$.  Then,
\begin{equation}\label{eq:limsup-poisson} \limsup_{n \to \infty} n^{\frac{p}{d}-1}\EE\sqa{ \CPp\bra{ \cN^{n\rho}, \cM^{n \rho}}} \le \beta_{\pP} \int_{\Omega} \rho^{1-\frac{p}{d}}.\end{equation}
Moreover, if $\rho$ is the uniform density and $\Omega$ is a cube or its boundary is $C^2$, then the limit exists and equals the right-hand side.
\end{theorem}

After having introduced some general notation and proved some  basic facts, we split the proof into four main cases. We deal first with the case of a uniform density on a cube and establish existence of the limit via subadditivity. Then, we consider H\"older densities on a cube and move next to general domains. Finally, we establish existence of the limit for uniform densities on domains with $C^2$ boundary.

\subsection{General facts}\label{sec:strategy}
Although each case has its distinctive features, the underlying strategy is common and relies on \cref{prop:partition} in combination with a preliminary application of the thinning operation. To avoid repetitions and introduce a general notation, we give a  description of the construction and show a first lemma which uses the fundamental ideas upon which we elaborate in the next sections.

Let $\cN$, $\cM$ be two independent Poisson point processes on $\Omega$ with common intensity given by a finite measure $\lambda$. In our applications, $\lambda$ is Lebesgue measure or $\lambda = n \rho$, but for simplicity here we omit to specify it. We apply the $\eta$-thinning to $\cN = \cN^{1-\eta} \cup \cN^{\eta}$, obtaining independent Poisson point processes with respective intensities $(1-\eta)\lambda$, $\eta \lambda$, and  similarly to $\cM = \cM^{1-\eta} \cup \cM^{\eta}$.  Given a finite Borel partition $\Omega = \bigcup_{k=1}^K \Omega_k$, for each $k=1, \ldots, K$,  
we pick a minimizer $G_k\subseteq \cK\bra{\cN^{1-\eta}_{\Omega_k}, \cM^{1-\eta}_{\Omega_k}}$ for the problem
$$ \CPp( \cN^{1-\eta}_{\Omega_k}, \cM^{1-\eta}_{\Omega_k}).$$
Writing
\[ Z_k = \min\cur{ |\cN^{1-\eta}_{\Omega_k}|, |\cM^{1-\eta}_{\Omega_k}|},\]
we notice that $G_k = \emptyset$ if and only if $Z_k < \cspan$ (by \cref{rem:uniqueness-minimizer} for $p>1$, $G_k$ is a.s.\ unique. For $p=1$ we can consider a measurable selection). 

We define 
 point processes $\cU$, $\cV$ on $\Omega$ by setting $\cU_{\Omega_k} \subseteq \cN^{1-\eta}_{\Omega_k}$,  $\cV_{\Omega_k} \subseteq \cM^{1-\eta}_{\Omega_k}$, given by all the points, respectively in $\cN^{1-\eta}_{\Omega_k}$ and $\cM^{1-\eta}_{\Omega_k}$, which do not belong to the set of vertices of $G_k$.  In particular, if $G_k = \emptyset$, then $\cU_{\Omega_k} = \cN^{1-\eta}_{\Omega_k}$, $\cV_{\Omega_k} = \cM^{1-\eta}_{\Omega_k}$. Notice that by construction the $K$ pairs of processes $\bra{ (\cU_{\Omega_k}, \cV_{\Omega_k})}_{k=1}^K$ are independent, but for any $k$ the two processes $\cU_{\Omega_k}$, $\cV_{\Omega_k}$ are not in general independent. 
 For later use, we prove:
\begin{lemma}
For every $k=1, \ldots, K$ such that 
\begin{equation}\label{eq:condition-omega-k-not-too-small}
\lambda(\Omega_k) > 4 \cspan,
\end{equation}
we have, for every $q  \ge 1$,
\begin{equation}\label{eq:bound-moment-q-remainder-uv} \EE\sqa{ |\cU_{\Omega_k}|^q +|\cV_{\Omega_k}|^q}\lesssim_{q} \lambda(\Omega_k)^{\frac{q}{2}}.\end{equation}
\end{lemma}
\begin{proof}
In the event 
   \begin{equation*}\label{eq:fraction-points} A_k = \cur{Z_k \ge (1-\eta)\lambda(\Omega_k)/2},
   \end{equation*}
     since $\eta \in (0,1/2)$ we have that $Z_k \ge \cspan$ hence by assumption \ref{as:spanning}, every feasible solution (in particular the optimal solution $G_k$) spans a subgraph of $\cK_{ \cN^{1-\eta}(\Omega_k), \cM^{1-\eta}(\Omega_k)}$ isomorphic to $\cK_{Z_k, Z_k}$, so that
 $$    |\cU_{\Omega_k}| \le |\cN^{1-\eta}_{\Omega_k}| - Z_k \le \abs{  |\cM^{1-\eta}_{\Omega_k}| - |\cN^{1-\eta}_{\Omega_k}|}.$$
 Using \eqref{eq:momentPoi}, we have
 $$ \EE\sqa{ |\cU_{\Omega_k}|^q  I_{A_k} } \le \EE\sqa{ \abs{  |\cM^{1-\eta}_{\Omega_k}| - |\cN^{1-\eta}_{\Omega_k}|}^q} \les_q  \lambda(Q_L)^{\frac{q}{2}}.$$
  By the union bound and \eqref{eq:density-bound-below-Poi} with $n = (1-\eta) \lambda(\Omega_k)$, $\gamma  =1/2$, we have
 \begin{equation*}\label{eq:estimate-few-points-Ak} \begin{split} \PP( A^c_k) & \le \PP(|\cN^{1-\eta}_{\Omega_k}| < (1-\eta) \lambda(\Omega_k)/2) + \PP(|\cM^{1-\eta}_{\Omega_k}| < (1-\eta) \lambda(\Omega_k)/2) ) \\
 & \les_q \lambda(\Omega_k)^{-q}.
 \end{split}\end{equation*}
Therefore,
  \[\begin{split} \EE \sqa{ |\cU_{\Omega_k}|^q  I_{A^c_k}} & \le \EE \sqa{  |\cN^{1-\eta}_{\Omega_k}|^q  I_{A^c_k}} \le \EE\sqa{|\cN^{1-\eta}_{\Omega_k}|^{2q} }^{\frac{1}{2}} \PP(A^c_k)^{\frac{1}{2}} \\
 & \les_q \lambda(\Omega_k)^{\frac{q}{2}}.
 \end{split}\]
 Arguing similarly for $|\cV_{\Omega_k}|$, we obtain  \eqref{eq:bound-moment-q-remainder-uv}.
 \end{proof}
 
For $k = 1, \ldots, K$, we define 
\[ \x^k = \cN^{1-\eta}_{\Omega_k} \setminus \cU_{\Omega_k}, \quad \y^k = \cM^{1-\eta}_{\Omega_k} \setminus \cV_{\Omega_k},\]
 so that by construction  $|\x^k|  = |\y^k| =n_k$, with
\[ n_k = \begin{cases}Z_k & \text{if $Z_k \ge \cspan$,}\\
0 & \text{otherwise.}
\end{cases}\]
Moreover, since the optimizer $G_k$ is a feasible solution in $\cK(\x^k, \y^k)$, we have
$$ \CPp(\cN^{1-\eta}_{\Omega_k}, \cM^{1-\eta}_{\Omega_k}) = \CPp( \x^k, \y^k).$$
We then let $\x^0 = \cN^\eta \cup \cU$, $\y^0 = \cM^\eta \cup \cV$. In the event 
\begin{equation}\label{eq:event-many-points}
 \cur{ \min\cur{|\cN^\eta|, |\cM^\eta|} \ge \min\cur{K, \cspan}},
\end{equation}
    \cref{prop:partition} applies for any choice of points $\z = (z_k)_{k=1}^K$ with $z_k \in \Omega_k$, yielding the inequality
\begin{equation}\label{eq:subadditive-random} 
\CPp\bra{ \cN, \cM}   - \sum_{k=1}^K \CPp(\cN^{1-\eta}_{\Omega_k}, \cM^{1-\eta}_{\Omega_k})  
\les  \CPp(\cN^\eta \cup \cU, \cM^\eta \cup \cV)+ \Mp(\cN^\eta \cup \cU, \z)  +   \sum_{k=1}^K    \diam(\Omega_k)^p. 
\end{equation}
 By \cref{rem:remove-z-points}, if also
\begin{equation}\label{eq:event-many-points-each-omegak}  \cur{ \min_{k=1,\ldots, K} \min\cur{ |\cN^\eta_{\Omega_k}|, |\cM^{\eta}_{\Omega_k}|} \ge 1}\end{equation}
 then the term  $\Mp(\cN^\eta \cup \cU, \z)$ can be removed in \eqref{eq:subadditive-random}. 

Once \eqref{eq:subadditive-random} is established, the next step is to take expectation and carefully estimate the ``error terms'' in the right-hand side. To convey the main ideas, we start with the simplest case when $K$ is kept fixed as we let $n \to \infty$ in the intensity of the process $\lambda = n \rho$.

\begin{lemma}\label{lem:subadditivity-finite-partition}
%
With the notation and assumptions of \cref{thm:limit-poisson}, fix $K \in \mathbb{N}$ and consider a Borel partition $\Omega =  \bigcup_{k=1}^K \Omega_k$. Then,
\begin{equation}\label{eq:limsup-poisson-finite-partition} \limsup_{n \to \infty} n^{\frac{p}{d}-1} \EE\sqa{ \CPp\bra{ \cN^{n\rho}, \cM^{n \rho}}} \le \sum_{k=1}^K \limsup_{n \to \infty} n^{\frac{p}{d}-1}\EE\sqa{ \CPp\bra{ \cN^{n\rho}_{\Omega_k}, \cM^{n \rho}_{\Omega_k}}} .\end{equation}
\end{lemma}

\begin{proof}
We can assume that each $\Omega_k$ is not negligible. Then, condition \eqref{eq:condition-omega-k-not-too-small} with $\lambda = n \rho$ holds if $n$ is sufficiently large.
Letting 
\begin{equation}\label{eq:A-finiteK} A = \bigcap_{k=1}^K \cur{ \min\cur{|\cN^{n\eta\rho}_{\Omega_k}|, |\cM^{n\eta\rho}_{\Omega_k}|} \ge  \cspan} =\bigcap_{k=1}^K A_k,\end{equation}
we have that  both \eqref{eq:event-many-points} and \eqref{eq:event-many-points-each-omegak} hold $A$.

 By the union bound in combination with \eqref{eq:density-bound-below-Poi}, we estimate, for every $q \ge 1$,
\begin{equation}\label{eq:union-bound} \PP(A^c) \le \sum_{k=1}^K \PP(A^c_k)  \les_{q,\eta, K} n^{-q}.\end{equation}
 Combined with the trivial inequality $\CPp\bra{ \cN^{n\rho}, \cM^{n\rho}} \les |\cN^{n\rho}|$ we obtain that
\[ \EE\sqa{ \CPp\bra{ \cN^{n\rho}, \cM^{n\rho}}  I_{A^c} } \le \EE\sqa{ |\cN^{n\rho}|^2}^{\frac{1}{2}} \PP(A^c)^{\frac{1}{2}} \les_{q,K}  n^{\frac{1-q}{2}},\]
which is infinitesimal if $q>1$ (even without dividing by $n^{1-p/d}$). Therefore,
\begin{equation*}
\label{eq:Ac-does-not-matter} 
\limsup_{n \to \infty} n^{\frac{p}{d}-1}\EE\sqa{ \CPp\bra{ \cN^{n\rho}, \cM^{n \rho}} I_A}  =  \limsup_{n \to \infty} n^{\frac{p}{d}-1}\EE\sqa{ \CPp\bra{ \cN^{n\rho}, \cM^{n \rho}} }
\end{equation*}
and we only need to prove the following inequality, for fixed $\eta$,
\begin{equation*}\label{eq:limsup-poisson-finite-partition-eta}
\limsup_{n \to \infty} n^{\frac{p}{d}-1}\EE\sqa{ \CPp\bra{ \cN^{n\rho}, \cM^{n \rho}} I_A}  - \sum_{k=1}^K \limsup_{n \to \infty} n^{\frac{p}{d}-1}\EE\sqa{ \CPp\bra{ \cN^{n\rho}_{\Omega_k}, \cM^{n \rho}_{\Omega_k}}}  \les_{K} \eta^{1-\frac{p}{d}},
\end{equation*}
and finally let $\eta\to 0$ to obtain the thesis.
To this aim, we multiply \eqref{eq:subadditive-random} by $I_A$ and take expectation, obtaining the inequality
\begin{multline}\label{eq:exp-sub-explicit}
\EE\sqa{ \CPp\bra{ \cN^{n\rho}, \cM^{n\rho}}  I_{A} }  - \sum_{k=1}^K  \EE\sqa{ \CPp\bra{ \cN^{(1-\eta)n\rho}_{\Omega_k}, \cM^{(1-\eta)n\rho}_{\Omega_k}}  } 
 \\ \les   \EE\sqa{ \CPp\bra{ \cN^{\eta n\rho}\cup \cU, \cM^{\eta n\rho}\cup \cV}}+K.  
 \end{multline}
Since, for each $k=1, \ldots, K$,
\begin{equation*}\label{eq:exp-sub-finite-partition} \begin{split}  \limsup_{n \to \infty}  n^{\frac{p}{d}-1}\EE&\sqa{ \CPp\bra{ \cN^{(1-\eta)n\rho}_{\Omega_k}, \cM^{(1-\eta)n\rho}_{\Omega_k}}  } \\
& = (1-\eta)^{1-\frac{p}{d}} \limsup_{n \to \infty} \bra{(1-\eta)n}^{\frac{p}{d}-1}\EE\sqa{ \CPp\bra{ \cN^{(1-\eta)n\rho}_{\Omega_k}, \cM^{(1-\eta)n\rho}_{\Omega_k}}  }\\
& \le \limsup_{n \to \infty} n^{\frac{p}{d}-1}\EE\sqa{ \CPp\bra{ \cN^{n\rho}_{\Omega_k}, \cM^{n\rho}_{\Omega_k}}  },
\end{split}\end{equation*}
we need to focus only on the terms in the right-hand side of \eqref{eq:exp-sub-explicit}. Since the last term is constant, we are left with the proof of 	
\begin{equation}\label{tofinishfinite} 
\limsup_{n \to \infty} n^{\frac{p}{d}-1}\EE\sqa{ \CPp\bra{ \cN^{\eta n\rho}\cup \cU, \cM^{\eta n\rho}\cup \cV} }\les \eta^{1-\frac{p}{d}}. 
\end{equation}
We first notice that by  \eqref{eq:bound-moment-q-remainder-uv} and H\"older inequality we have  for every $q\ge 1$,
 \begin{equation}\label{momentuVfinite}
  \EE\sqa{ |\cU|^q +|\cV|^q}\lesssim_{q}  K^{\frac{q}{2}}n^{\frac{q}{2}}.
 \end{equation}
We now use  assumption~\ref{as:growth} so that 
\begin{equation}\label{eq:last-term-finite-partition} 
\EE\sqa{ \CPp\bra{ \cN^{\eta n\rho}\cup \cU, \cM^{\eta n\rho}\cup \cV}}   \les \EE\sqa{ | \cN^{\eta n\rho}\cup \cU|^{1-\frac{p}{d}}}
 + \EE\sqa{ \Mp\bra{ \cN^{\eta n\rho}\cup \cU, \cM^{\eta n\rho}\cup \cV}}.
 \end{equation}
 To estimate the  first term in the right-hand side, we use H\"older inequality and \eqref{momentuVfinite} with $q=1$,
 \begin{equation}\label{eq:last-term-easy-finite-partition} \begin{split}
 \EE\sqa{| \cN^{\eta n\rho}\cup \cU|^{1-\frac{p}{d}}} & \les \EE\sqa{ |\cN^{\eta n \rho}|}^{1-\frac{p}{d}} + \EE\sqa{  |\cU|}^{1-\frac{p}{d}}\\
 & \les n^{1-\frac{p}{d}}\bra{\eta^{1-\frac{p}{d}} + C_K n^{-\frac{1}{2}(1-\frac{p}{d})}}.\end{split}\end{equation}
 
 \noindent For the second term, thanks to \eqref{momentuVfinite} we may use  \cref{prop:density-helps-matching} with $H=n^{1/2}$
 and $h=\min\cur{\EE\sqa{|\cN^{\eta n\rho}|}, \EE\sqa{|\cM^{\eta n\rho}|}}\sim n\eta$ so that for some $\alpha<2$ and $\beta>0$
 \[
  \EE\sqa{ \Mp\bra{ \cN^{\eta n\rho}\cup \cU, \cM^{\eta n\rho}\cup \cV}}\les n^{1-\frac{p}{d}}\bra{\eta^{1-\frac{p}{d}} + C_{K,\eta} n^{-\frac{\beta}{2} (2-\alpha)}}.  
 \]
Plugging  this and  \eqref{eq:last-term-easy-finite-partition} in \eqref{eq:last-term-finite-partition} concludes the proof of \eqref{tofinishfinite}.
 \end{proof}
 
 \begin{remark}\label{rem:limit-case-finite-partition}
We notice that the proof above yields also the inequality
 \begin{equation}\begin{split}\label{eq:liminf-poisson-finite-partition} \liminf_{n \to \infty} n^{\frac{p}{d}-1}\EE\sqa{ \CPp\bra{ \cN^{n\rho}, \cM^{n \rho}}}  & \le \liminf_{n \to \infty} n^{\frac{p}{d}-1}\EE\sqa{ \CPp\bra{ \cN^{n\rho}_{\Omega_1}, \cM^{n \rho}_{\Omega_1}}} \\
 & \quad  + \sum_{k=2}^K \limsup_{n \to \infty} n^{\frac{p}{d}-1} \EE\sqa{ \CPp\bra{ \cN^{n\rho}_{\Omega_k}, \cM^{n \rho}_{\Omega_k}}} .\end{split}\end{equation}
 This follows by repeating the argument only along a subsequence $n_\ell \to \infty$ such that
 \[ \lim_{\ell \to \infty} n_{\ell}^{\frac{p}{d}-1}\EE\sqa{ \CPp\bra{ \cN^{n_\ell \rho}_{\Omega_1}, \cM^{n_\ell \rho}_{\Omega_1}}} = \liminf_{n \to \infty} n^{\frac{p}{d}-1}\EE\sqa{ \CPp\bra{ \cN^{n\rho}_{\Omega_1}, \cM^{n \rho}_{\Omega_1}}}.\]
\end{remark}

\subsection{Uniform density on a cube}\label{sec:cube-constant-density}
In this section we consider the case of a uniform measure on a cube. Up to rescaling (see \eqref{eq:homogeneity}) it is equivalent to consider 
two independent Poisson point processes $\cN_{Q_L}$ and $\cM_{Q_L}$ with intensity one on $Q_L$ and prove that 
\[
f(L) =\frac{1}{|Q_L|} \EE\sqa{ \CPp\bra{ \cN_{Q_{L}}, \cM_{Q_L}}}
\]
has a limit as $L\to \infty$.
\begin{proposition}\label{prop:uniform}
Let $d \ge 3$, $p\in [1,d)$ and let $\pP = (\cF_{n,n})_{n \in \mathbb{N}}$ be a combinatorial optimization problem
over complete bipartite graphs such that assumptions \ref{as:isomorphism}, \ref{as:spanning}, \ref{as:bddegree}, \ref{ass:local-merging} and  \ref{as:growth} hold.
Then, there exists $\beta_{\pP} \in (0, \infty)$ (depending on $p$ and $d$) such that 
\begin{equation}\label{eq:limit-cube} 
\lim_{L \to \infty} f(L) = \beta_{\pP}. 
\end{equation}

\end{proposition}
\begin{proof}
%

We split the proof into several steps. In the first two steps we establish basic properties of $f$, before
moving to the main argument. This  follows the strategy of the previous section and ultimately relies upon an application of \cref{lem:sub}.\\

\medskip 
\setcounter{proof-step}{0}

\noindent{\emph{Step \stepcounter{proof-step}\arabic{proof-step}. Continuity and upper bound.}}
Writing $z = \min\cur{n,m}$, we first notice that by Assumption \ref{as:growth} and \eqref{eq:matching-iid} of \cref{prop:matching-iid},
\[\EE\sqa{ \CPp \bra{ (X_i)_{i=1}^n, (Y_j)_{j=1}^m}}  \les z^{1-\frac{p}{d}} + \EE\sqa{ \Mp\bra{ (X_i)_{i=1}^n, (Y_j)_{j=1}^m}} \les z^{1-\frac{p}{d}}.\]
This proves on the one hand that $f$ is bounded from above as
\begin{equation}\label{eq:uniform-bound-f} f(L) \les L^{p-d} \EE\sqa{ \min\cur{ |\cN_{Q_L}|, |\cM_{Q_L}|}^{1-\frac{p}{d}} }\les L^{p-d} \EE\sqa{ |\cN_{Q_L}| }^{1-\frac{p}{d}} \les 1.\end{equation}
On the other hand, combining it with dominated convergence it also gives continuity of $f$ thanks to the representation formula
\begin{equation*}\label{eq:conditioning} \begin{split} 
f(L) & = \sum_{n, m=0}^\infty \frac{1}{L^d} \EE\sqa{\CPp(\cN_{Q_{L}}, \cM_{Q_L}) \Big |  |\cN_{Q_L}|=n,  |\cM_{Q_L}|=m} e^{-2L^d} \frac{ (L^d)^{n+m}}{n! m!}\\
& =  L^{p-d} \sum_{n, m=0}^\infty \EE\sqa{\CPp \bra{ (X_i)_{i=1}^n, (Y_j)_{j=1}^m}} e^{-2L^d} \frac{ (L^d)^{n+m}}{n! m!}
\end{split}\end{equation*}
where  $(X_i)_{i=1}^n$, $(Y_j)_{j=1}^m$ are i.i.d.\ points on $Q_1$.
%
We also notice that by a simple scaling argument, if $\cN^{\lambda}$, $\cM^{\lambda}$ are independent Poisson processes  of intensity $\lambda>0$ on $Q_L$ then 
\begin{equation}\label{eq:homogeneity}  
\frac{1}{|Q_L|} \EE\sqa{ \CPp\bra{ \cN_{Q_{L}}^{\lambda}, \cM_{Q_L}^{\lambda} }} = \frac{\lambda^{1-\frac{p}{d}} }{|Q_{\lambda L}|} \EE\sqa{  \CPp\bra{ \cN_{Q_{\lambda^{\frac{1}{d}} L}}^1, \cM_{Q_{\lambda^{\frac{1}{d}} L}}^1}} = \lambda^{1-\frac{p}{d}} f(\lambda^{\frac{1}{d}} L).\end{equation}
Combined with  \eqref{eq:uniform-bound-f}, it yields that  for any cube $Q$ and $\lambda>0$,
\begin{equation}\label{eq:general-upper-bound-cube}
\EE\sqa{ \CPp\bra{ \cN_{Q}^{\lambda}, \cM_{Q}^{\lambda} }} \les |Q| \lambda^{1-\frac{p}{d}}. 
\end{equation}

\noindent{\emph{Step \stepcounter{proof-step}\arabic{proof-step}. Lower bound.}}
The spanning assumption \ref{as:spanning} yields that, if e.g.\ $\cspan \le n \le m$, then
$$ \CPp\bra{  (X_i)_{i=1}^n, (Y_j)_{j=1}^m} \ge  \sum_{i=1}^n \min_{j=1, \ldots, m} |X_i - Y_j|^p.$$
The following classical lower bound, e.g.\ proved in \cite[Chapter 2]{steele1997probability},
$$ \EE\sqa{  \min_{j=1, \ldots, m} |X_i - Y_j|^p } \gtrsim m^{-\frac{p}{d}}$$
entails that
$$  \CPp\bra{  (X_i)_{i=1}^n, (Y_j)_{j=1}^m}  \gtrsim m^{-\frac{p}{d}} \cdot n.$$
Writing $Z = \min\cur{|\cN_{Q_L}|, |\cM_{Q_L}|}$, we deduce that
$$ f(L) \gtrsim L^{p-d} \EE\sqa{   \max\cur{|\cN_{Q_L}|,|\cM_{Q_L}|}^{-\frac{p}{d}} Z I_{\cur{ Z\ge \cspan}}}.$$
Let
$$ A =  \cur{ |Q_L|/2 \le Z \le \max\cur{|\cN_{Q_L}|, |\cM_{Q_L}|} \le 3|Q_L|/2}.$$
By \eqref{eq:density-bound-below-Poi} with $\eta=1/2$, we have $\PP(A) \to 1$ as $L \to \infty$. Therefore id $L$ is large enough,
\[  f(L) 
   \gtrsim L^{p-d} \EE\sqa{   \max\cur{|\cN_{Q_L}|,|\cM_{Q_L}|}^{-\frac{p}{d}} Z I_{A}}  \gtrsim L^{p-d} \EE\sqa{ L^{d-p} I_{A}} \gtrsim 1.
\]

In the remaining steps we prove the following claim. There exists $\beta=\beta(p,d)>0$ such that for every $\eta\in(0,1/2)$, there exists $C(\eta)>0$
such that, for every $m \in \N$, $m \ge 1$ and $L \ge C(\eta)$,
\begin{equation}\label{toprovemonotonicity}
 f(mL)-f( (1-\eta)L)\les  \eta^{1-\frac{p}{d}} + C(\eta) L^{-\beta}.
\end{equation}
This would conclude the proof of \eqref{eq:limit-cube} by  \cref{lem:sub}.

\noindent{\emph{Step \stepcounter{proof-step}\arabic{proof-step}. Partitioning and exclusion of the event in which few points are sampled.}}  
Using the notation from  \cref{sec:strategy}, we partition $\Omega = Q_{mL}$ into $K=m^d$ cubes $Q_i = Q_L + L z_i \subseteq Q_{mL}$ with $z_i\in \mathbb{Z}^d$
and two independent Poisson processes $\cN$, $\cM$ of unit intensity on $Q_{mL}$.

We first reduce  to the event
$$ A = \cur{ \min\cur{ |\cN^{\eta}_{Q_{mL}}|,  |\cM^{\eta}_{Q_{mL}}|} \ge \eta |Q_{mL}|/2 },$$
which contains \eqref{eq:event-many-points} 
provided  $L$ is sufficiently large (depending on $\eta$ only, not on $m$).
We first argue that $A^c$ is of small probability. Indeed, using a union bound we find that for every $q \ge 1$, 
\[  \PP(A^c)  \le \PP\bra{ |\cN^{\eta}_{Q_{mL}}| < |Q_{mL}|/2 } +\PP\bra{ |\cM^{\eta}_{Q_{mL}}| < |Q_{mL}|/2 }    \stackrel{\eqref{eq:density-bound-below-Poi}}{\les_{\eta,q}} |Q_{mL}|^{-q}.\]
If $A^c$ holds, we use the trivial bound that follows from Assumption \ref{as:bddegree}:
$$ \CPp( \cN_{Q_{mL}}, \cM_{Q_{mL}}) \les |\cN_{Q_{mL}}| |Q_{mL}|^{\frac{p}{d}},$$
so that for any given $\beta>0$ and  provided  we choose $q$ sufficiently large.
\begin{equation}\label{eq:bad-set} \begin{split} \frac 1 {|Q_{mL}|} \EE\sqa{ \CPp( \cN_{Q_{mL}}, \cM_{Q_{mL}}) I_{A^c}} & \les |Q_{mL}|^{\frac{p}{d}-1} \EE\sqa{ |\cN_{Q_{mL}}|^2}^{\frac{1}{2}} \PP(A^c)^{\frac{1}{2}} \\
& \les_{\eta,q} |Q_{mL}|^{\frac{p}{d}-1} |Q_{mL}| \cdot |Q_{mL}|^{-q} \les_\eta L^{-\beta}.
\end{split}\end{equation}

\noindent If $A$ holds, letting $\z$ be the set of centres of the $m^d$ cubes, inequality \eqref{eq:subadditive-random} reads
\begin{equation}\label{eq:sub-real} \begin{split}  \CPp( \cN_{Q_{mL}}, \cM_{Q_{mL}})  - \sum_{i=1}^{m^d} \CPp(\cN^{1-\eta}_{Q_i}, \cM^{1-\eta}_{Q_i}) 
& \les \CPp(\cN^\eta\cup \cU, \cM^\eta \cup \cV) \\
& \quad + \Mp(\cN^\eta \cup \cU, \z) + m^d L^p.
\end{split}\end{equation}
 Notice that by the properties of the Poisson point process, the law of  $\CPp(\cN^{1-\eta}_{Q_i}$, $\cM^{1-\eta}_{Q_i})$ equals that of $\CPp(\cN^{1-\eta}_{Q_L}$, $\cM^{1-\eta}_{Q_L})$. In particular 
 \[\begin{split} \frac{1}{|Q_L|} \EE\sqa{\CPp(\cN^{1-\eta}_{Q_i}, \cM^{1-\eta}_{Q_i})I_A} & \le \frac{1}{|Q_L|} \EE\sqa{\CPp(\cN^{1-\eta}_{Q_i}, \cM^{1-\eta}_{Q_i})}  \\
& = \frac{1}{|Q_L|} \EE\sqa{\CPp(\cN^{1-\eta}_{Q_L}, \cM^{1-\eta}_{Q_L})}\\
& \stackrel{\eqref{eq:homogeneity}}{=} (1-\eta)^{1-\frac{p}{d}}f((1-\eta)^{\frac{1}{d}}L) \le f((1-\eta)^{\frac{1}{d}}L).\end{split}\]
We thus obtain from \eqref{eq:sub-real},
\begin{multline*}
  \frac{1}{|Q_{mL}|}  \EE\sqa{ \CPp( \cN_{Q_{mL}}, \cM_{Q_{mL}}) I_A}    - f((1-\eta)^{\frac{1}{d}}L) \les \frac{1}{|Q_{mL}|}\EE\sqa{ \CPp(\cN^\eta\cup \cU, \cM^\eta \cup \cV)}\\  
 + \frac{1}{|Q_{mL}|}  \EE\sqa{ \Mp(\cN^\eta \cup \cU, \z) I_A} + L^{p-d}.
\end{multline*}
In the final two steps we prove that 
\begin{equation}\label{claimMz}
  \frac{1}{|Q_{mL}|}  \EE\sqa{ \Mp(\cN^\eta \cup \cU, \z) I_A}\les L^{p-d}
\end{equation}
and 
\begin{equation}\label{claimNM}
 \frac{1}{|Q_{mL}|}\EE\sqa{ \CPp(\cN^\eta\cup \cU, \cM^\eta \cup \cV)}\les \eta^{1-\frac{p}{d}} + C(\eta) L^{-\beta}.
\end{equation}
In combination with \eqref{eq:bad-set} this would conclude the proof of \eqref{toprovemonotonicity}.\\

\noindent{\emph{Step \stepcounter{proof-step}\arabic{proof-step}. Proof of \eqref{claimMz}.}}  
On $A$, we have $|\cN^{\eta}_{Q_{mL}}| \ge m^d$,  thus (randomly) choosing $m^d$ points from $\cN^{\eta}$, we find after relabelling a family $(X_i)_{i=1}^{m^d}$ of points i.i.d.\ and uniformly distributed on $Q_{mL}$. Recalling that $\z$ denotes the set of centres of the $m^d$ cubes $Q_i$ we can bound 
\[  \EE\sqa{ \Mp(\cN^\eta \cup \cU, \z) I_A} \le  \EE\sqa{ \Mp((X_i)_{i=1}^{m^d}, \z)}.\]
We then use \eqref{eq:matching-below-wasserstein} with $n=m^d$ and $\lambda$ the uniform density on the cube $Q_{mL}$, so that
\[ \EE\sqa{ \Mp((X_i)_{i=1}^{m^d}, \z)} \les \EE\sqa{ \W_{Q_{mL}}\bra{ \sum_{i=1}^{m^d} \delta_{X_i}, \frac{m^d}{|Q_{mL}|}}} + \EE\sqa{ \W_{Q_{mL}}\bra{ \mu^\z, \frac{m^d}{|Q_{mL}|}}} \les  m^d L^p,\]
having used \eqref{eq:transport-iid} to bound the first term (the second term is trivially estimated by transporting the mass on each cube $Q_i$ to its center). This proves \eqref{claimMz}.\\
\medskip

\noindent{\emph{Step \stepcounter{proof-step}\arabic{proof-step}. Proof of \eqref{claimNM}.}}
  We use Assumption \ref{as:growth} (on $Q_{mL}$ instead of $Q_1$, see \cref{remgrowth}), so that
\[ \begin{split} \EE\sqa{ \CPp(\cN^\eta\cup \cU, \cM^\eta \cup \cV) } &\les  (mL)^p\EE\sqa{ (|\cN^\eta_{Q_{mL}}|+|\cU|)^{1-\frac{p}{d}}  } + \EE\sqa{ \Mp(\cN^\eta\cup \cU, \cM^\eta \cup \cV)}.
\end{split}\]
We further bound the first contribution using H\"older inequality
\begin{equation}\label{eq:first-contribution-cube}\begin{split} \EE\sqa{ (|\cN^\eta_{Q_{mL}}|+|\cU|)^{1-\frac{p}{d}}  }&    \le   \EE\sqa{  |\cN^{\eta}_{Q_{mL}}|}^{1-\frac{p}{d}}+ \EE\sqa{ |\cU|}^{1-\frac{p}{d}} \\
& \les \eta^{1-\frac{p}{d}} (mL)^{d-p} + \EE\sqa{ |\cU|}^{1-\frac{p}{d}}.
\end{split}\end{equation}
To proceed further, let us recall that in \cref{sec:strategy} we argued that $\bra{ \bra{\cU_{Q_i}, \cV_{Q_i}}}_{i=1}^{m^d}$ are independent (and also independent from $\cN^{\eta}$, $\cM^{\eta}$). Moreover, since the law of each $\bra{\cN^{1-\eta}_{Q_i}, \cM^{1-\eta}_{Q_i}}$ coincides  with that of $\bra{\cN^{1-\eta}_{Q_L}, \cM^{1-\eta}_{Q_L}}$ (up to a translation by $-Lz_i$, since $Q_i = Q_L+Lz_i$) it follows that the same property holds for the processes $\bra{\cU_{Q_i}, \cV_{Q_i}}$: their law coincides with that of  $\bra{\cU_{Q_L}, \cV_{Q_L}}$ (also up to translating by $-Lz_i$).  
%

Using \eqref{eq:bound-moment-q-remainder-uv} with $q=1$, we obtain
\[ \EE\sqa{ |\cU|} = m^d \EE\sqa{ |\cU_{Q_L}|} \les m^d L^{\frac{d}{2}},\]
thus \eqref{eq:first-contribution-cube} yields
\[ \frac{(mL)^p}{|Q_{mL}|}\EE\sqa{(|\cN^\eta_{Q_{mL}}|+|\cU|)^{1-\frac{p}{d}}  } \les  \eta^{1-\frac{p}{d}} + L^{\frac{p-d}{2}}.\] 
Combining this with \cref{thm:bound-matching}, concludes the proof of \eqref{claimNM}.
\end{proof}

\subsection{H\"older density on a  cube}

%

In this section, we still assume that $\Omega = Q$ is a cube, but consider the case of a general H\"older continuous density $\rho$, uniformly bounded from above and below. Up to rescaling and translation, it is sufficient to consider the case $\Omega = (0,1)^d$. 
 
 The proof of \eqref{eq:limsup-poisson} in this case  is obtained by combining  the case of  constant density treated above together with  \cref{lem:subadditivity-finite-partition} and the following claim:
 there exists a constant $C = C(\rho)>0$ such that, for $r<C$ and for every cube $Q  \subseteq (0,1)^d$ with side length $r$ the following inequality holds:
 \begin{equation}\label{eq:inequality-cube-holder-iid}
 \EE\sqa{ \CPp\bra{  \cN^{n \rho}_{Q}, \cM^{n \rho}_{Q} }}\le (1+C^{-1} r^\alpha) \EE\sqa{ \CPp\bra{  \cN^{n \rho(Q)/r^d}, \cM^{n \rho(Q)/r^d}}},
 \end{equation}
 where $\cN^{n \rho(Q)/r^d}$, $\cM^{n \rho(Q)/r^d}$ are two independent Poisson point processes with constant intensity  $n \rho(Q)/r^d$ on the cube $(0,r)^d$, and $\alpha$ denotes the H\"older exponent of $\rho$. 
 
Indeed, assume that the claim holds and let us prove \eqref{eq:limsup-poisson}. Given any $r< C(p,\rho)$ of the form $r = 1/K^{1/d}$, we consider a partition of $(0,1)^d = \bigcup_{k=1}^K Q_k$ into $K$ disjoint sub-cubes of side length $r$, so that
\[\begin{split} \limsup_{n \to \infty} n^{\frac{p}{d}-1} \EE & \sqa{ \CPp\bra{  \cN^{n \rho}, \cM^{n \rho} }} \\
&  \stackrel{\eqref{eq:limsup-poisson-finite-partition}}{\le} \sum_{k=1}^{K} \limsup_{n \to \infty}n^{\frac{p}{d}-1}\EE\sqa{ \CPp\bra{  \cN^{n \rho}_{Q_k}, \cM^{n \rho}_{Q_k} }}\\
& \stackrel{\eqref{eq:inequality-cube-holder-iid}}{\le} (1+C^{-1} r^\alpha)\sum_{k=1}^{K}  \limsup_{n \to \infty}n^{\frac{p}{d}-1}\EE\sqa{ \CPp\bra{  \cN^{n \rho(Q_k)/r^d}, \cM^{n \rho(Q_k)/r^d} }}\\
& = \beta_{\pP}  (1+C^{-1} r^\alpha) \sum_{k=1}^{K}  \rho(Q_k)^{1-\frac{p}{d}} r^p,
\end{split}\]
where the last line follows from  \eqref{eq:limsup-poisson} in the case of a cube and constant intensity. Letting $K \to \infty$, we have that $r\to 0$ and the easily seen convergence
\[ \lim_{K \to \infty} \sum_{k=1}^{K}  \rho(Q_k)^{1-\frac{p}{d}} r^p = \lim_{K \to \infty}\int_{(0,1)^d} \sum_{k=1}^K I_{Q_k} \bra{ \frac{ \rho(Q_k)}{r^d}}^{-\frac{p}{d}} \rho  = \int_{(0,1)^d} \rho^{1-\frac{p}{d}}.\]
This would conclude the proof of \eqref{eq:limsup-poisson} also in this case.

%
%
 
%
We now prove  \eqref{eq:inequality-cube-holder-iid} for which we   closely follow \cite[Lemma 2.5]{ambrosio2022quadratic}. Up to translating, we may  assume that $Q = (0,r)^d$. We write $\rho_0 = \min_{(0,1)^d} \rho$ and define $\rho^r(x) = \rho(rx)r^d/\rho(Q)$ for $x \in (0,1)^d$, so that $\int_{(0,1)^d} \rho^r = 1$, and for every $x$, $y\in (0,1)^d$,
\[ \rho^r(x) - \rho^r(y) \le \frac{ \nor{\rho}_{C^\alpha}}{\bar{\rho}} r^\alpha |x-y|^\alpha,\]
thus $\nor{ \rho^r - 1 }_{C^\alpha}\les r^\alpha$ if $r$ is sufficiently small.
 We define $S: Q \to Q$ as  $S(x) = r T^{-1}(x/r)$, where $T$ is the map provided by \cref{prop:map-heat-semigroup}.
 It holds $\Lip S = \Lip T^{-1}$,  and $S_\sharp 1/r^d = \rho/\rho(Q)$. Therefore,  $S\bra{ \cN^{n\rho(Q)/r^d}}= (S(X_i))_{i=1}^{N^{n\rho(Q)/r^d}(Q)}$, which is a Poisson point process on $Q$ with intensity $n \rho$, i.e., it has the same law as $\cN^{n \rho}_Q$, and similarly $S\bra{\cM^{n\rho(Q)/r^d}}$ has the same law as $\cM^{n\rho}_Q$. Therefore,
\[\begin{split} \EE\sqa{ \CPp\bra{  \cN^{n \rho}_{Q}, \cM^{n \rho}_{Q} }}& =  \EE\sqa{ \CPp\bra{  S(\cN^{n \rho(Q)/r^d}), S(\cM^{n \rho(Q)/r^d}) }}\\
& \stackrel{\eqref{eq:lipschitz-bound-deterministic}}{\le} (\Lip S)^p \EE\sqa{ \CPp\bra{  \cN^{n \rho(Q)/r^d}, \cM^{n \rho(Q)/r^d}}}.
\end{split}\]
This proves the claim since  $(\Lip S)^p = (\Lip T^{-1})^p \le 1+C r^\alpha$ if $r$ is sufficiently small.
%

\begin{remark}\label{rem:general-domain-holder-density}
Let us notice that 
the fact that $\Omega$ is a cube is not used in the proof of \eqref{eq:inequality-cube-holder-iid}, which therefore holds true for every bounded domain $\Omega$ 
and H\"older continuous density $\rho$ uniformly bounded from above and below. In particular, combining \eqref{eq:inequality-cube-holder-iid} with \eqref{eq:general-upper-bound-cube} we obtain that there exists $C=C(\rho)>0$ such that, for every cube $Q \subseteq \Omega$ with side length $r<C$,
\begin{equation}\label{eq:inequality-cube-holder-general}
\EE\sqa{ \CPp\bra{  \cN^{n \rho}_{Q}, \cM^{n \rho}_{Q} }}\les |Q| n^{1-\frac{p}{d}},
\end{equation}
where the implicit constant depends on $p$, $d$ and $\rho$ only.
\end{remark}

\subsection{General density on a domain}

We prove \eqref{eq:limsup-poisson} for a domain $\Omega$ and a H\"older density $\rho$.  The main difficulty here is that since we rely on the result established in the previous section we need to partition $\Omega$ into cubes. 
 This is accomplished relying on the Whitney-type decomposition provided by \cref{lem:decomp}. We begin by fixing a Whitney decomposition
 $\cQ =(Q_i)_i$ such that every cube $Q_i$ has side length $r< C$, where $C=C(\rho)$ is as in \cref{rem:general-domain-holder-density}.  Then, by \cref{lem:decomp}, for every  sufficiently small $\delta>0$ we  have a finite Borel partition of $\Omega = \bigcup_{k=1}^K \Omega_k$, whose elements are collected into the two disjoint sets $\cQ_{\delta}$, $\cR_{\delta}$.

We fix $\eta \in (0,1/2)$ and use  the construction from \cref{sec:strategy}. We set $\delta=n^{-\gamma}$ for $\gamma>0$ to be fixed below. The first constraint is that \eqref{eq:condition-omega-k-not-too-small} holds with $\lambda=n\rho$ so that we need $n\delta^d\gg1$, i.e.
\begin{equation}\label{eq:gamma-first-condition}
\gamma d<1.
\end{equation}
We first reduce to the case when there are many points in each $\Omega_k$. Defining the event $A$ as in \eqref{eq:A-finiteK} and arguing as in 
  \eqref{eq:union-bound} gives here, for every $q>0$, the inequality
\[ \begin{split} \PP(A^c) & \les_{q,\eta} \sum_{k=1}^K (n |\Omega_k|)^{-q} \les_{q,\eta} n^{-q} \delta^{1-d-dq} = n^{-q(1-d\gamma)+(d-1)\gamma },
\end{split}\]
where we used \eqref{eq:whitney-general-q} with $\alpha = -qd$ in the second inequality. Under the assumption \eqref{eq:gamma-first-condition} this is infinitesimal provided  $q$ is chosen sufficiently large.
Arguing exactly as before we can thus reduce ourselves to the case where $A$ holds. In that case,  both \eqref{eq:event-many-points} and \eqref{eq:event-many-points-each-omegak} hold  and thus by \eqref{eq:subadditive-random}
\begin{equation}\label{eq:exp-sub-explicit-2} \begin{split} \EE\sqa{ \CPp\bra{ \cN^{n\rho}, \cM^{n\rho}}  I_{A} } & - \sum_{k=1}^K  \EE\sqa{ \CPp\bra{ \cN^{(1-\eta)n\rho}_{\Omega_k}, \cM^{(1-\eta)n\rho}_{\Omega_k}}  } \\
& \qquad \les   \EE\sqa{ \CPp\bra{ \cN^{\eta n\rho}\cup \cU, \cM^{\eta n\rho}\cup \cV}} + \sum_{k=1}^K  \diam(\Omega_k)^p.  
\end{split}\end{equation}
We start by considering the left-hand side of \eqref{eq:exp-sub-explicit-2}. For $\Omega_k\in \cR_\delta$ we use the simple bound $ \CPp\bra{ \cN^{(1-\eta)n\rho}_{\Omega_k}, \cM^{(1-\eta)n\rho}_{\Omega_k}} \les \diam(\Omega_k)^p  |\cN^{(1-\eta)n\rho}_{\Omega_k}|$,  to estimate
 \[\begin{split}  n^{\frac{p}{d}-1}\sum_{\Omega_k \in \cR_\delta}  \EE\sqa{ \CPp\bra{ \cN^{(1-\eta)n\rho}_{\Omega_k}, \cM^{(1-\eta)n\rho}_{\Omega_k}}  } & \les  n^{\frac{p}{d}-1} \delta^p \sum_{\Omega_k \in \cR_\delta}  \EE\sqa{ |\cN^{(1-\eta)n\rho}_{\Omega_k}|}  \\
 & \les  n^{\frac{p}{d}-1} \delta^p \cdot \delta^{1-d} \cdot n \delta^d = n^{-\gamma + \frac{p}{d}(1-d\gamma)}.
 \end{split}\]
 This tends to zero provided $\gamma d >p/(p+1)$ which is in particular true if (recall that $p<d$)
 \begin{equation}\label{eq:gamma-second-condition}
 \gamma d > d/(d+1).
 \end{equation}
Notice that this  condition is compatible with \eqref{eq:gamma-first-condition}. Under condition \eqref{eq:gamma-second-condition} we thus have 
\begin{multline*}
 \limsup_{n\to \infty} n^{1-\frac{p}{d}} \sum_{k=1}^K  \EE\sqa{ \CPp\bra{ \cN^{(1-\eta)n\rho}_{\Omega_k}, \cM^{(1-\eta)n\rho}_{\Omega_k}}  } \\
 =\limsup_{n\to \infty}  \sum_{\Omega_k\in \cQ_\delta} n^{1-\frac{p}{d}} \EE\sqa{ \CPp\bra{ \cN^{(1-\eta)n\rho}_{\Omega_k}, \cM^{(1-\eta)n\rho}_{\Omega_k}}  }. 
\end{multline*}
Since every $\Omega_k\in \cQ_\delta$ is a cube, we may combine  \eqref{eq:limsup-poisson} in $\Omega_k$ together with the  precise limit procedure, 
justified by  the domination given in \eqref{eq:inequality-cube-holder-general} (this is why each cube $Q_i$ in the Whitney partition has side length $r<C$), to obtain 
 \begin{equation*}\label{eq:main-contribution-G-delta}
 \limsup_{n\to \infty} n^{1-\frac{p}{d}} \sum_{k=1}^K  \EE\sqa{ \CPp\bra{ \cN^{(1-\eta)n\rho}_{\Omega_k}, \cM^{(1-\eta)n\rho}_{\Omega_k}}  } 
  \le (1-\eta)^{1-\frac{p}{d}} \int_\Omega \rho^{1-\frac{p}{d}}.\end{equation*}

\noindent We now turn to  the right-hand side of \eqref{eq:exp-sub-explicit-2}. The last term is easily estimated using directly \eqref{eq:whitney-general-q} with $\alpha=p$. 
In particular, if $p<d-1$ we  notice that 
\[
 n^{\frac{p}{d}-1}  \sum_{k=1}^K  \diam(\Omega_k)^p\les n^{\frac{p}{d}-1} \delta^{1-(d-p)}=(n\delta^d)^{-(1-\frac{p}{d})} \delta
\]
which goes to zero if \eqref{eq:gamma-first-condition} holds.
%

We finally estimate   the first term in the right-hand side of  \eqref{eq:exp-sub-explicit-2}.  We argue as in  \eqref{eq:last-term-finite-partition} and \eqref{eq:last-term-easy-finite-partition} 
which we combine with \cref{prop:density-helps-matching-partition} 
to obtain that for every $\eps>0$,
\begin{multline}\label{almostpartition}
 n^{1-\frac{p}{d}}\EE\sqa{ \CPp\bra{ \cN^{\eta n\rho}\cup \cU, \cM^{\eta n\rho}\cup \cV}}\les  \EE\sqa{|\cU|/n}^{1-\frac{p}{d}}+\eta^{1-\frac{p}{d}}\\
 +C(\eta,\eps,\gamma) n^{\eps}\bra{ \bra{\max\cur{n^{\frac{p}{d}}\delta^{p+1},n^{\frac{2}{d}} \delta^3}}^{\alpha} +\bra{n\delta^d}^{-\beta}}.
\end{multline}
Using \eqref{eq:whitney-general-q} with $\alpha = d/2<d-1$ we have 
\[
 \EE\sqa{|\cU|/n}^{1-\frac{p}{d}}\les   \bra{\sum_{k=1}^K (|\Omega_k|/n)^{\frac{1}{2}}}^{1-\frac{p}{d}}  \les n^{\frac{1}{2}} \delta^{1-\frac{d}{2}} = \bra{\delta (n\delta^d)^{-\frac{1}{2}}}^{1-\frac{p}{d}}. \]
Under condition  \eqref{eq:gamma-first-condition} this term goes to zero. Regarding the term inside brackets in \eqref{almostpartition} we notice that if $q=\max\cur{p,2}$, then under condition \eqref{eq:gamma-first-condition},
\[
 \max\cur{n^{\frac{p}{d}}\delta^{p+1},n^{\frac{2}{d}} \delta^3}=n^{-\gamma+ \frac{q}{d}(1-d\gamma)}.
\]
In particular, as above  this term goes to zero under condition \eqref{eq:gamma-second-condition}.\\
We can thus choose first $\gamma$ satisfying both \eqref{eq:gamma-first-condition} and \eqref{eq:gamma-second-condition} and then $\eps=\eps(\alpha,\beta,\gamma,p)>0$ such that 
\[
 \lim_{n\to \infty}  n^{\eps}\bra{ \bra{\max\cur{n^{\frac{p}{d}}\delta^{p+1},n^{\frac{2}{d}} \delta^3}}^{\alpha} +\bra{n\delta^d}^{-\beta}}=0. 
\]
With this choice we find 
\[
 \limsup_{n\to \infty }n^{1-\frac{p}{d}}\EE\sqa{ \CPp\bra{ \cN^{\eta n\rho}\cup \cU, \cM^{\eta n\rho}\cup \cV}}\les \eta^{1-\frac{p}{d}},
\]
from which we conclude the proof of \eqref{eq:limsup-poisson} after sending $\eta\to 0$.

\subsection{Uniform density on a domain} In this last case, we assume that $\Omega$ is a bounded domain with $C^2$ boundary and $\rho = I_{\Omega}/|\Omega|$ is uniform. 
After a simple rescaling, it is more convenient to argue with Poisson point processes $\cN^n_{\Omega}$, $\cM^{n}_\Omega$ with constant intensity $n$ (on $\Omega$) so that the thesis reduces to 
\[ \lim_{n \to \infty} n^{\frac{p}{d}-1}\EE\sqa{ \CPp\bra{ \cN^{n}_{\Omega}, \cM^{n}_{\Omega}}}  =  \beta_{\pP} |\Omega|.\]
Since the boundary of $\Omega$ is $C^2$, we can apply the result from the  previous section and obtain the upper bound
\[ \limsup_{n \to \infty} n^{\frac{p}{d}-1}\EE\sqa{ \CPp\bra{ \cN^{n}_{\Omega}, \cM^{n}_{\Omega}}}  \le \beta_{\pP} |\Omega|.\]

To prove the corresponding lower bound, we follow closely the argument of \cite[Theorem 24]{BaBo}: we fix a cube $Q$ sufficiently large so that $\Omega \subseteq Q$ and introduce a Poisson point process
$\cN^n_{Q}$ with intensity $n$ on $Q$. For $k=2,\ldots, K$, let  $\Omega_k$ be the connected components of $Q\backslash \Omega$ so that  $Q\backslash \Omega=\cup_{k=2}^K\Omega_k$. Notice that for every $k$  either $\partial \Omega_k$ is $C^2$ or is the union of $\partial Q$ and a $C^2$ surface. 
In particular each $\Omega_k$ satisfies 
\eqref{eq:green-kernel-bound}.  Using \eqref{eq:liminf-poisson-finite-partition} with the decomposition $Q = \Omega \cup \bigcup_{k=2}^K \Omega_k$, we obtain
\[ \begin{split} \beta_{\pP} |Q|=\liminf_{n\to \infty}n^{\frac{p}{d}-1}\EE\sqa{ \CPp\bra{ \cN^{n}_{Q}, \cM^{n}_Q}}  & \le  \liminf_{n \to \infty} n^{\frac{p}{d}-1}\EE\sqa{ \CPp\bra{ \cN^{n}_{\Omega}, \cM^{n}_{\Omega}}}  \\
& \quad + \sum_{k=2}^K\limsup_{n \to \infty} n^{\frac{p}{d}-1}\EE\sqa{ \CPp\bra{ \cN^{n}_{ \Omega_k}, \cM^{n}_{ \Omega_k}}}.
\end{split}\]
Now for every $k$, using \eqref{eq:limsup-poisson} we have 
\[ \limsup_{n \to \infty} n^{\frac{p}{d}-1}\EE\sqa{ \CPp\bra{ \cN^{n}_{ \Omega_k}, \cM^{n}_{ \Omega_k}}}  \le  \beta_{\pP}|\Omega_k|.\]
Therefore,
\[ \liminf_{n \to \infty} n^{\frac{p}{d}-1}\EE\sqa{ \CPp\bra{ \cN^{n}_{\Omega}, \cM^{n}_{\Omega}}}   \ge \beta_{\pP} |Q| - \beta_\pP \sum_{k=2}^K |\Omega_k| = \beta_{\pP} | Q|,\]
which is the desired conclusion.

\section{Proof of main result}\label{sec:main}

From \cref{thm:limit-poisson}, we deduce our main result \cref{thm:main}. We follow a relatively standard strategy,  using de-Poissonization and concentration of measure arguments with the 
 necessary adjustments to deal with our setting. First, we argue that \cref{thm:limit-poisson}  yields similar convergence in the case of a deterministic number of independent points. 

\begin{proposition}\label{prop:limit-mean-iid}
Let $d \ge 3$, $p\in [1,d)$ and let $\pP = (\cF_{n,n})_{n \in \mathbb{N}}$ be a combinatorial optimization problem over complete bipartite graphs such that assumptions \ref{as:isomorphism}, \ref{as:spanning}, \ref{as:bddegree}, \ref{ass:local-merging} and  \ref{as:growth} hold. Then, with $\beta_{\pP}\in (0, \infty)$ given by
\cref{thm:limit-poisson} the following hold.

Let $\Omega \subseteq \R^d$ be a bounded domain with Lipschitz boundary and such that \eqref{eq:green-kernel-bound} holds and let $\rho$ be a H\"older continuous probability density on $\Omega$, uniformly strictly positive and bounded from above. 

Given i.i.d.\ random variables $(X_i)_{i=1}^\infty$, $(Y_j)_{j=1}^\infty$ with common law $\rho$, we have 
\begin{equation}\label{eq:limit-mean-main-iid}
\limsup_{n \to \infty} n^{\frac{p}{d}-1} \EE\sqa{ \CPp\bra{ (X_i)_{i=1}^n, (Y_j)_{j=1}^n}}  \le \beta_{\pP} \int_{\Omega} \rho^{1-\frac{p}{d}}.\end{equation}

Moreover, if $\rho$ is the uniform density and $\Omega$  is either  a cube or has $C^2$ boundary, the limit exists and is equal to the right-hand side. 
\end{proposition}
\begin{remark}\label{rem:phom}
 The only properties we used to established \cref{prop:limit-mean-iid} are the subadditivity property \eqref{eq:sub}, the growth condition \eqref{eq:upper-bound-deterministic} as well as the $p-$homogeneity of the problem. In particular it holds for every bipartite $p-$homogeneous functional $\C$   satisfying
 \begin{itemize}
  \item  For  every $\Omega\subset \R^d$ and every partition $\Omega = \cup_{k=1}^K \Omega_k$,  $K \in \mathbb{N}$, if $\x^0$, $\y^0 \subseteq \Omega$ are such that $\min\cur{|\x^0|, |\y^0|} \ge \max\cur{\cspan, K}$,
 for every $k =1, \ldots, K$,  $\x^k$, $\y^k \subseteq \Omega_k$ are such that  $|\x^k| = |\y^k| =n_k$, with either $n_k \ge \cspan$ or $n_k = 0$ and $\z=(z_k)_{k=1}^K$ with $z_k \in \Omega_k$, for every $k=1,\ldots, K$ then 
 \begin{equation}\label{Sp}  \C\bra{ \x^0 \cup \bigcup_{k=1}^K \x^k, \y^0 \cup  \bigcup_{k=1}^K \y^k }  - \sum_{k=1}^K \C(\x^k, \y^k) 
  \lesssim  \C(\x^0, \y^0)+  \Mp(\z, \x^0)+  \sum_{k=1}^K   \diam(\Omega_k)^p.
 \end{equation}
 \item There exists $\creg \ge 0$ such that, for every $\x, \y \subseteq (0,1)^d$, we have
\begin{equation} \label{Rp} \C(\x, \y) \le \creg \bra{ \min\cur{|\x|^{1-\frac{p}{d}}, |\y|^{1-\frac{p}{d}}}+ \Mp(\x,\y)}.\end{equation}
 \end{itemize}

\end{remark}

\begin{proof}
The proof is similar to the proof of \cref{lem:subadditivity-finite-partition}. 
We set $\x=(X_i)_{i=1}^{n}$ and $\y=(Y_j)_{j=1}^{n}$. 
Let $\eta\in(0,1/2)$ and consider two independent copies $\cN^{(1-\eta)n\rho}$ and $\cM^{(1-\eta)n\rho}$ of Poisson point processes with intensity $(1-\eta)n\rho$ on $\Omega$. 
We claim that 
\begin{equation}\label{claimdepoi}
 \limsup_{n\to \infty}  n^{\frac{p}{d}-1} \EE\sqa{ \CPp\bra{ \x, \y}}- \limsup_{n\to \infty} n^{\frac{p}{d}-1} \EE\sqa{ \CPp\bra{ \cN^{(1-\eta)n\rho}, \cM^{(1-\eta)n \rho}}}\les \eta^{1-\frac{p}{d}}.
\end{equation}
 By \cref{thm:limit-poisson}, this would conclude the proof of \eqref{eq:limit-mean-main-iid} since $\eta$ is arbitrary. We introduce the random variables $N=\max\cur{n-|\cN^{(1-\eta)n\rho}|,0}$ and $M=\max\cur{n-|\cN^{(1-\eta)n\rho}|,0}$
and notice that by the concentration properties of Poisson random variables, also $N$ and $M$ have the concentration property. Moreover, the event 
\[
 A=\cur{ |N-\eta n|\le \eta n/2}\cap \cur{|M- \eta n|\le \eta/2}
\]
is of overwhelming small probability and thus arguing exactly as in the proof of \cref{lem:subadditivity-finite-partition} we have 
\[
 \limsup_{n\to \infty}  n^{\frac{p}{d}-1}  \EE\sqa{ \CPp\bra{ \x, \y}}=\limsup_{n\to \infty} n^{1-\frac{p}{d}} \EE\sqa{ \C_\pP( \x,  \y ) I_A}.
\]
We let $\cN= (X_i)_{i=n-N+1}^n$ and $\cM=(Y_j)_{j=n-M+1}^n$ so that  in $A$, $\x=\cN^{(1-\eta)n\rho}\cup \cN$, $\y=\cM^{(1-\eta)n\rho}\cup \cM$ and  $\min\cur{|\cN|,|\cM|}\ges \eta n$. \\
In $A$ we let $\x^1\subset \cN^{(1-\eta)n\rho}$ and $\y^1\subset \cM^{(1-\eta)n\rho}$ be such that $|\x^1|=|\y^1|$ and 
$$ \C_\pP( \cN^{(1-\eta)n\rho},  \cM^{(1-\eta)n\rho} ) = \C_\pP(\x^1, \y^1 ).$$
We then set $\cU=\cN^{(1-\eta)n\rho}\backslash \x^1$, $\cV=\cM^{(1-\eta)n\rho}\backslash \y^1$, $\x^0=\cU\cup \cN$ and $\y^0=\cV\cup \cM$. 
Using \cref{lem:sub} on $\Omega$ with  $K=1$, i.e.\ a trivial partition, we find that in $A$,
\begin{multline*}
 \C_\pP(\x,\y)-\C_\pP( \cN^{(1-\eta)n\rho},  \cM^{(1-\eta)n\rho} )\les \C_\pP(\x^0,\y^0) +1\\
 \stackrel{\eqref{eq:upper-bound-deterministic}}{\les} \min\cur{ |\x^0|^{1-\frac{p}{d}}, |\y^0|^{1-\frac{p}{d}} } + \Mp(\x^0, \y^0)  + 1.
\end{multline*}
 Multiplying by $I_A$, taking expectation and arguing exactly as in \eqref{tofinishfinite} (using in particular \cref{prop:density-helps-matching}) we conclude the proof of  \eqref{claimdepoi}.\\

With a similar argument one can prove that 
\[
 \liminf_{n\to \infty} n^{\frac{p}{d}-1} \EE\sqa{ \CPp\bra{ \cN^{n\rho}, \cM^{n \rho}}}\le \liminf_{n\to \infty} n^{\frac{p}{d}-1} \EE\sqa{ \CPp\bra{ \x, \y}},
\]
which concludes the proof of \cref{prop:limit-mean-iid}.
\end{proof}

To conclude the proof of \cref{thm:main}, we prove a concentration bound, which improves \eqref{eq:limit-mean-main-iid} to complete convergence. The argument requires minimal assumptions on the combinatorial optimization problem and relies essentially on the validity 
of a Poincar\'e inequality.

\begin{proposition}\label{prop:concentration}
Let $d \ge 3$, $p\in [1,d)$ and let $\pP = (\cF_{n,n})_{n \in \mathbb{N}}$ be a combinatorial optimization problem over complete bipartite graphs such that assumptions  \ref{as:bddegree} and  \ref{as:growth} hold. Let $\Omega \subseteq \R^d$ be a bounded domain with Lipschitz boundary and let $\rho$ be a probability density on $\Omega$, uniformly strictly positive and bounded from above. Let $(X_i)_{i=1}^\infty$, $(Y_j)_{j=1}^\infty$ be i.i.d.\ random variables with common law $\rho$.

For every $q \ge 2$ and  $\varepsilon>0$,
\begin{equation} \label{eq:concentration-cost-n}
\PP\bra{ n^{\frac{p}{d}-1}\abs{ \CPp\bra{ (X_i)_{i=1}^n, (Y_j)_{j=1}^n} - \EE\sqa{\CPp\bra{ (X_i)_{i=1}^n, (Y_j)_{j=1}^n}}}  > \varepsilon} \le_q \frac{1}{ \varepsilon^q n^{\frac{\alpha q}{2}}},\end{equation}
with 
$$ \alpha = \begin{cases} 1-2/d & \text{if $p \in [1,2)$,}\\
1-p/d & \text{if $p \ge 2$.}\end{cases}$$
In particular,  complete (hence $\PP$-a.s.) convergence holds:
$$ \lim_{n \to \infty} n^{\frac{p}{d}-1}\abs{ \CPp\bra{ (X_i)_{i=1}^n, (Y_j)_{j=1}^n} - \EE\sqa{\CPp\bra{ (X_i)_{i=1}^n, (Y_j)_{j=1}^n}}}  = 0.$$
\end{proposition}

\begin{remark}[Poincar\'e inequality]
We first recall that for every  Lipschitz function $F: \Omega^{2n} \to \R$ we have the following $L^q$-Poincar\'e inequality,
\begin{equation}
\label{eq:poincare-q} \EE\sqa{ \abs{ F\bra{ (X_i)_{i=1}^n, (Y_j)_{j=1}^n} - \EE\sqa{ F\bra{ (X_i)_{i=1}^n, (Y_j)_{j=1}^n }}}^q } \les_q \EE\sqa{ \abs{ \nabla F \bra{ (X_i)_{i=1}^n, (Y_j)_{j=1}^n}}^q}.
\end{equation}
Here  $\abs{\nabla F}$  denotes the usual Euclidean norm of the gradient. We stress the fact that the implicit constant in \eqref{eq:poincare-q} does not depend upon $n$.

Inequality \eqref{eq:poincare-q} is a consequence of well-known facts: first, the assumptions on $\Omega$ yield the $L^2$-Poincar\'e inequality with respect to the uniform measure,
$$ \int_{\Omega} \abs{ u - \frac{1}{|\Omega|}\int_\Omega u }^2 \les \int_\Omega |\nabla u|^2. $$
Using  that the constant $c = \int_{\Omega} u \rho$ minimizes $\int_{\Omega} \abs{ u - c }^2 \rho$ and that $\rho$ is bounded from above and below, and, we obtain the weighted version
$$ \int_{\Omega} \abs{ u - \int_\Omega u \rho }^2 \rho \le \int_{\Omega} \abs{ u - \frac{1}{|\Omega|}\int_\Omega u }^2\rho \le C \int_\Omega |\nabla u|^2 \rho.$$
for some $C = C(\rho, \Omega) \in (0, \infty)$. Then, a standard tensorization argument \cite[Corollary 5.7]{ledoux2001concentration} entails that the inequality holds also on the product space $\Omega^{2 n}$, endowed with the product measure $\rho^{\otimes 2n}$, with the same constant $C$. This yields \eqref{eq:poincare-q} with $q=2$. 

The general case $q\ge 2$ follows finally from  the chain rule. Preliminarily, we notice that if $\mu$ is a probability measure on $\R^D$, then the validity of the inequality
\begin{equation}\label{eq:poincare-mean} \int \abs{ u - \int u d \mu}^q d \mu \les \int \abs{\nabla u}^qd \mu  \end{equation}
for every Lipschitz function $u: \R^D \to \R$ is equivalent to
\begin{equation}\label{eq:poincare-median}  \int \abs{ u - m_u}^q d \mu \les \int \abs{\nabla u}^qd \mu, \end{equation}
where $m_u$ denotes a median of (the law) of $u$, i.e.\ any $m \in \R$ such that $\mu( u \le m) \ge 1/2$ and $\mu( u \ge m) \ge 1/2$. Indeed,
\[ \abs{ m_u - \int u d \mu} \le \abs{ \int (m_u - u) d \mu } \le \int \abs{ m_u - u} d \mu.\]
Since $m_u$ can be characterized as a minimizer for $c \mapsto \int \abs{ u - c}d\mu$, we also have
\[ \int \abs{ m_u - u} d \mu \le  \int \abs{ u - \int u d \mu} d \mu.\]
Using Jensen's inequality, we obtain
\[ \abs{ m_u - \int u d \mu}^q \le \min \cur{ \int \abs{ m_u - u}^q d \mu, \int \abs{ u - \int u d \mu}^q d \mu.}\]
Then, assuming that \eqref{eq:poincare-mean} or \eqref{eq:poincare-median} holds, using the triangle inequality and the bound above, we obtain the validity of the other inequality. 

To conclude, we assume that \eqref{eq:poincare-median} holds for $q=2$ and argue that it also holds for any $q \ge 2$. Up to adding a suitable constant, we can assume that $u$ is Lipschitz with $m_u = 0$. We then consider the Lipschitz function $v = \abs{u}^{q/2} \operatorname{sign}(u)$ (recall that in our case the support of $\mu = \rho^{\otimes 2n}$ is bounded, hence we can assume that also $u$ is bounded), so that $m_v = 0$ and apply the $q=2$ case of \eqref{eq:poincare-median}:
\[\begin{split} \int \abs{u}^q d \mu & = \int \abs{v}^2 d \mu \les \int |\nabla v|^2 d \mu \les \int \abs{u}^{q-2} |\nabla u|^2 d \mu \\
& \les \bra{ \int \abs{u}^q d \mu}^{1-2/q} \bra{ \int \abs{\nabla u}^q d \mu}^{2/q} .
\end{split}\]
Dividing both sides by $\bra{ \int \abs{u}^q d \mu}^{1-2/q}$ yields the desired conclusion.
\end{remark}
\begin{proof}[Proof of \cref{prop:concentration}]
The second statement follows choosing $q$ sufficiently large in \eqref{eq:concentration-cost-n} so that the right-hand side in \eqref{eq:concentration-cost-n} is summable.  
We thus focus on the proof of \eqref{eq:concentration-cost-n}. Given a feasible  $G \subseteq \cK_{n,n}$, i.e., $G \in \cF_{n,n}$, and $\x = (x_i)_{i=1}^n$, $\y = (y_j)_{j=1}^n \subseteq \Omega$,  write 
$$ w_G(\x, \y) = \sum_{ \cur{(1,i), (2,j)} \in E_G } | x_i - y_j|^p.$$
Since $p \ge 1$, $w_G$ is Lipschitz with a.e.\ derivative given by 
\[ \nabla_{x_i} w_G(\x, \y) =  \sum_{ (2,j) \in \cN_G((1,i))} p |x_i-y_j|^{p-2} (x_i-y_j),\]
and
\[ \nabla_{y_j} w_G(\x, \y) =  -\sum_{ (1,i) \in \cN_G((1,j))} p |x_i-y_j|^{p-2} (x_i-y_j).\]
Notice also that $w_G$ is differentiable at every $(\x,\y)$ such that $x_i\neq y_j$ for every $i,j$. Since $G \in \cF_{n,n}$, assumption \ref{as:bddegree} yields that the sums above contain at most $\c_3$ terms, hence we bound, using Cauchy-Schwarz inequality,
$$ \abs{ \nabla_{x_i} w_G(\x, \y)}^2 \les \sum_{ (2,j) \in \cN_G((1,i))} |x_i-y_j|^{2(p-1)},$$
and similarly
$$ \abs{ \nabla_{y_j} w_G(\x, \y)}^2  \les \sum_{ (2,j) \in \cN_G((1,i))} |x_i-y_j|^{2(p-1)}.$$
Summing upon $i$ and $j\in\cur{1,\ldots, n}$, we obtain, for the Euclidean norm of the gradient, the inequality
$$ \abs{ \nabla w_G(\x, \y)}^2  \les  \sum_{ \cur{(1,i),(2,j)} \in E_G} |x_i-y_j|^{2(p-1)}.$$
If $p \ge 2$, we simply bound each term $|x_i-y_j|^{2(p-1)} \le \diam(\Omega)^{p-2} |x_i-y_j|^{p}$, obtaining
$$  \abs{ \nabla w_G(\x, \y)}^2  \les \sum_{ \cur{(1,i),(2,j)} \in E_G} |x_i-y_j|^{p}  = w_G(\x, \y).$$
If $p \in [1,2)$, we use H\"older inequality and the fact that $|E_G| \les n$ (again by assumption \ref{as:bddegree}), to obtain
$$ \abs{ \nabla w_G(\x, \y)}^2  \les \bra{ \sum_{ \cur{(1,i),(2,j)} \in E_G} |x_i-y_j|^{p}}^{\frac{1}{r}} n^{1-\frac{1}{r}} = w_G(\x, \y)^{\frac{1}{r}} n^{1-\frac{1}{r}},$$
with $r = p/(2(p-1))$.

\noindent Using the trivial bound $w_G(\x,\y) \les n$, it follows in particular that each $w_G(\x, \y)$ has a Lipschitz constant bounded independently of $G$ (although the bound depends upon $n$). Therefore, also
$$ \CPp\bra{ \x, \y} = \inf_{G \in \cF_{n,n}} w_G(\x, \y),$$
is Lipschitz, hence differentiable at Lebesgue a.e.\ $(\x, \y)$, by Rademacher theorem. 
Let $(\x,\y)$ be a point of differentiability for both $w_G$ and $\CPp$ (which holds for Lebesgue a.e. point).  Let $G = G(\x, \y) \in \cF_{n,m}$ be any minimizer for the problem on the graph $\cK(\x,\y)$ (which is a.e. unique if $p>1$ by Remark \cref{rem:uniqueness-minimizer}).
For every $(\x', \y')$, we have the inequality
\[
 \CPp(\x',\y')\le w_G(\x', \y'),
\]
with equality at $(\x,\y)$, hence we obtain the identities,
\begin{equation}\label{eq:identity-derivative-cpp}
 \nabla_{x_i} \CPp(\x,\y)= \nabla_{x_i} w_G(\x, \y), \quad \nabla_{y_j} \CPp(\x,\y)= \nabla_{y_j} w_G(\x, \y).
\end{equation}
Therefore,
\[ \abs{ \nabla  \CPp\bra{ \x, \y} }^2     \les \begin{cases} \CPp\bra{ \x, \y}^{\frac{1}{r}} n^{1-\frac{1}{r}} & \text{if $p \in [1,2)$,}  \\
 \CPp\bra{ \x, \y} & \text{if $p \ge 2$.}
 \end{cases}
\]
If now $\x=(X_i)_{i=1}^n$ and $\y=(Y_j)_{j=1}^n$, combining this with \eqref{eq:poincare-q} and \eqref{eq:upper-bound-deterministic} yields  
\[ \begin{split} \EE&\sqa{ \abs{ \CPp\bra{ \x, \y } - \EE\sqa{ \CPp\bra{ \x, \y}}}^q } \les_q	 \EE\sqa{ \abs{ \nabla \CPp\bra{ \x, \y}}^{q}}\\
&\qquad \qquad  \qquad \qquad \les \begin{cases} \vspace{1em}
   \EE\sqa{ \bra{ n^{1-\frac{p}{d}}+  \Mp \bra{ \x, \y}}^{\frac{q}{2r}}} n^{(1-\frac{1}{r})\frac{q}{2}}
& \text{if $p \in [1,2)$,}\\
  \EE\sqa{ \bra{n^{1-\frac{p}{d}}+ \Mp \bra{ \x, \y}}^{\frac{q}{2}}} & \text{if $p\ge 2$.}
\end{cases}
\end{split}\]
 By the equivalence between $\Mp$ and $\W$ (recall \eqref{eq:wass-equals-matching}), the triangle inequality \eqref{eq:triangle} and \eqref{eq:jensen} and finally using \eqref{eq:matching-iid} with $qp$ instead of $p$, we bound from above
$$ \EE\sqa{ \bra{ \Mp \bra{ (X_i)_{i=1}^n, (Y_j)_{j=1}^n}}^{\frac{q}{2}}} \les n^{(1-\frac{p}{d})\frac{q}{2}}.$$
If $p \ge 2$, we conclude at once that
$$ \EE\sqa{ \abs{ \CPp\bra{ \x, \y } - \EE\sqa{ \CPp\bra{ \x, \y}}}^q } \les_q n^{(1-\frac{p}{d})\frac{q}{2}},$$
hence \eqref{eq:concentration-cost-n} by Markov inequality. If $p \in [1,2)$, we bound similarly and obtain, after simple computations,
$$ \EE\sqa{ \abs{ \CPp\bra{ \x, \y } - \EE\sqa{ \CPp\bra{ \x, \y}}}^q } \les n^{\bra{ (1-\frac{p}{d})(1-\frac{1}{p})+\frac{1}{p}-\frac{1}{2}}q},$$
which leads to the corresponding case of \eqref{eq:concentration-cost-n} by Markov inequality.
\end{proof}


\begin{remark}[uniqueness of minimizers]\label{rem:uniqueness-minimizer}
If $p>1$, for Lebesgue a.e.\ $(\x, \y)$, the minimizer $G \in \cF_{n,m}$ for the problem on $\cK(\x, \y)$ is unique. 
This in particular yields that it is unique a.s., when $\x = (X_i)_{i=1}^n$, $\y = (Y_j)_{j=1}^m$ are random i.i.d.\ with a common density $\rho$. 
For simplicity, we argue in the case of $|\x| = |\y|$ only, but the same result holds in general. 

Let $(\x, \y)$ be a differentiability point for $\CPp(\x, \y)$ with $X_i\neq Y_j$ for every $i,j$. Notice that by the previous proof this holds a.s. .
Let $G, G' \in \mathcal{F}_{n,n}$ be minimizers for the problem on $\cK(\x, \y)$, so that by \eqref{eq:identity-derivative-cpp} we obtain  that, for every $i\in [n]$, $\nabla_{x_i} w_G (\x, \y) =  \nabla_{x_i} w_{G'} (\x, \y)$, i.e.,
$$ \sum_{ (2,j) \in \cN_G((1,i))} \abs{x_i - y_j}^{p-2} (x_i-y_j)  = \sum_{ (2,j) \in \cN_{G'}((1,i))} \abs{x_i - y_j}^{p-2} (x_i-y_j) $$
Assuming that $E_G \neq E_{G'}$, we can find $i$, $j \in [n]$ such that $(2,j) \in \cN_{G'}((1,i)) \setminus \cN_G((1,i))$ (up to exchanging the roles of $G$ and $G'$). Then, 
\begin{equation}\label{eq:identity-direct}\begin{split} 
 \abs{x_i-y_j}^{p-2}(x_i-y_j)= & \sum_{(2,k) \in \cN_{G}((1,i)} \abs{x_i-y_k}^{p-2}(x_i-y_k) \\
 & - \sum_{ (2,k)\in \cN_{G'}((1,i)) \setminus \cur{(2,j)} } \abs{x_i-y_k}^{p-2}(x_i-y_k).
 \end{split}
\end{equation}
We notice that the right-hand side above is a function $U(\x, \y)$ which however does not depend on the variable $y_j$. The map $$z \in \R^d \mapsto \abs{z}^{p-2} z \in \R^d$$
is invertible, with a Borel inverse which we denote by $f$,  hence  we can rewrite \eqref{eq:identity-direct} equivalently as the identity
$$ y_j = x_i - f\bra{ U (\x, \y)},$$
where right-hand side is a Borel function of $(\x, \y)$ which does not depend on $y_j$. This identity however cannot hold on a set of positive Lebesgue measure.
%
\end{remark}

\section{Bounds for the Euclidean assignment problem}\label{sec:ot}

In this section we establish some novel upper bounds for the random Euclidean assignment problem, in the case of not necessarily i.i.d.\ uniformly distributed points. 

\subsection{Matching of i.i.d.\ points} 
We begin with a general upper bound for the Wasserstein distance between the empirical measure of i.i.d.\ points and the corresponding common law when $d\ge 3$ and $p\ge 1$. As a consequence, we also obtain a similar bound for the Euclidean assignment problem.
We derive the general case of a H\"older continuous
law bounded above and below on an open connected set with Lipschitz boundary $\Omega$ from the case of the uniform law on a cube $Q \subseteq \R^d$. In that case,  
it is a well-known result, marginally discussed in \cite{AKT84}, where the  focus is on the $d=2$ case. 
However, we point out that the case $d\ge 3$, $p\ge d/2$ was, to our knowledge, not explicitly covered in the literature until the proof provided by \cite{Le17}, 
which clearly extends to any $p\neq 2$ (see also  \cite{goldman2021convergence}).

\begin{proposition}\label{prop:matching-iid}
Let $d \ge 3$, $p \ge 1$ and $\Omega$ be a bounded connected open set with Lipschitz boundary. For every H\"older continuous density $\rho: \Omega\mapsto \R$ bounded above and below and independent sequences 
$(X_{i})_{i=1}^\infty$, $(Y_j)_{j=1}^\infty$ of i.i.d.\ random variables with common law $\rho$,
\begin{equation}\label{eq:transport-iid}
 \EE\sqa{  \W\bra{\sum_{i=1}^n \delta_{X_i}, n\rho } } \les |\Omega|^{\frac{p}{d}} n^{1-\frac{p}{d}},                                 
\end{equation}
and therefore
\begin{equation}\label{eq:matching-iid}
 \EE\sqa{  \Mp\bra{ (X_i)_{i=1}^n, (Y_j)_{j=1}^m} } \les |\Omega|^{\frac{p}{d}} \min\cur{n,m}^{1-\frac{p}{d}}.                                
\end{equation}
\end{proposition}
\begin{proof}
Inequality \eqref{eq:matching-iid} follows  from \eqref{eq:transport-iid} assuming e.g.\ $n \le m$ and  \eqref{eq:matching-below-wasserstein} with $\lambda = \rho$.
Hence, we focus on the proof of \eqref{eq:transport-iid}.  By Jensen inequality \eqref{eq:jensen}, it is enough to prove this bound for large $p$ so that we may assume without loss of generality that $p>d/(d-1)$.
 We then set $\mu= \frac{1}{n} \sum_{i=1}^n \delta_{X_i}$. 
 
  We first prove the statement in the case $\Omega=Q$ is a cube. By scaling we may assume that $Q=(0,1)^d$ is the unit cube.
 By Proposition \ref{prop:map-heat-semigroup}, there is a bi-Lipschitz map $T: Q\mapsto Q$ with Lipschitz constant depending
 only on $\rho$ such that $T\sharp \rho=1$. Then,  $X_i'=T(X_i)$ are i.i.d.\ uniformly distributed on $Q$
  Letting $\mu'=T\#\mu=\frac{1}{n}\sum_{i=1}^n \delta_{X'_i}$ we have 
  \[
  \W(\mu, \rho)\les \W(\mu',1) 
 \]
 and the statement follows from \cite{Le17}.
 
 Consider now $\Omega$ a general  bounded connected open set with Lipschitz boundary. We say that $\Omega$ is well-partitioned if there exists convex polytopes $(\Omega_k)_{k=1}^K$  covering $\Omega$, with $|\Omega_k\cap \Omega_{k'}|=0$ for $k\neq k'$ and  
 such that each $\Omega_k$ is bi-Lipschitz homeomorphic to a cube. By \cite{Garcia-Slepcev},  every connected and Lipschitz domain is bi-Lipschitz homeomorphic to a well-partitioned and
 smooth domain so that arguing exactly as above we may assume that $\Omega$ itself is smooth and well-partitioned.
Let $T_k: \Omega_k\mapsto Q_k$ be Lipschitz homeomorphisms between $\Omega_k$ and some cubes $Q_k$. We then define $\rho_k=T_k\#\rho$, $n_k=\mu(\Omega_k)$ and $\mu_k= \frac{n}{n_k}T_k\#\mu$.
Notice in particular that we may
write $\mu_k=\frac{1}{n_k}\sum_{i=1}^{n_k} \delta_{Y_i}$ where $(Y_i)_{i=1}^\infty$ are i.i.d.\ with common law $\rho_k/\rho_k(Q_k)$ and that $n_k$ is a Binomial random variable with parameters $n$ and $\rho_k(Q_k)=\rho(\Omega_k)$.
Using \eqref{eq:mainsub} with $\eps=1$ we thus find
\[ \begin{split}
 \W(\mu,\rho) & \les \sum_{k=1}^K \W_{\Omega_k}\bra{\mu, \frac{n_k}{n\rho(\Omega_k)}\rho} +\W\bra{\sum_{k=1}^K \frac{n_k}{n \rho(\Omega_k)} \rho I_{\Omega_k},\rho}\\
 & \stackrel{\eqref{eq:estimCZ}}{\les} \sum_{k=1}^K  \frac{n_k}{n}\W_{Q_k}\bra{\mu_k, \rho_k} +\nor{\sum_{k=1}^K \bra{\frac{n_k}{n \rho(\Omega_k)}-1} I_{\Omega_k} \rho}_{W^{-1,p}(\Omega)}^p\\
 & \stackrel{\eqref{eq:Lp}}{\les}\sum_{k=1}^K  \frac{n_k}{n}\W_{Q_k}\bra{\mu_k, \rho_k} +\sum_{k=1}^K |\Omega_k| \lt|\frac{n_k}{n \rho(\Omega_k)}-1\rt|^p.
\end{split}\]
Taking the expectation and using the concentration properties of binomial random variables \eqref{eq:binomial-concentration} we find 
\[
 \EE\sqa{\W(\mu,\rho)}\les \sum_{k=1}^K\EE\sqa{ \frac{n_k}{n}\W_{Q_k}\bra{\mu_k, \rho_k}} + \frac{1}{n^{\frac{p}{2}}}.
\]
By the first part of the proof and the concentration properties of Binomial random variables we get 
\[
 \EE\sqa{ \frac{n_k}{n}\W_{Q_k}\bra{\mu_k, \rho_k}}\les \frac{1}{n^{\frac{p}{d}}}
\]
which concludes the proof of \eqref{eq:matching-iid} since $p/2>p/d$.
\end{proof}

\begin{remark}\label{rem:matching}
By translation and scaling invariance, when $\Omega=Q$ and $\rho= \frac{1}{|\Omega|} I_{Q}$ is the uniform measure of a cube $Q\subset \R^d$, the implicit constant in \eqref{eq:matching-iid} does not depend on $Q$.
\end{remark}


\subsection{Matching with a fraction of i.i.d.\ points} 
In this section we extend the bound \eqref{eq:matching-iid} for the matching  to the case where most of the points are still  i.i.d.\ but essentially no assumption is made on the remaining points. 
This is used in  \cref{thm:limit-poisson} and in the de-Poissonization procedure (see Proposition \ref{prop:limit-mean-iid}). Just like in \cref{thm:limit-poisson} we will have to
consider three different situations. Let us however set some common notation. Letting  $\cN$, $\cM$, $\cU$ and $\cV$ be point processes  on $\Omega$ ($\cN$ and $\cM$ will contain the i.i.d.\ points),  we want to estimate 
\[
 \EE\sqa{ \Mp( \cU \cup \cN, \cV \cup \cM) }.
\]
Setting 
\[Z=\min\cur{|\cU|+|\cN|,|\cV|+|\cM|},\]
we want to construct two (random)  subsets $\cS \subseteq \cU \cup \cN$, $\cT \subseteq \cV \cup \cM$, both containing $Z$ points, so that 
\begin{equation}\label{triangleMp} 
\Mp\bra{ \cU \cup \cN,  \cV \cup \cM } \le \W\bra{ \mu^{\cS}, \nu^{\cT}}\stackrel{\eqref{eq:triangle}}{\les} \W\bra{ \mu^{\cS}, Z\rho} +\W\bra{ \nu^{\cT}, Z\rho}, 
\end{equation}
where $\mu^{\cS}, \mu^{\cT}$ are the associated empirical measures. 
We then separately estimate the two terms on the right-hand side of \eqref{triangleMp}. 
Since the construction is completely symmetric, we detail it only for $\cS \subseteq \cU \cup \cN$. It is given  as the union of two sets, a ``good'' set $\cG \subseteq \cN$ and a ``bad'' set
$\cB \subseteq \cU$. We first define the set $\cG$ by sampling without replacement 
 \[
  |\cG|=\min\cur{|\cN|,Z}
 \]
points from $\cN$. Similarly, the set $\cB$ is constructed by sampling without replacement 
\[
 |\cB|=\max\cur{Z-|\cN|,0}
\]
points from $\cU$. Notice that 
\begin{equation}\label{Z}
 Z=|\cG|+|\cB|
\end{equation}
and that when conditioned on $|\cG|$, the points in $\cG$ are still i.i.d.\ with common law $\rho$.
We then write $\mu^{\cS} = \mu^{\cG} + \mu^{\cB}$ for the associated empirical measure.  
Using the triangle inequality \eqref{eq:triangle} and \eqref{eq:convexity}, we then  split the estimate in two:
\[\begin{split} \W (\mu^{\cS}, Z\rho ) & \les   \W \bra{ \mu^{\cG} + \mu^{\cB}, |\cG|\rho + \mu^{\cB}} + \W \bra{|\cG|\rho + \mu^{\cB}, Z\rho}\\
& \les \W \bra{ \mu^{\cG},|\cG|\rho}+  \W \bra{|\cG|\rho + \mu^{\cB}, Z\rho }.
\end{split}\]
 Taking expectation we find
 \begin{equation}\label{commonstartpoint}
  \EE\sqa{ \W (\mu^{\cS}, Z\rho )}\les \EE\sqa{\W \bra{ \mu^{\cG},|\cG|\rho}}+  \EE\sqa{\W \bra{|\cG|\rho + \mu^{\cB}, Z\rho }}.
 \end{equation}
To estimate the first term in the right-hand side, we will rely on \eqref{eq:transport-iid}.   It is in the estimate of  the last term that we need to argue differently depending on the cases. 
In the first one (see \cref{prop:density-helps-matching}), since we have a good control on the moments of $|\cU|$ we can directly appeal to \cref{prop:density-helps}. In the two other cases
(see \cref{prop:density-helps-matching-partition,thm:bound-matching}) we need to combine it with a localization argument.\\

We start with the first case.
\begin{proposition}\label{prop:density-helps-matching}
 Let $d \ge 3$, $p \ge 1$ and $\Omega$ be a bounded domain with Lipschitz boundary,  $\rho: \Omega\mapsto \R$ be a  H\"older continuous density bounded above and below and $(X_{i})_{i=1}^\infty$, $(Y_j)_{j=1}^\infty$ be  
 independent sequences 
of i.i.d.\ random variables with common law $\rho$. \\
 Then, there exists $\alpha=\alpha(p,d)<2$ and $\beta=\beta(p)>0$ such that the following holds. Let  $M, N\in \N$ be random variables satisfying concentration (recall  \cref{def:concen}) 
 and set $\cN=(X_{i})_{i=1}^N$, $\cM = (Y_{j})_{j=1}^M$ and $h=\min\cur{\EE\sqa{M},\EE\sqa{N}}$. 
 Then, for every point processes $\cU$, $\cV$, for which  there exists $1\le H\le h$ such that for every $q\ge 1$ 
\begin{equation}\label{hyp:UVhelps}
 \EE\sqa{ |\cU|^q+|\cV|^q}\le C(q) H^q
\end{equation}
for some $C(q)>0$, we have 
\begin{equation*}\label{eq:density-helps-random} 
\EE\sqa{ \Mp( \cU \cup \cN, \cV \cup \cM) } \les  h^{1-\frac{p}{d}}\bra{1+ \bra{\frac{H^\alpha}{h}}^\beta}.
\end{equation*}
Here the implicit constant depends only on $p$, $d$, the constants involved in the concentration properties of $M,N$ and $(C(q))_{q\ge 1}$ from \eqref{hyp:UVhelps}. 
   
\end{proposition}
\begin{proof}
Starting from \eqref{triangleMp} and  \eqref{commonstartpoint} we first estimate by \eqref{eq:matching-iid} and H\"older inequality,
\[
 \EE\sqa{\W \bra{ \mu^{\cG},|\cG|\rho}}\les \EE\sqa{|\cG|^{1-\frac{p}{d}}}\le \EE\sqa{|\cG|}^{1-\frac{p}{d}}.
\]
Since $|\cG|\le \min\cur{M,N} +|\cV|$, by \eqref{hyp:UVhelps} with $q=1$ and $H\le h$, we have $\EE\sqa{|\cG|}\les h$ and thus
\[
 \EE\sqa{\W \bra{ \mu^{\cG},|\cG|\rho}}\les h^{1-\frac{p}{d}}.
\]
We are then  left with the proof of 
\begin{equation}\label{toprovehelps}
 \EE\sqa{\W \bra{|\cG|\rho + \mu^{\cB}, Z\rho }}\les h^{1-\frac{p}{d}} \bra{\frac{H^\alpha}{h}}^\beta.
\end{equation}
We first single out the event 
\[
 A=\cur{|\cG|\ge h/2}
\]
and claim that for $q\ge 1$
\begin{equation}\label{Achelps}
 \PP\sqa{A^c}\les_q h^{-q}.
\end{equation}
Indeed, since $A^c\subset \{N\le \EE\sqa{N}/2\}\cup\{M\le \EE\sqa{M}/2\}$, \eqref{Achelps} follows by combining a  union bound together with the concentration properties of $M$ and $N$.
Since 
\[
\W \bra{|\cG|\rho + \mu^{\cB}, Z\rho }\le  \W\bra{ \mu^{\cB}, |\cB|\rho}\les |\cB|\le |\cU|,
\]
 we then find 
\[
 \EE\sqa{\W \bra{|\cG|\rho + \mu^{\cB}, Z\rho }I_{A^c}}\les \EE\sqa{|\cU|^2}^{\frac{1}{2}} \PP\sqa{A^c}^{\frac{1}{2}}\stackrel{\eqref{hyp:UVhelps}\&\eqref{Achelps}}{\les_q} h^{-q} H.
\]
By taking $q$ large enough, in order to prove \eqref{toprovehelps} it is therefore sufficient to show 
\begin{equation}\label{toprovehelpsA}
  \EE\sqa{\W \bra{|\cG|\rho + \mu^{\cB}, Z\rho }I_A}\les h^{1-\frac{p}{d}} \bra{\frac{H^\alpha}{h}}^\beta.
\end{equation}
We start with the case $p>d/(d-1)$. By  \cref{prop:density-helps},
\begin{multline*}
 \EE\sqa{\W \bra{|\cG|\rho + \mu^{\cB}, Z\rho }I_A}\les\EE\sqa{\frac{|\cB|^{1+\frac{p}{d}}}{|\cG|^{\frac{p}{d}}} I_A}
 \les h^{-\frac{p}{d}}\EE\sqa{ |\cU|^{1+\frac{p}{d}} I_A}\stackrel{\eqref{hyp:UVhelps}}{\les} h^{1-\frac{p}{d}} \frac{H^{1+\frac{p}{d}}}{h}.
\end{multline*}
This  proves \eqref{toprovehelpsA} in this case with $\alpha=1+p/d$ and $\beta=1$.\\
If now $p<d/(d-1)<2$, we use Jensen's inequality \eqref{eq:jensen} to obtain 
\begin{multline*}
  \EE\sqa{\W \bra{|\cG|\rho + \mu^{\cB}, Z\rho }I_A}\les\EE\sqa{ Z^{1-\frac{p}{2}} \bra{\mathsf{W}^2\bra{|\cG|\rho + \mu^{\cB}, Z\rho }I_A}^{\frac{p}{2}}}\\
  \le \EE\sqa{ Z}^{1-\frac{p}{2}} \EE\sqa{\mathsf{W}^2\bra{|\cG|\rho + \mu^{\cB}, Z\rho }I_A}^{\frac{p}{2}}.
\end{multline*}
Recalling \eqref{Z} we find $\EE\sqa{ Z}\les h+H\les h$. Using finally \eqref{toprovehelpsA} with $p=2$ we conclude that 
\[
 \EE\sqa{\W \bra{|\cG|\rho + \mu^{\cB}, Z\rho }I_A}\les h^{1-\frac{p}{2}} \bra{h^{1-\frac{2}{d}} \bra{\frac{H^\alpha}{h}}^\beta}^{\frac{p}{2}}= h^{1-\frac{p}{d}}  \bra{\frac{H^\alpha}{h}}^{\beta\frac{p}{2}}.
\]
This proves \eqref{toprovehelpsA} also in this case.
\end{proof}
We now consider the case when the moment bounds for $\cU$ and $\cV$ are only valid after restricting  on a Whitney-type decomposition from \cref{lem:decomp}.

\setcounter{proof-step}{0}
\begin{proposition}\label{prop:density-helps-matching-partition}
Let $d \ge 3$, $p \ge 1$ and $\Omega \subseteq \R^d$ be a bounded connected open set with Lipschitz boundary and such that \eqref{eq:green-kernel-bound} holds. 
Fix a Whitney partition $\cQ = (Q_i)_i$, and for  $\delta >0$ let $(\Omega_k)_{k=1}^K = \cQ_\delta \cup \cR_\delta$ be given  by \cref{lem:decomp}. Let finally $\rho$ be a H\"older continuous probability density on $\Omega$, bounded above and below.\\
Then, there exist $\alpha=\alpha(p,d)>0$ and $\beta=\beta(p,d) >0$ such that the following holds. For every $\eta\in (0,1)$, $\eps>0$ and $\gamma\in( 0, 1/d)$, there exists $C(\eta,\eps,\gamma)$ such that 
for every Poisson point processes $\cN^{\eta n \rho}, \cM^{\eta n \rho}$ with intensity  $\eta n\rho$ and every point processes $\cU$ and $\cV$ on $\Omega$ such that 
\begin{equation}\label{hypmomUVhelpmatching}
 \EE\lt[|\cU_{\Omega_k}|^q+ |\cV_{\Omega_k}|^q\rt]\les_q (n|\Omega_k|)^{\frac{q}{2}} \qquad \forall q>0,
\end{equation}
if $\delta= n^{-\gamma}$ then
\begin{multline*}
 n^{\frac{p}{d}-1}\EE\sqa{ \Mp( \cU \cup \cN^{\eta n\rho}, \cV \cup \cM^{\eta n\rho}) }\les 
 \eta^{1-\frac{p}{d}}\\
 +C(\eta,\eps,\gamma) n^{\eps}\bra{ \bra{\max\cur{n^{\frac{p}{d}}\delta^{p+1},n^{\frac{2}{d}} \delta^3}}^{\alpha} +\bra{n\delta^d}^{-\beta} }. 
\end{multline*}

\end{proposition}

\begin{proof}
Using the notation from the beginning of this section, we start as above from \eqref{triangleMp} and  \eqref{commonstartpoint} and  estimate by \eqref{eq:matching-iid},
\[
 \EE\sqa{\W \bra{ \mu^{\cG},|\cG|\rho}}\les  \EE\sqa{|\cG|}^{1-\frac{p}{d}}.
\]
Since $|\cG|\le |\cN^{n\eta \rho}|$ we get
\[
 \EE\sqa{\W \bra{ \mu^{\cG},|\cG|\rho}}\les (\eta n)^{1-\frac{p}{d}}.
\]
In order to conclude the proof it is thus enough to show
\begin{equation}\label{mainclaimdensityhelps}
n^{\frac{p}{d}-1} \EE\sqa{\W\bra{ \mu^{\cB}+|\cG|\rho, Z\rho}}\le C(\eta,\eps,\gamma) n^{\eps}\bra{ \bra{\max\cur{n^{\frac{p}{d}}\delta^{p+1},n^{\frac{2}{d}} \delta^3}}^{\alpha} +\bra{n\delta^d}^{-\beta}}. 
\end{equation}
\setcounter{proof-step}{0}

\noindent{\emph{Step \stepcounter{proof-step}\arabic{proof-step}. Reduction to a ``good'' event.}}
We let 
\[
 A=\{|\cG|\in[\eta n/2, 3\eta n]\}\cap \bigcap_{k=1}^K  \cur{ \max\cur{ |\cU_{\Omega_k}|, |\cV_{\Omega_k}| }\le (n |\Omega_k|)^{\frac{1}{2}} \cdot n^{\eps} }.
\]
and claim that 
\begin{equation}\label{estimAcdensityhelp}
 n^{\frac{p}{d}-1}\EE\sqa{\W\bra{ \mu^{\cB}+|\cG|\rho, Z\rho} I_{A^c}}\le C(\eta, \eps,\gamma)\bra{n\delta^d}^{-\beta}.
\end{equation}
We first prove that for every $q>0$,
\begin{equation}\label{estimPAcdensityhelp}
 \PP\sqa{A^c}\les_q C(\eta) n^{-q} +\delta^{1-d} n^{-\eps q}. 
\end{equation}
To prove this we use a union bound and split
\begin{multline*}
  \PP\sqa{A^c}\le \PP\sqa{|\cG|\notin [\eta n/2, 3\eta n] }\\
  +\sum_{k=1}^K \PP\sqa{|\cU_{\Omega_k}|\ge (n |\Omega_k|)^{\frac{1}{2}} \cdot n^{\eps} }+ \sum_{k=1}^K \PP\sqa{|\cV_{\Omega_k}|\ge (n |\Omega_k|)^{\frac{1}{2}} \cdot n^{\eps} }.
\end{multline*}
Regarding the first term we notice that 
\begin{multline*}
 \{|\cG|\notin[\eta n/2, 3\eta n]\}\subset \{ |\cN^{\eta n \rho}|< \eta n/2\}\cup \{ |\cN^{\eta n \rho}|> 3\eta n\} \cup \{ |\cM^{\eta n \rho}|< \eta n/2\}.
\end{multline*}
Using once more a union bound and \eqref{eq:density-bound-below-Poi}, we find 
\[
 \PP\sqa{|\cG|\notin [\eta n/2, 3\eta n] }\les_q C(\eta) n^{-q}.
\]
Regarding the two sums, by \eqref{hypmomUVhelpmatching}, we have for every $k\in [1,K]$
\[
 \PP\sqa{|\cU_{\Omega_k}|\ge (n |\Omega_k|)^{\frac{1}{2}} \cdot n^{\eps} }\les_q n^{-\eps q}
\]
and similarly for $\cV$. Since $K\les \delta^{1-d}$ by \eqref{eq:whitney-general-q} this concludes the proof of \eqref{estimPAcdensityhelp}.\\
We now turn to \eqref{estimAcdensityhelp}. As above by the bound $\W\bra{ \mu^{\cB}+|\cG|\rho, Z\rho}\les |\cU|$ and  Cauchy-Schwarz, we have 
\[
\EE\sqa{\W\bra{ \mu^{\cB}+|\cG|\rho, Z\rho} I_{A^c}}\les \EE\sqa{|\cU|^2}^{\frac{1}{2}}\PP\sqa{A^c}^{\frac{1}{2}}.
\]
Using once more Cauchy-Schwarz together with \eqref{hypmomUVhelpmatching} with $q=2$ we have 
\[
  \EE\sqa{|\cU|^2}^{\frac{1}{2}}\les K^{\frac{1}{2}} n^{\frac{1}{2}}
\]
so that by \eqref{estimPAcdensityhelp} and $K\les \delta^{1-d}$
\[
n^{\frac{p}{d}-1} \EE\sqa{\W\bra{ \mu^{\cB}+|\cG|\rho, Z\rho} I_{A^c}}\les_q C(\eta) \delta^{\frac{1}{2}(1-d)} (n^{-q} +\delta^{1-d} n^{-\eps q})^{\frac{1}{2}}n^{\frac{p}{d}-\frac{1}{2}}.
\]
Since $\delta=n^{-\gamma}$, this concludes the proof of \eqref{estimAcdensityhelp} provided we choose $q$ large enough depending on $\eps$ and $\gamma$.

\medskip
\noindent In the remaining two steps we prove that in $A$,
\begin{equation}\label{mainclaimdensityhelpsA}
 n^{\frac{p}{d}-1}\W\bra{ \mu^{\cB}+|\cG|\rho, Z\rho}\les_\eta  n^{\eps}\bra{ \bra{\max\cur{n^{\frac{p}{d}}\delta^{p+1},n^{\frac{2}{d}} \delta^3}}^{\alpha} +\bra{n\delta^d}^{-\beta} }. 
\end{equation}
After taking expectation and in combination with \eqref{estimAcdensityhelp} this would conclude the proof of \eqref{mainclaimdensityhelps}. From this point all the estimates are deterministic.\\

\medskip 
\noindent{\emph{Step \stepcounter{proof-step}\arabic{proof-step}. Estimate for $p>d/(d-1)$.}}
We  first use  \eqref{eq:mainsub}, e.g.\ with $\varepsilon=1$, to obtain
\begin{equation}\label{eq:subadd-epsilon-1}
\W_{\Omega} \bra{ \mu^{\cB}+ |\cG|\rho, Z\rho }  \les \sum_{k=1}^K \W_{\Omega_k}\bra{ \mu^{\cB} + |\cG|\rho, \alpha_k \rho}
+ \W_{\Omega}\bra{ \sum_{k=1}^K \alpha_k I_{\Omega_k} \rho, Z\rho },
\end{equation}
with
\begin{equation}\label{eq:alpha_k} \alpha_k = \frac{ \mu^{\cB}(\Omega_k)}{\rho(\Omega_k)} + |\cG|.\end{equation}
We bound the terms in the right-hand side separately. For the sum of ``local'' terms, we estimate differently according to $\Omega_k\in \cR_\delta$ or $\Omega_k\in \cQ_\delta$.
In the first case we use the naive bound
\[ \begin{split} \W_{\Omega_k}\bra{ {\mu^{\cB} + |\cG|\rho}, \alpha_k \rho}  & \stackrel{\eqref{eq:convexity}}{\le} \W_{\Omega_k}\bra{ \mu^{\cB} ,  \frac{\mu^{\cB}(\Omega_k)}{\rho(\Omega_k)} \rho}
 \stackrel{\eqref{eq:w-trivial}}{\le} \diam(\Omega_k)^p |\cU_{\Omega_k}| \\
 & \les n^{\frac{1}{2}+\varepsilon} \delta^{p + \frac{d}{2}}.\end{split}\]
Since $K\les \delta^{1-d}$ we find
\begin{equation}\label{eq:second-term-matching-whitney}n^{\frac{p}{d}-1}
\sum_{\Omega_k \in \cR_\delta} \W_{\Omega_k}\bra{ \mu^{\cB} + |\cG| \rho, \alpha_k \rho} \les  n^\eps n^{\frac{p}{d}} \delta^{1+p} (n\delta^{-d})^{ - \frac{1}{2}}\le n^\eps n^{\frac{p}{d}} \delta^{1+p}.
\end{equation}
If $\Omega_k \in \cQ_\delta$ is a cube,  we use instead \cref{prop:density-helps} with $\mu^{\cB}$ instead of $\mu$ and $|\cG|$ instead of $h$, so that
\[  \begin{split} \W_{\Omega_k}\bra{ {\mu^{\cB} + |\cG| \rho}, \alpha_k\rho}  & \les_\eta   \frac{\mu^{\cB}(\Omega_k)^{1+\frac{p}{d}}}{n^{\frac{p}{d}}}    \les_\eta n^{-\frac{p}{d}} |\cU_{\Omega_k}|^{1+\frac{p}{d}} \\
&  \les_\eta n^{(1+\frac{p}{d})\eps}  n^{\frac{1}{2} (1-\frac{p}{d})} |\Omega_k|^{\frac{1}{2}(1+\frac{p}{d})}.
\end{split}\]
Summing this inequality yields
\begin{equation*}\begin{split} n^{\frac{p}{d}-1}\sum_{\Omega_k \in \cQ_\delta} \W_{\Omega_k}\bra{ \mu^{\cB} + |\cG| \rho, \alpha_k \rho}   & \les_\eta   n^{(1+\frac{p}{d})\eps}  n^{-\frac{1}{2} (1-\frac{p}{d})} \sum_{k=1}^K   |\Omega_k|^{\frac{1}{2}(1+\frac{p}{d})} \\
&  \stackrel{ \eqref{eq:whitney-general-q}}{\les_\eta}  n^{2 \eps}  n^{-\frac{1}{2} (1-\frac{p}{d})}\max\cur{1,\delta^{\frac{1}{2}(d-2-p)}}.
\end{split}\end{equation*}
Notice that since $n \delta^d\ge 1$, 
\[
 n^{-\frac{1}{2} (1-\frac{p}{d})}\max\cur{1,\delta^{\frac{1}{2}(d-2-p)}}\le (n\delta^d)^{-\frac{1}{2} (1-\frac{p}{d})}
\]
so that 
\begin{equation}\label{eq:third-term-matching-whitney}
 n^{\frac{p}{d}-1}\sum_{\Omega_k \in \cQ_\delta} \W_{\Omega_k}\bra{ \mu^{\cB} + |\cG| \rho, \alpha_k \rho} \les_\eta n^{2 \eps} (n\delta^d)^{-\frac{1}{2} (1-\frac{p}{d})}.
\end{equation}

We then consider the last term in \eqref{eq:subadd-epsilon-1}. Using  \cref{lem:peyre} with $Z \rho \gtrsim  \eta n$ in place of $\lambda$ (recall that we assume here that $A$ holds),  we get
\begin{equation}\label{firststepglobalmatchingwhitney}  \W_{\Omega}\bra{ \sum_{k=1}^K \alpha_k I_{\Omega_k} \rho, Z\rho } \les_\eta  n^{1-p} \nor{ \sum_{k=1}^K \alpha_kI_{\Omega_k} \rho - Z\rho}_{W^{-1,p}(\Omega)}^p.
\end{equation}
Recalling \eqref{eq:alpha_k} and that $Z = \mu^{\cB}(\Omega) + |\cG|$, we can rewrite
\[  \sum_{k=1}^K \alpha_kI_{\Omega_k} \rho - Z\rho  = \sum_{k=1}^K  \frac{ \mu^{\cB}(\Omega_k)}{\rho(\Omega_k)}   \bra{I_{\Omega_k}- \rho(\Omega_k)} \rho.
\]
By \eqref{estimW1pWhitney} of \cref{lem:W1pWhitney} with $h= n^{1+2\eps}$ we thus have in $A$
\[
 \nor{ \sum_{k=1}^K \alpha_kI_{\Omega_k} \rho - Z\rho}_{W^{-1,p}(\Omega)}\les n^{\eps} \delta^{1-\frac{d}{2}} |\log (\delta )| n^{\frac{1}{2}}.  
\]
Combining this with \eqref{firststepglobalmatchingwhitney} we get that in $A$,
\[n^{\frac{p}{d}-1} \W_{\Omega}\bra{ \sum_{k=1}^K \alpha_k I_{\Omega_k} \rho, Z\rho } \les  n^{p\eps} (n\delta^d)^{-p\frac{(d-2)}{2d}} |\log (\delta )|^p . 
\]
Inserting this estimate, \eqref{eq:second-term-matching-whitney} and \eqref{eq:third-term-matching-whitney} in \eqref{eq:subadd-epsilon-1} we finally obtain (notice that $p(d-2)>d-p$ for $p>d/(d-1)$) that in $A$,
\[
 n^{\frac{p}{d}-1}\W\bra{ \mu^{\cB}+|\cG|\rho, Z\rho}\les_\eta n^{\max\cur{p,2} \eps} \bra{ n^{\frac{p}{d}} \delta^{1+p}+(n\delta^d)^{-\frac{(d-p)}{2d}} |\log (\delta )|^p }.
\]
Up to replacing $\eps$ by $\max\cur{p,2} \eps$ and choosing $\beta< \frac{(d-p)}{2d}$ this  concludes  the proof of \eqref{mainclaimdensityhelpsA} if $p>d/(d-1)$.\\

\medskip
\noindent{\emph{Step \stepcounter{proof-step}\arabic{proof-step}. Estimate for $p\le d/(d-1)$.}}
Since $2>d/(d-1)\ge p$, we may use Jensen's inequality \eqref{eq:jensen} to infer that in $A$,
\[\begin{split}
n^{\frac{p}{d}-1} \W\bra{ \mu^{\cB}+|\cG|\rho, Z\rho}& \le n^{\frac{p}{d}-1} Z^{1-\frac{p}{2}}\bra{\mathsf{W}^2\bra{\mu^{\cB}+|\cG|\rho, Z\rho}}^{\frac{p}{2}}\\
 & \les (\eta +n^{-\frac{1}{2}})^{1-\frac{p}{2}}\bra{n^{\frac{2}{d}-1}\mathsf{W}^2\bra{\mu^{\cB}+|\cG|\rho, Z\rho}}^{\frac{p}{2}}\\
 & \les \bra{n^{\frac{2}{d}-1}\mathsf{W}^2\bra{\mu^{\cB}+|\cG|\rho, Z\rho}}^{\frac{p}{2}}.
\end{split}\]
Using  \eqref{mainclaimdensityhelpsA} for $p=2$ concludes the proof of \eqref{mainclaimdensityhelpsA} also in this case.
\end{proof}

Finally, we  consider the case of a cube $Q_{mL}$ decomposed into cubes of sidelength $L$. The difficulty compared to the previous two cases is 
to obtain bounds which are independent of $m$. This is achieved using the additional independence for the point processes $\cU$, $\cV$.
While we believe that a direct proof combining Green kernel bounds in the spirit of the proof of \cref{lem:W1pWhitney} together with a Rosenthal
type inequality for the (non independent) random variables $\mu^{\cB}(Q_i)$ should be possible we give
a more elementary proof based on subadditivity and concentration.  

\begin{proposition}\label{thm:bound-matching}
Let $d \ge 3$, $\eta \in (0,1/2)$, $L \ge 1$ and $m \in \mathbb{N}\setminus\cur{0}$. Let $\cU$, $\cV$ be point processes on $Q_{mL}$ such that the restrictions $(\cU_{Q_i}, \cV_{Q_i})_{i}$ on all sub-cubes 
$Q_i= Q_L + Lz_i\subseteq Q_{mL}$, with $z_i \in \mathbb{Z}^d$, are independent copies (translated by the vector $Lz_i$) of the pair of processes $(\cU_{Q_L}, \cV_{Q_L})$ and such that for every $q\ge 1$, there exists $C(q)>0$ such that 
\begin{equation}\label{hypUV} \EE\sqa{ |\cU_{Q_i}|^q +|\cV_{Q_i}|^q}\le C(q) L^{d\frac{q}{2}}.\end{equation}
Let $\cN^{\eta}$, $\cM^{\eta}$ independent Poisson processes on $Q_{mL}$ with constant intensity $\eta$, also independent from $(\cU,\cV)$. Then, 
for every $p \in [1, d)$, there exists $C(\eta) = C(\eta,p,d, (C(q))_{q\ge 1})>0$ and $\alpha= \alpha(p,d)>0$, such that, if $L \ge C(\eta)$,
\[\EE\sqa{ \frac{1}{|Q_{m L}|} \Mp\bra{ \cU \cup \cN^\eta,  \cV \cup \cM^{\eta} }} \lesssim  \eta^{1-\frac{p}{d}} + \frac{ C(\eta)}{L^\alpha}.\]
\end{proposition}

\begin{remark}
Let us preliminarily notice that, for any $R \subseteq Q_{mL}$ that is the disjoint union of $k$ cubes among the cubes $Q_i = Q_L + Lz_i$, $z_i \in \mathbb{Z}^d$, we have the upper bound, if $q \ge 1$,
\begin{equation} \label{eq:moment-ur}\begin{split} \EE\sqa{ |\cU_R|^{q}} & = k^{q} \EE\sqa{ \bra{\frac 1 k \sum_{Q_i \subseteq R} |\cU_{Q_i}|}^{q}} \le k^{q} \EE\sqa{\frac 1 k \sum_{Q_i \subseteq R} |\cU_{Q_i}|^{q}} \\
& \les k^{q} L^{d\frac{q}{2}} \les (kL^d)^{q} = |R|^{q}.
\end{split}\end{equation}
In particular, we have 
\begin{equation} \label{eq:moment-qml} \EE\sqa{ |\cU|^{q} } \les m^{dq} L^{d\frac{q}{2}} \les |Q_{mL}|^q.
\end{equation}
 Moreover, by Rosenthal inequalities \cite{rosenthal1970subspaces},  if $q \ge 2$,
\[\begin{split}  \EE\sqa{ \abs{ |\cU_R| - \EE\sqa{|\cU_R|}}^q } & = \EE\sqa{ \abs{ \sum_{Q_i\subseteq R} \bra{ |\cU_{Q_i}| - \EE\sqa{|\cU_{Q_i}|}} }^q} \\
 & \les k \EE\sqa{ \abs{ |\cU_{Q_L}| - \EE\sqa{|\cU_{Q_L}|}}^q } + k^{\frac{q}{2}} \EE\sqa{ \abs{ |\cU_{Q_L}| - \EE\sqa{|\cU_{Q_L}|}}^2 }^{\frac{q}{2}} \\
 & \les k L^{d\frac{q}{2}} +  k^{\frac{q}{2}} L^{d\frac{q}{2}} \les |R|^{\frac{q}{2}}.
 \end{split}\]
 We will use all these bounds in the proof below.
\end{remark}
%


\begin{proof}[Proof of \cref{thm:bound-matching}]
 For simplicity, we write throughout the proof $Q$ instead of $Q_{mL}$. 
 As in the previous two proofs, we start from \eqref{triangleMp} and  \eqref{commonstartpoint} (with $\rho=I_Q/|Q|$) and  estimate by \eqref{eq:matching-iid}, see also \cref{rem:matching},
\[
 \EE\sqa{\frac{1}{|Q|}\W_Q \bra{ \mu^{\cG},\frac{|\cG|}{|Q|}}}\les  \EE\sqa{|\cG|}^{1-\frac{p}{d}} |Q|^{\frac{p}{d}-1}.
\]
Since $|\cG|\le |\cN^\eta|$ we get
\[
\EE\sqa{\frac{1}{|Q|}\W_Q \bra{ \mu^{\cG},\frac{|\cG|}{|Q|}}}\les \eta^{1-\frac{p}{d}}.
\]
In order to conclude the proof it is thus enough to show

 \setcounter{proof-step}{0}
\begin{equation}\label{eq:bound-bad-part} \EE\sqa{ \frac{1}{|Q|} W_{Q}^p \bra{ \frac{ |\cG|}{|Q|} + \mu^{\cB}, \frac{Z}{|Q|} } } \les \frac{C(\eta)}{L^\alpha}.\end{equation}
We split the proof into several steps. We first consider the case $p\ge 2\ge d/(d-1)$.\\
\medskip 

\noindent{\emph{Step \stepcounter{proof-step}\arabic{proof-step}. Concentration bounds for $\mu^\cB$.}}
In this intermediate step, we collect some facts about $\mu^{\cB}(R)$, where $R \subseteq Q$ is a disjoint union of $k$ cubes $Q_i = Q_L + Lz_i$, $z_i \in \mathbb{Z}^d$. First of all, the construction of $\cB$ ensures that $\EE\sqa{ \mu^{\cB}(Q_i)}$ does not depend on $Q_i$ (one could in fact prove that $(\mu^{\cB}(Q_i))_{i}$ is an exchangeable sequence). We deduce that
\begin{equation} \label{eq:constant-mean} \EE\sqa{ \mu^{\cB}(Q_i)} = \frac{ \EE\sqa{ \mu^{\cB}(Q)}}{m^d}, \quad \text{hence} \quad  \frac{\EE\sqa{ \mu^{\cB}(R)}}{|R|} = \frac{ \EE\sqa{ \mu^{\cB}(Q)}}{|Q|}.\end{equation}
Indeed, when conditioned on $Z=z$, $|\cN^\eta|=n$,  $|\cU| = u_Q$, and the number of points $|\cU_R|=u_R \le u_Q$,  $\mu^{\cB}(R)$ is the number of ``successes''
in the random sampling procedure, without replacement which we used to define $\cB$, with $b = \max\cur{z-n,0}$ draws from an urn containing $ u_Q$ marbles, $u_R$ of which have the desired feature 
(their extraction defines a success). This is explicitly given by a hypergeometric distribution with parameters $\bra{u_Q, u_R, b}$: given $s_R \le u_R$,
\[\PP\bra{ \mu^{\cB}(R) = s_R  | B}=   { u_R \choose s_R} {u_Q-u_R \choose b - s_R} / { u_Q \choose b},\]
where for brevity we write
$$ B = \cur{ Z=z, |\cN^\eta|=n_Q,  |\cU| = u_Q, |\cU_R| = u_R}.$$
Specializing to $R = Q_i$, we see that this quantity does not depend on $Q_i$, since $|\cU_{Q_i}|$ are i.i.d.\ variables, hence the joint laws of the variables $(Z, |\cN^\eta|, |\cU|,  |\cU_{Q_i}| )$ involved in the definition of the law of $\mu^{\cB}(Q_i)$ do not depend on $i$. 

\noindent Using the concentration inequality \eqref{eq:concentration-hyper} for hypergeometric random variables, we have
$$ \EE\sqa{\abs{ \mu^{\cB}(R) - \EE\sqa{\mu^{\cB}(R)|B}}^p |B} \les \bra{u_R}^{\frac{p}{2}},$$
 from which we find, thanks to \eqref{eq:moment-ur} (recall that $p\ge 2$),
\begin{equation}\label{eq:density-fluctation} \EE\sqa{ \abs{ \mu^\cB(R)- \EE\sqa{ \mu^\cB(R)}} ^p  } \lesssim  |R|^{\frac{p}{2}}.\end{equation}

\noindent{\emph{Step \stepcounter{proof-step}\arabic{proof-step}. Subadditivity bound.}}
Using \eqref{eq:constant-mean} and \eqref{eq:density-fluctation} above, we are in a position to follow closely the main argument of \cite[Proposition 5.4]{goldman2021convergence}. We  define, for a rectangle $R \subseteq Q$ that is a union of cubes $Q_i$'s,
\[
 \fref(R)= \EE\lt[\frac{1}{|R|}\W_R\bra{ \mu^\cB + \frac{|\cG|}{|Q|}, \frac{\mu^{\cB}(R)}{|R|} + \frac{|\cG|}{|Q|}  }\rt].
\]
We say that $\cR$ is an admissible partition  of $R$ if it is made of rectangles satisfying the following conditions. Each $R_k\in \cR$ is a union of cubes $Q_i$, it is of moderate aspect ratio and $3^{-d}|R|\le |R_j|\le |R|$. We claim that there exists $C_\eta=C(d,p,\eta)>0$ such that    for   every admissible partition $\cR$ 
of $R$  and every $\eps\in(0,1)$, we have 
 \begin{equation}\label{onestep}
  \fref(R)\le (1+\eps)\sum_i \frac{|R_i|}{|R|} \fref(R_i) +\frac{C_\eta}{\eps^{p-1}} \frac{1}{|R|^{\frac{p(d-2)}{2d}}}.
 \end{equation}
 Setting 
 \[
  \alpha = \frac{ \mu^{\cB}(R)}{|R|}+\frac{ |\cG|}{|Q|},  \quad  \alpha_i = \frac{ \mu^{\cB}(R_i)}{|R_i|}+\frac{|\cG|}{|Q|}
 \]
and using \eqref{eq:mainsub}, this reduces to 
\begin{equation}\label{toproveonestep}
 \EE\sqa{ \frac{1}{|R|}\W_{R}\bra{\sum_{i} \alpha_i I_{R_i},\alpha}}
 \le \frac{C_\eta}{ |R|^{\frac{p(d-2)}{2d}}}.
\end{equation}

\noindent First, we single out the event 
$$ A= \cur{ \min\cur{ |\cN^\eta|,  |\cM^\eta|} \ge \eta |Q| /2}.$$
Notice that on $A$, we have $\alpha\ges  \eta$.
By the concentration bound \eqref{eq:density-bound-below-Poi}, for every $q\ge 1$, $\mathbb{P}(A^c) \les_q (\eta |Q|)^{-q}\le (\eta |R|)^{-q}$. Therefore, if $A^c$ holds, we can use the trivial bound

\[\begin{split}\frac{1}{|R|}\W_{R}\bra{\sum_{i} \alpha_i I_{R_i},\alpha}& \le \frac{1}{|R|}\W_{R}\bra{\sum_{i} \frac{ \mu^{\cB}(R_i)}{|R_i|} I_{R_i},\frac{ \mu^{\cB}(R)}{|R|}} \\
& \le |R|^{\frac{p}{d}-1} \mu^{\cB}(R)
\le |R|^{\frac{p}{d}-1} |\cU_R|.
\end{split}\]
Using   Cauchy-Schwarz inequality and \eqref{eq:moment-ur} with $q=2$, we get for any $q \ge 1$,
\[  \EE \sqa{  \frac{1}{|R|}\W_{R}\bra{\sum_{i} \alpha_i I_{R_i},\alpha} I_{A^c}}   \les_{\eta,q} |R|^{\frac{p}{d}-q},
\]
which is estimated by the right-hand side of \eqref{toproveonestep} provided we choose $q$ large enough.\\

\noindent If $A$ holds, we use \eqref{eq:estimCZ} in combination with \eqref{eq:Lp} (recall that for rectangles of moderate aspect ratio the Sobolev constant is uniformly bounded) to get   
\begin{align*}
 \frac{1}{|R|}\W_{R}\bra{\sum_{i} \alpha_i I_{R_i},\alpha}&\les \frac{ |R|^{\frac{p}{d}-1} }{ \alpha^{p-1}} \sum_{i} |R_i| \abs{\alpha_i - \alpha }^p\\
 &\les \eta^{1-p} |R|^{\frac{p}{d}}\sum_{i} \abs{\alpha_i - \alpha }^p.
\end{align*}
We thus have 
\[\begin{split}
 \EE\sqa{\frac{1}{|R|}\W_{R}\bra{\sum_{i} \alpha_i I_{R_i},\alpha}I_A} & \les \eta^{1-p} |R|^{\frac{p}{d}}\sum_{i}\EE\sqa{\abs{\alpha_i - \alpha }^pI_A}\\ &
 \le \eta^{1-p} |R|^{\frac{p}{d}}\sum_{i}\EE\sqa{\abs{\alpha_i - \alpha }^p}.\end{split}\]
 Using that $\alpha_i-\alpha=\frac{\mu^{\cB}(R_i)}{|R_i|}-\frac{\mu^{\cB}(R)}{|R|}$, \eqref{eq:constant-mean} and triangle inequality we have 
\[ \begin{split}
  \sum_{i}\EE\sqa{\abs{\alpha_i - \alpha }^p}
  & \les \sum_i \EE\sqa{\abs{\frac{\mu^{\cB}(R_i)}{|R_i|}- \EE\sqa{\frac{\mu^{\cB}(R_i)}{|R_i|}} }^p} + \EE\sqa{\abs{\frac{\mu^{\cB}(R)}{|R|}- \EE\sqa{\frac{\mu^{\cB}(R)}{|R|}} }^p}\\
  & \stackrel{\eqref{eq:density-fluctation}}{\les} |R|^{-\frac{p}{2}}. 
 \end{split}\]
This proves  
\[
 \EE\sqa{\frac{1}{|R|}\W_{R}\bra{\sum_{i} \alpha_i I_{R_i},\alpha}I_A}\les \frac{\eta^{1-p}}{ |R|^{\frac{p(d-2)}{2d}}},
\]
concluding the proof of \eqref{toproveonestep}.

\noindent{\emph{Step \stepcounter{proof-step}\arabic{proof-step}. Dyadic approximation.}} Starting from the cube $Q=Q_{mL}$, we build a sequence of finer and finer partitions of $Q_{mL}$ by rectangles of moderate aspect ratios that are unions of sub-cubes $Q_i$'s. We let $\mathcal{R}_0=\{Q_{mL}\}$ and define $\mathcal{R}_k$ inductively as follows.
 Let $R\in \mathcal{R}_k$. Up to translation we may assume that $R=\prod_{i=1}^d (0, m_i L)$ for some $m_i\in \N$. We then split each interval $(0,m_i L)$ into $(0,\lfloor\frac{m_i}{2}\rfloor L)\cup(\lfloor\frac{m_i}{2}\rfloor L, m_i L)$. 
 It is readily seen that this induces an admissible partition of $R$. Let us point out that when $m_i=1$ for some $i$, the corresponding interval   $(0,\lfloor\frac{m_i}{2}\rfloor L)$ is empty.
 This procedure stops after a finite number of steps $K$ once $\mathcal{R}_K=\{Q_L+z_i, z_i\in [0,m-1]^d \cap \mathbb{Z}^d\}$. It is also readily seen that $2^{K-1}<m\le 2^K$ and that for every $k\in [0,K]$ and every $R\in \mathcal{R}_k$ we have $|R|\sim (2^{K-k} L)^d$. \\
We prove via a downward induction the existence of $\Lambda_\eta>0$ such that  for every $k\in [0,K]$ and every $R\in \mathcal{R}_{k}$,
\begin{equation}\label{induction}
 \fref(R)\le \fref(Q_L)+ \Lambda_\eta(1+\fref(Q_L)) L^{-\frac{d-2}{2}} \sum_{j=K-k}^K 2^{- j\frac{d-2}{2}}.
\end{equation}
The statement is clearly true for $k=K$, since the law of the point process on each cube $Q_i = Q_L + z_i$ is the same, hence $f(Q_i) = f(Q_L)$. Assume that it holds true for $k+1$. Let $R\in \mathcal{R}_{k}$. Applying \eqref{onestep} with $\eps= (2^{K-k} L)^{-(d-2)/2}\ll1$, we get 
\[ \begin{split}
 \fref(R)&\le (1+ \eps) \sum_{R_i\in \mathcal{R}_{k+1}, R_i\subset R} \frac{|R_i|}{|R|} \fref(R_i) + \frac{C_\eta}{\eps^{p-1}} \frac{1}{|R|^{\frac{p(d-2)}{2d}}}\\
 &\stackrel{\eqref{induction}}{\le} (1+\eps) \lt(\fref(Q_L)+ \Lambda_\eta(1+\fref(Q_L))L^{-\frac{d-2}{2}} \sum_{j=K-k+1}^K 2^{- j\frac{d-2}{2}}\rt) \\
 & \qquad \qquad +  C_\eta(2^{K-k} L)^{-\frac{d-2}{2}}\\
 &\le  \fref(Q_L)+  \Lambda_\eta(1+\fref(Q_L))L^{-\frac{d-2}{2}}\cdot\\
 & \qquad \qquad \cdot \lt[\sum_{j=K-k+1}^K 2^{- j\frac{d-2}{2}}+2^{-(K-k)\frac{d-2}{2}}\lt( \frac{C_\eta+1}{\Lambda_\eta}+L^{-\frac{d-2}{2}} \sum_{j=K-k+1}^K 2^{- j\frac{d-2}{2}}  \rt)\rt].
\end{split}\]
If $L$ is large enough (depending on $\eta$)  then 
$$ \bra{ \sum_{j=K-k+1}^K 2^{- j\frac{d-2}{2}}}\eta^{1-3p} L^{-\frac{(d-2)}{2}}\les_\eta  \bra{ \sum_{j=0}^\infty 2^{- j\frac{d-2}{2}}} L^{-\frac{(d-2)}{2}	} \le \frac{1}{2}.$$ Finally, choosing $\Lambda_\eta\ge 2(C+1)$ yields \eqref{induction}. 
Applying \eqref{induction} to $R=Q_{mL}$ and using that $\sum_{j\ge 0} 2^{- j\frac{d-2}{2}}<\infty$, we get 
\begin{equation}\label{conclusiondyadic}
 \fref(Q_{mL})\le \fref(Q_L)+ \Lambda_\eta(1+\fref(Q_L)) \frac{1}{L^{\frac{d-2}{2}}}.
\end{equation}

\noindent{\emph{Step \stepcounter{proof-step}\arabic{proof-step}. Conclusion in the case $p\ge 2$.}}
We finally claim that 
\begin{equation}\label{claimfQL}
 f(Q_L)\le \frac{C_\eta}{L^{\frac{1}{2}(d-p)}}.
\end{equation}
Arguing verbatim as in the proof of \eqref{toproveonestep} of {\it Step 5}, we see that it is enough to assume that we are in the event $A= \cur{ \min\cur{ |\cN^\eta|,  |\cM^\eta|} \ge \eta |Q| /2}$. 
Since in this case  $|\cG|/|Q| \gtrsim \eta$,  \cref{prop:density-helps} yields
 \[ \begin{split}  \EE\sqa{ \frac{1}{|Q_L|} \W_{Q_L}\bra{ \mu^\cB + \frac{|\cG|}{|Q|}, \frac{\mu^\cB(Q_L)}{|Q_L|}+ \frac{|\cG|}{|Q|}} I_A }  &
 \les \frac{1}{|Q_L| \eta^{\frac{p}{d}}} \EE\sqa{  \bra{ \mu^\cB(Q_L)}^{1+\frac{p}{d}} I_A }\\
 & \le \frac{1}{|Q_L| \eta^{\frac{p}{d}}} \EE\sqa{  \bra{ |\cU_{Q_L}|}^{1+\frac{p}{d}}  }  \\& \stackrel{\eqref{hypUV}}{\lesssim} \frac{L^{\frac{d+p}{2}}}{L^d \eta^{\frac{p}{d}}}   \les \frac{1}{\eta^{\frac{p}{d}}L^{\frac{d-p}{2}}}.
\end{split}\]
This proves \eqref{claimfQL}. Inserting this into \eqref{conclusiondyadic} finally gives (recall that $p>2$)
\[
 \fref(Q)\le \frac{C_\eta}{L^{\frac{1}{2}(d-p)}}.
\]
This concludes the proof of \eqref{eq:bound-bad-part} with $\alpha=(d-p)/2$ when $p\ge 2$.\\

\noindent{\emph{Step \stepcounter{proof-step}\arabic{proof-step}. The case  $p\le 2$.}}
If $p\le 2$, we argue as in the previous two proofs and use 
 \eqref{eq:jensen}
to obtain (recall that $Z=|\cG|+|\cB|\le |\cU|+|\cN^\eta|$)
\begin{multline*}  \EE\sqa{ \frac{1}{|Q|} \W_{Q} \bra{ \frac{ |\cG|}{|Q|} + \mu^{\cB}, \frac{ Z}{|Q|} } }    \le \EE\sqa{ \bra{ \frac{Z}{|Q|}}^{1-\frac{p}{2}} \bra{ \frac{1}{|Q|}  \mathsf{W}_{Q}^2 \bra{ \frac{ |\cG|}{|Q|} + \mu^{\cB}, \frac{ Z}{|Q|} } }^{\frac{p}{2}}  }\\ 
\le \bra{ \frac { \EE\sqa{Z}}{|Q|}}^{1-\frac{p}{2}} \EE\sqa{  \frac{1}{|Q|}  \mathsf{W}_{Q}^2 \bra{ \frac{ |\cG|}{|Q|} + \mu^{\cB}, \frac{ Z}{|Q|} }  }^{\frac{p}{2}}\\
  \les \bra{L^{-\frac{d}{2}}+\eta }^{1-\frac{p}{2}}\bra{ \frac{C(\eta)}{L^\alpha}}^{\frac{p}{2}} \les \bra{ \frac{C(\eta)}{L^\alpha}}^{\frac{p}{2}},
  \end{multline*}
where in the last step we used \eqref{eq:moment-qml} and \eqref{eq:bound-bad-part} with $p=2$. This concludes the proof of \eqref{eq:bound-bad-part} for any $p<d$. \qedhere
\end{proof}

\printbibliography

\end{document}